\documentclass[final,11pt,a4paper]{amsart} 

\usepackage{soul}
\usepackage[dvipsnames]{xcolor}
\usepackage{graphicx}
\usepackage[all]{xy}
\usepackage{placeins}
\usepackage{enumitem}
\usepackage{amssymb}
\usepackage{latexsym}
\usepackage{amsmath}
\usepackage{mathrsfs}
\usepackage{array,booktabs}
\usepackage{verbatim}
\usepackage{fullpage}
\usepackage[notref, notcite]{showkeys}  
\usepackage{color}
\usepackage{cancel}
\usepackage{amsfonts}
\usepackage{psfrag,epsfig}
\usepackage{mathtools}

\usepackage{hyperref}

\newcommand{\hyref}[2]{ \hyperref[#2]{#1~\ref*{#2}} }

\newtheorem{thmIntro}{Theorem}    
\usepackage{stackengine,graphicx,amsmath,fp,scalerel}
\newcount\crosswd
\newcount\termwd
\newcommand\rawcrossout[2]{\ensurestackMath{%
  \setbox0=\hbox{$#2$}%
  \crosswd=\wd0\relax%
  \setbox0=\hbox{$#1$}%
  \termwd=\wd0\relax%
  \FPdiv\myscale{\the\termwd}{\the\crosswd}%
  \stackengine{0pt}{#1}{\stretchrel*{\scalebox{\myscale}[1]{#2}}{#1}}{O}{c}{F}{T}{L}}}
\def\XX{\kern-3pt/}
\def\YY{\kern-.5pt}
\newcommand\crossout[1]{\rawcrossout{#1}{\YY/\YY}}

\newcommand{\Canakci}{\c{C}anak\c{c}\i}
\newcommand{\Ilke}{\.{I}lke }
\newcommand{\I}{\.{I}}

\numberwithin{figure}{section}
\numberwithin{table}{section}

\theoremstyle{plain}
\newtheorem{theorem}{Theorem}[section]
\newtheorem{prop}[theorem]{Proposition}

\newtheorem{lemma}[theorem]{Lemma}
\newtheorem{corollary}[theorem]{Corollary}
\newtheorem{proposition}[theorem]{Proposition}

\theoremstyle{definition}
\newtheorem{notation}[theorem]{Notation}
\newtheorem{remark}[theorem]{Remark}

\newtheorem{example}[theorem]{Example}
\newtheorem{definition}[theorem]{Definition}
\newtheorem{observation}[theorem]{Observation}
\newtheorem{convention}[theorem]{Convention}

\definecolor{dgreen}{rgb}{0,0.45,0}
\definecolor{beaublue}{rgb}{0.74, 0.83, 0.9}
\definecolor{darklavender}{rgb}{0.45, 0.31, 0.59}
\definecolor{darkorchid}{rgb}{0.6, 0.2, 0.8}
\definecolor{darkpastelpurple}{rgb}{0.59, 0.44, 0.84}
\definecolor{electricviolet}{rgb}{0.56, 0.0, 1.0}
\definecolor{capri}{rgb}{0.0, 0.75, 1.0}
\definecolor{indiagreen}{rgb}{0.07, 0.53, 0.03}
\definecolor{pastelbrown}{rgb}{0.51, 0.41, 0.33}
\definecolor{prune}{rgb}{0.44, 0.11, 0.11}
\definecolor{ao}{rgb}{0.0, 0.0, 1.0}
\definecolor{bistre}{rgb}{0.24, 0.17, 0.12}
\definecolor{britishracinggreen}{rgb}{0.0, 0.26, 0.15}
\definecolor{brown(traditional)}{rgb}{0.59, 0.29, 0.0}
\psfrag{cast}{\color{brown(traditional)}\scriptsize $\circledast$}
\psfrag{1_}{\color{brown(traditional)} \scriptsize $\tilde 1$}
\psfrag{2_}{\color{brown(traditional)} \scriptsize $\tilde 2$}
\psfrag{3_}{\color{brown(traditional)} \scriptsize $\tilde 3$}

\newcommand{\sX}{\mathsf{X}}

\DeclareMathAlphabet{\mathpzc}{OT1}{pzc}{m}{it}

\newcommand{\bC}{\mathbb{C}}

\newcommand{\bH}{\mathbb{H}}

\newcommand{\bN}{\mathbb{N}}

\newcommand{\bR}{\mathbb{R}}

\newcommand{\bT}{\mathbb{T}}

\newcommand{\bZ}{\mathbb{Z}}

\newcommand{\ga}{\calg_{\gamma}}
\DeclareMathOperator{\match}{\textup{ Match }}
\DeclareMathOperator{\cross}{\textup{cross }}
\DeclareMathOperator{\acc}{\mathrm{Acc }}
\newcommand{\cala}{\mathcal{A}}
\newcommand{\calb}{\mathcal{B}}

\newcommand{\cald}{\mathcal{D}}

\newcommand{\calf}{\mathcal{F}}
\newcommand{\calg}{\mathcal{G}}
\newcommand{\calh}{\mathcal{H}}

\newcommand{\cals}{\mathcal{S}}
\newcommand{\calt}{\mathcal{T}}

\newcommand{\calz}{\mathcal{Z}}

\renewcommand{\setminus}{\backslash}

\DeclareMathOperator{\bad}{\mathrm{Bad }}

\newcommand{\arr}{\ar@{-}[r]}

\newcommand{\mun}{\mu^{(n)}}
\newcommand{\muone}{\mu^{(1)}}
\newcommand{\muze}{\mu^{(0)}}
\newcommand{\muinf}{\mu^{\infty}}
\newcommand{\mubul}{\mu^{\bullet}}
\renewcommand{\phi}{\varphi}
\renewcommand{\epsilon}{\varepsilon}
\usepackage{tikz}
\usepackage{tikz-cd}
\usetikzlibrary{snakes}
\usetikzlibrary{shapes.geometric,positioning}
\usetikzlibrary{arrows,decorations.pathmorphing,decorations.pathreplacing}
\usetikzlibrary{matrix, calc, arrows}
\usetikzlibrary{decorations.pathmorphing,shapes}
\usetikzlibrary{decorations.pathreplacing}
\newcounter{sarrow}

\tikzset{join/.code=\tikzset{after node path={%
\ifx\tikzchainprevious\pgfutil@empty\else(\tikzchainprevious)%
edge[every join]#1(\tikzchaincurrent)\fi}}}

\tikzset{>=stealth',every on chain/.append style={join},
         every join/.style={->}}

\tikzset{vertex/.style={circle,fill=black,inner sep=1pt,outer sep=2pt},
         tinyvertex/.style={font=\scriptsize,minimum size=6pt},
         smallvertex/.style={inner sep=1pt, font=\small},
         >=stealth',
         leadsto/.style={-angle 90,decorate,decoration=snake,very thick},
         cut/.style={decorate,decoration=saw,very thick}}
\tikzset{
    partial ellipse/.style args={#1:#2:#3}{
        insert path={+ (#1:#3) arc (#1:#2:#3)}
    }
}
\begin{document}

\title[Infinite rank surface cluster algebras]{Infinite rank surface cluster algebras}
\thanks{This work was partially supported by EPSRC grant EP/N005457/1.}


\author{\Ilke  \Canakci}
\address{School of Mathematics, Statistics and Physics, Newcastle University, Newcastle Upon Tyne NE1 7RU, United Kingdom.}
\email{ilke.canakci@newcastle.ac.uk}

\author{Anna Felikson}
\address{Department of Mathematical Sciences, Durham University, Lower Mountjoy, Stockton Road, Durham, DH1 3LE, United Kingdom.}
\email{anna.felikson@durham.ac.uk}

\begin{abstract}
We generalise surface cluster algebras to the case of infinite surfaces where the surface contains finitely many accumulation points of boundary marked points. To connect different triangulations of an infinite surface, we consider infinite mutation sequences. 

We show transitivity of infinite mutation sequences on triangulations of an infinite surface and  examine different types of mutation sequences. Moreover, we use a hyperbolic structure on an infinite surface to extend the notion of surface cluster algebras to infinite rank by giving cluster variables as lambda lengths of arcs. Furthermore, we study the structural properties of infinite rank surface cluster algebras in combinatorial terms, namely we extend ``snake graph combinatorics" to give an expansion formula for cluster variables. We also show skein relations for infinite rank surface cluster algebras.
\end{abstract}

\maketitle

{\small
\setcounter{tocdepth}{2}
\tableofcontents
}

\section{Introduction}

We introduce cluster algebras associated with infinite bordered surfaces. Our work extends cluster algebras from (finite) marked surfaces \cite{FST,FT} and generalises surface cluster algebra combinatorics advanced in \cite{MSW,CS}.

Cluster algebras were introduced by a groundbreaking work of Fomin and Zelevinsky \cite{FZ1} and further developed in \cite{BFZ,FZ2,FZ4} with the original motivation to give an algebraic framework for the study of dual  canonical basis in Lie theory. Cluster algebras are combinatorially defined commutative algebras given by generators, {\it cluster variables}, and relations which are iteratively defined via a process called {\it mutation}.

An important class of cluster algebras, called {\it cluster algebras associated with surfaces} or {\it surface cluster algebras}, was introduced in \cite{FST} by using  combinatorics of (oriented) Riemann surfaces with marked points and consequently, building on an earlier work of \cite{FG1,FG2}, an intrinsic formulation of surface cluster algebras was given in \cite{FT} where lambda lengths of curves serve as cluster variables (see also \cite{P}).  Surface cluster algebras also constitute a significant class in terms of classification of cluster algebras since it was shown in \cite{FeSTu,FeSTu2,FeSTu3} that all but finitely many cluster algebras of {\it finite mutation type} are those associated with triangulated surfaces (or to orbifolds for the skew-symmetrizable case).  An initial data to construct a cluster algebra corresponds to a {\it triangulation} of the surface, namely a maximal collection of non-crossing {\it arcs}. 

Combinatorial and geometric aspects of surface cluster algebras have been studied remarkably widely. For instance, an expansion formula for cluster variables was given by Musiker, Schiffler and Williams \cite{MSW} extending the work of \cite{S2,ST,S3,MS} and bases for surface cluster algebras were constructed in \cite{MSW2,T,FeTu}. Further combinatorial aspects of surface cluster algebras were studied in \cite{CS,CS2,CS3,CLS}. Moreover, it was shown in \cite{M} that the quantum cluster algebra coincides with the quantum upper cluster algebra for the surface type under certain assumptions. In addition, combinatorial topology of triangulated surfaces and surface cluster algebra combinatorics have been influential in representation theory, e.g. in cluster categories associated with marked surfaces \cite{BZ,QZ,CaSc,ZZZ} and  in gentle algebras associated with marked surfaces \cite{ABCP,Labardini,Ladkani,CaSc}.

Infinite rank cluster algebras have appeared in a number of different contexts. In particular, in the study of cluster categories associated with the infinity-gon and the infinite double strip in \cite{IT, IT2, HJ, LP}. Moreover, triangulations of the infinity-gon were used in \cite{BHJ,HJ} to give full classification of $SL_2$-tilings and also in \cite{GG} to associate subalgebras of the coordinate ring $\bC[Gr(2,\pm\infty)]$ of an infinite rank Grassmannian $Gr(2,\pm\infty).$ Furthermore,    infinite rank cluster algebras were considered in \cite{HL,HL2} in the context of quantum affine algebras and also in  \cite{FK}  to show that the equations of generalised $Q$-systems can be given as mutations. 
See also~\cite{GG1,G} where infinite rank cluster algebras are given as colimits of finite rank cluster algebras.

However so far, to the best of our knowledge, infinite rank cluster algebras were considered with the focus only on {\it finite mutation sequences}. One exception to this restriction was given in an independent recent work of Baur and Gratz  \cite{BG} which considers infinite triangulations of unpunctured surfaces and classifies those in the same equivalence class. The article \cite{BG} concentrates on two examples, namely, on triangulations of the infinity-gon as well as of the {\it completed} infinity-gon where triangulations are completed with strictly asymptotic arcs (those connecting to limit points). The first is motivated to overcome finiteness constraints in the corresponding representation theory whereas the latter is motivated to introduce a combinatorial model to capture 
representation theory of the polynomial ring in one variable. However \cite{BG} does not study the associated cluster algebra structures. 

\bigskip

In this paper, we aim to consider infinite sequences of mutations and subsequently we introduce cluster algebras for {\it infinite surfaces} allowing infinite mutation sequences. More precisely, we have a three-fold goal for this paper: 

\begin{itemize}
\item {\it topological:} not to have any restrictions on the type of triangulations, i.e. being able to mutate between any two triangulations of a fixed surface;
\item {\it geometric:} to introduce cluster algebras from infinite surfaces by associating a hyperbolic structure to the surface and considering lambda lengths of arcs as cluster variables without having to deal with combinatorics of (infinite) quiver mutations;
\item {\it combinatorial:} to extend further  main properties of surface cluster algebras from finite case using ``snake graph calculus", in particular:
\begin{itemize}
\item[-] to give an expansion formula for cluster variables in terms of intersection pattern of  arcs with a fixed initial triangulation;
\item[-] to show skein relations, i.e. certain identities in the cluster algebra associated with generalised Ptolemy relations in the surface, for infinite surface cluster algebras.
\end{itemize} 
\end{itemize}

Our setting is the following. By an {\it infinite surface} we mean a connected oriented surface of a finite genus, with finitely many interior marked points (punctures), with finitely many boundary components, and with infinitely many {\it boundary marked points}, located in such a way that they have only finitely many {\it accumulation points} of boundary marked points.

We start with investigating the combinatorics of infinite mutations in Section~\ref{Sec:Combinatorics}. In order to be able to mutate between any two triangulations, we need to introduce infinite sequences of mutations, called {\it infinite mutations}, and moreover, {\it infinite sequences of infinite mutations}. When we refer to infinite mutation sequences, we only permit sequences ``converging" in a certain sense. Our main result establishes the transitivity of infinite sequences of infinite mutations on triangulations of a given infinite surface.

\begin{thmIntro}[Theorem~\ref{Thm:T-->T'}] \label{ThmA}
For any two triangulations $T$ and $T'$ of an infinite surface $\cals,$ there exists a mutation sequence $\mubul$ such that $\mubul(T)=T'$, where $\mubul$ is a ``finite mutation",  a ``finite sequence of infinite mutations" or an ``infinite sequence of infinite mutations" (see also Definition~\ref{Def:Mutation}). 
\end{thmIntro}

Furthermore, 
we show that the type of the mutation $\mubul$ required in Theorem~\ref{ThmA} depends on combinatorial properties of triangulations $T$ and $T'$ such as  intersection numbers of curves of $T'$ with the curves in $T$. We stress that there are some pairs of triangulations requiring infinite mutations and there are pairs of triangulations for which infinite sequences of infinite mutations are unavoidable. More precisely, we have the following result.

\begin{thmIntro}[Theorem~\ref{infinite sequences}, Theorem~\ref{bad are all limit}, Theorem~\ref{no mun from stronger}] Let $T$ and $T'$ be triangulations of an infinite surface. Then
\begin{enumerate}
\item  the mutation sequence $\mu^{\bullet}$ satisfying $\mu^{\bullet}(T)=T'$ in Theorem A can be chosen  ``finite" if and only if $T'$ crosses $T$ in finitely many places;
\item there exist triangulations $T$ and $T'$ such that there is no ``finite sequence of infinite mutations'' transforming $T$ to $T'$;
\item given some sufficient conditions (described in terms of intersections of $T$ and $T'$), there exists a ``finite sequence of infinite mutations'' transforming $T$ to $T'$;
\item given some sufficient conditions (described in certain combinatorial terms), there is no  ``finite sequence of infinite mutations'' transforming $T$ to $T'$.
\end{enumerate}
\end{thmIntro}

In Section~\ref{Sec:ClusAlg}, we follow \cite{FT} and \cite{P} and consider  triangulated surfaces with hyperbolic metrics such that every triangle of $T$ is an ideal triangle in this metric. We choose horocycles at every marked point (including accumulation points). Doing so, we require that if a sequence of marked points converge to an accumulation point $p_*$, then the corresponding horocycles converge to the horocycle at $p_*.$  Under these assumptions, the lambda lengths of the arcs satisfy some ``limit conditions". On the other hand, it turns out that as soon as these necessary conditions are satisfied, there exists a unique  hyperbolic structure and a unique choice of horocycles leading to the prescribed values of lambda lengths. In other words, lambda lengths of the arcs in an infinite surface can serve as coordinates on the decorated Teichm\"uller space. In order to achieve these results, we need to consider triangulations without infinite zig-zag (or leap-frog) pieces and we refer to such triangulations as {\it fan triangulations}. In this setting, the following version of Laurent phenomenon holds.

\begin{thmIntro}[Theorem~\ref{Laurent for punctures}] Lambda lengths of the arcs with respect to an initial fan triangulation are absolutely converging Laurent series in terms of the initial (infinite) set of variables corresponding to the fan triangulation.
 
\end{thmIntro}

Based on the result of this theorem, we can introduce infinite rank surface cluster algebras associated with {\it fan} triangulations. Subsequently, we show that the definition of cluster algebra associated with a fan triangulation is independent of the choice of an initial fan triangulation which allows us to speak about cluster algebras associated with a given infinite surface.

Section~\ref{Sec:SG} is devoted to the generalisation of combinatorial results known for (finite) surface cluster algebras. We extend the notion of snake graphs introduced in \cite{MSW,MS} to infinite snake graphs and give an expansion formula for cluster variables with respect to fan triangulations of the surface  by generalising the Musiker-Schiffler-Williams \cite{MSW} formula subject to some ``limit" conditions. This formula also
 manifestly gives cluster variables as Laurent series (with positive integer coefficients) in an (infinite) initial set of cluster variables associated with fan triangulations. 

\begin{thmIntro}[Theorem~\ref{Thm:MSW-main}]
Let $T$ be a fan triangulation of an infinite surface $\cals$ and  $\gamma$ be an arc. Let $x_{\gamma}$ be the  cluster variable associated with $\gamma$, let $\ga$ be the snake graph of $\gamma$ and $x_{\ga}$ be the Laurent series associated with  $\ga$. Then $x_{\gamma}=x_{\ga}.$

\end{thmIntro}

We also generalise the skein relations of \cite{MW,CS} by extending the technique of \cite{CS} for cluster algebras from {\it unpunctured} infinite surfaces.

\begin{thmIntro}[Theorem~\ref{Thm:Skein}]
Let $\cals$ be an infinite \emph{unpunctured} surface, $T$ be a fan triangulation of $\cals,$ and $\gamma_1,\gamma_2$ be crossing generalised arcs. Let $\gamma_3,\dots,\gamma_6$ be the (generalised) arcs obtained by smoothing a crossing of $\gamma_1$ and $\gamma_2$ and $x_1,\dots,x_6$ be the corresponding elements in the associated infinite surface cluster algebra $\cala(\cals)$ in terms of the initial cluster corresponding to $T$. Then the identity
\begin{align*}
x_1x_2=x_3x_4+x_5x_6
\end{align*} 
holds in $\cala(\cals)$.

\end{thmIntro}

Finally, in Section~\ref{Sec:Consq} we collect properties of infinite surface cluster algebras.

\begin{thmIntro}[Theorem~\ref{Thm:Consq}]
Let $\cals$ be an infinite surface and $\cala(\cals)$ be the corresponding cluster algebra. Then
\begin{itemize}
\item seeds are uniquely determined by their clusters; 
\item for any two seeds containing a  cluster variable $x_\gamma$ there exists a mutation sequence $\mu^\bullet$ (where $\mu^\bullet$ is a finite mutation, a finite sequence of infinite mutations or an infinite sequence of infinite mutations) such  that $x_\gamma$ belongs to every cluster obtained in the course of mutation $\mu^\bullet$;
\item there is a cluster containing a collection of  cluster variables $\{x_i \mid i\in I\}$, where $I$ is a finite or infinite index set, if and only if for every choice of $i,j\in I$ there exists a cluster containing $x_i$ and $x_j$.
\end{itemize}
Moreover, if $T$ is a fan triangulation of $\cals$ then
\begin{itemize}
\item the ``Laurent phenomenon'' holds, i.e. any cluster variable in $\cala(\cals)$ is an absolutely converging Laurent series in cluster variables corresponding to the arcs (and boundary arcs) of $T$;
\item ``Positivity'' holds, i.e. the coefficients in the Laurent series expansion of a cluster variable  in $\cala(\cals)$ are positive. 
\end{itemize}
If in addition the surface ${\cals}$ is unpunctured then 
\begin{itemize}
\item the exponent of initial variable $x_i$ in the denominator of a cluster variable $x_{\gamma}$ corresponding to an arc $\gamma$ is equal to the intersection number  $|\gamma\cap \tau_i|$ of $\gamma$ with the arc $\tau_i$ of $T$.
\end{itemize}

\end{thmIntro}

We would expect that our construction can be used as a combinatorial model to generalise  representation theory of finite dimensional algebras (e.g. surface cluster categories, surface gentle algebras, tilting theory, etc.) to the infinite case beyond  $A_{\infty}$ types.    

\bigskip

The paper is organised as follows. Section~\ref{Sec:Combinatorics} is devoted to combinatorics of infinite triangulations, in particular transitivity of infinite mutations and a discussion of different types of mutation sequences. Section~\ref{Sec:ClusAlg} deals with hyperbolic geometry and introduces surface cluster algebras associated with infinite surfaces. Section~\ref{Sec:SG} concerns with infinite snake graphs and introduces an expansion formula for cluster variables as well as  skein relations in this context. Finally in Section~\ref{Sec:Consq}, we establish some properties of infinite rank surface cluster algebras.
\bigskip

{\it Acknowledgements:}   We would  like to thank Peter J\o rgensen and Robert Marsh for stimulating discussions regarding potential representation theory associated with our combinatorial model.   We would also like to thank Karin Baur and Sira Gratz for sharing their preprint \cite{BG} shortly before publishing it on the arXiv. Last but not least, we are very grateful to the anonymous referee for significant comments and corrections they proposed.

\section{Triangulations and mutations of infinite surfaces}
\label{Sec:Combinatorics}

In this section, we introduce our setting for infinite surfaces, infinite triangulations, and infinite sequences of mutations. The main result of this section presents transitivity of infinite mutation sequences.
\subsection{Infinite triangulations} We first fix our setting for infinite surfaces.
\begin{definition}[Infinite surface]
Throughout the paper by an {\it infinite surface} $\cals$ we mean a connected oriented surface  
\begin{itemize}
\item[-] of a finite genus,
\item[-] with finitely many interior marked points (punctures),
\item[-] with finitely many boundary components, 
\item[-] with infinitely many {\it boundary marked points}, located in such a way that they have only finitely many {\it accumulation points} of boundary marked points.  
 \end{itemize}
 Accumulation points themselves are also considered as boundary marked points. 

\end{definition}

The following definition coincides with the one for finite surfaces.

\begin{definition}[Arc, compatible arcs, triangulation, boundary arc]
An {\it arc} on an infinite surface $\cals$ is a non-self-intersecting curve with both endpoints at the marked points of $\cals$, considered up to isotopy.
As usual, we assume that an arc $\gamma$ is disjoint from the boundary of $\cals$ except for the endpoints, and that it does not cut an unpunctured monogon from $\cals$. The two endpoints of $\gamma$ may coincide.
Two arcs are called {\it compatible} if they do not intersect (i.e.  if there are representatives in the corresponding isotopy classes which do not intersect). 
A {\it triangulation} of $\cals$ is a maximal (by inclusion) collection of mutually compatible arcs, i.e a maximal set of arcs on $\cals$ one can draw without crossings.
A {\it boundary arc} is a non-self-intersecting curve $\alpha\in \partial \cals$ such that the endpoints of $\alpha$ are  boundary marked points
and no other point of $\alpha$ is a boundary marked point.

\end{definition}

Obviously, any triangulation of an infinite surface contains infinitely many arcs and infinitely many triangles.

\begin{remark}
In case of an infinite punctured surface we will also consider {\it tagged triangulations} as in~\cite{FST}. We will not  
present this definition here as we will only use this notion occasionally and  in a very straightforward way.

\end{remark}

\begin{notation}
Denote by $\mathbb T_\cals$ the set of all triangulations of an infinite surface $\cals$.
When the surface is clear from the context, we will simply abbreviate this notation by $\mathbb T$.

\end{notation}

\begin{definition}[Convergence of arcs, limit of arcs, limit arc] \label{Def:LimitArc}
{\color{white} a}
\begin{itemize}
\item
We say that a sequence of arcs $\{\gamma_i\}_{i\in \bN}$,   {\it converges} to an arc (or to a boundary arc) $\gamma_*$ in an infinite surface $\cals$ if
\begin{itemize}
\item[(1)] the endpoints of  $\{\gamma_i\}$ converge to the endpoints of  $\gamma_*$ and
\item[(2)] for $i$ large, the arc $\{\gamma_i\}$ is ``almost isotopic'' to $\gamma_*$ in the following sense: 
\smallskip
\item[-]
Let  $\{p_i\}\to p_*$ and $\{q_i\}\to q_*$ be the endpoints of $\{\gamma_i\}$, $i\in \bN,$ converging to the endpoints of $\gamma_*$.
Let $U_{p_*}\subset \partial \cals$ and $U_{q_*}\subset \partial \cals$ be  neighbourhoods of $p_*$ and $q_*$ such that neither $U_{p_*}$ nor $U_{q_*}$ contains a connected component of $\partial \cals$. 
Let $\tilde \gamma_i$ be the arc obtained from $\gamma_i$ by shifting the endpoints to $p_*$ and $q_*$ without leaving  the sets  $U_{p_*}$ and $U_{q_*},$ respectively.
We say that $\{\gamma_i\}$ {\it converge} to $\gamma_*$ if  $\tilde \gamma_i$ is isotopic to $\gamma_*$ for all $i>N$ for some $N\in \mathbb N$.
\end{itemize}
\smallskip
\noindent
If $\{\gamma_i\}$ converge to $\gamma_*$ we also say that $\gamma_*$ is a {\it limit of the arcs}  
 $\gamma_{i}$ and write $\gamma_i\to \gamma_*$ as $i\to \infty$.

\item Given a triangulation $T$ of $\cals$, if   $\{\gamma_i \}\to \gamma_*$ as $i\to \infty$  and $\gamma_{i},\gamma_*\in T$ for all $i\in \bN$,  we say that  $\gamma_*$ is a {\it limit arc} of $T$. Graphically, we will show limit arcs (as well as boundary limit arcs) by dashed lines, see Fig.~\ref{limit arcs}.
\item Furthermore, if $p_*=q_*$ and the curves $\tilde \gamma_i$ (obtained from $\gamma_i$ by shifting the endpoints to $p_*=q_*$ inside $U_{p_*}$)
are contractible to the accumulation point $p_*$, we say that $\gamma_i$ {\it converges to the accumulation point} $p_*$ and write $\gamma_i\to p_*$.
\end{itemize}
\end{definition} 

\begin{figure}[!h]
\begin{center}
\psfrag{a}{\small (a)}
\psfrag{b}{\small (b)}
\epsfig{file=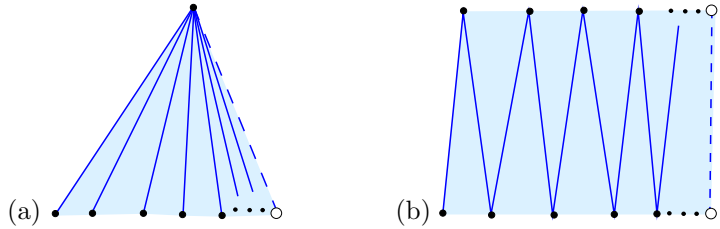,width=0.59\linewidth}
\caption{A limit arc in a triangulation may arise in two different ways: in an infinite fan (a) and as a limit of an infinite zig-zag (b). } 
\label{limit arcs}
\end{center}
\end{figure}
 
\begin{prop}
Let $T$ be a triangulation of $\cals$ containing a sequence of  arcs  $\gamma_i$, $i\in \bN$. 
If $\{\gamma_i\}\to \gamma_*$ as $i\to \infty$  then  $\gamma_*\in T$.

\end{prop}

\begin{proof} 
Suppose that  $\gamma_*\notin T$. Then $T$ contains an arc $\alpha$ which intersects  $\gamma_*,$ otherwise $T$ would not be a maximal set of compatible arcs. However, it is easy to see that if $\alpha$ intersects $\gamma_*$, then it also intersects  $\gamma_i$ for large enough $i$, which contradicts the assumption that $\alpha$ and $\gamma_i$ lie in the same triangulation $T$.

\end{proof}

\begin{remark}
Notice that an arc with an endpoint at an accumulation point or an arc connecting two accumulation points is {\it not} necessarily a limit arc, see Fig.~\ref{surf}
for an example. 

\end{remark}

\begin{figure}[!h]
\begin{center}
\epsfig{file=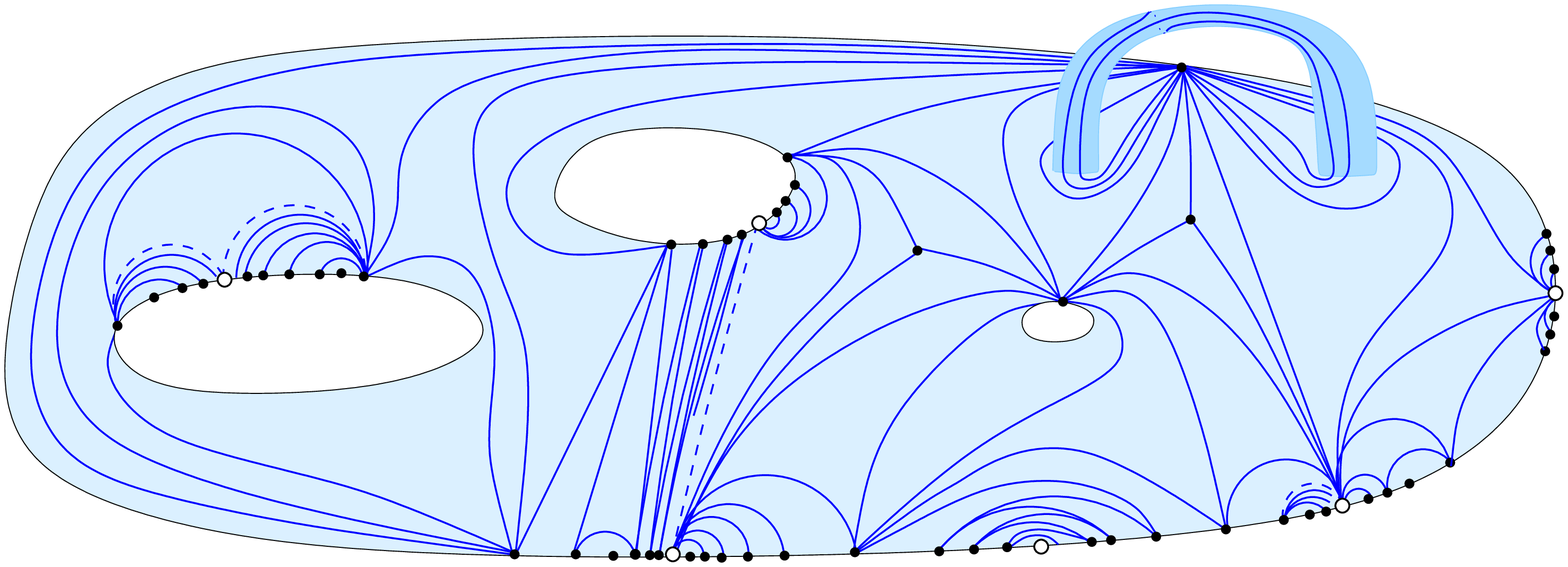,width=0.59\linewidth}
\caption{Triangulation of an infinite surface: marked points are denoted by dots, accumulation points are indicated by white dots and limit arcs are shown by dashed lines.} 
\label{surf}
\end{center}
\end{figure}

\begin{proposition}\label{Prop:FinManyAcc} Let $\cals$ be an infinite surface. Every triangulation of $\cals$ contains only finitely many limit arcs.
\end{proposition}

\begin{proof}
Every accumulation point is an endpoint of at most two limit arcs since each of  the right and left limits at this point gives rise to at most one  limit arc. The result follows since $\cals$ has finitely many accumulation points.
\end{proof}

\begin{proposition}
\label{Prop:ArcsFinInter} 
Two arcs on an infinite surface $\cals$ have only finitely many intersections.
\end{proposition}

\begin{proof}
Fix two arcs $\gamma,\gamma'$  on $\cals.$ Each accumulation point of $\cals$ has a right and a left neighbourhood  containing no endpoints of $\gamma$ and $\gamma'.$ We remove these neighbourhoods from $\cals$  as shown in Fig.~\ref{two arcs in finite S}. The obtained surface is finite, and hence the result follows.
\end{proof}

\begin{figure}[!h]
\begin{center}
\epsfig{file=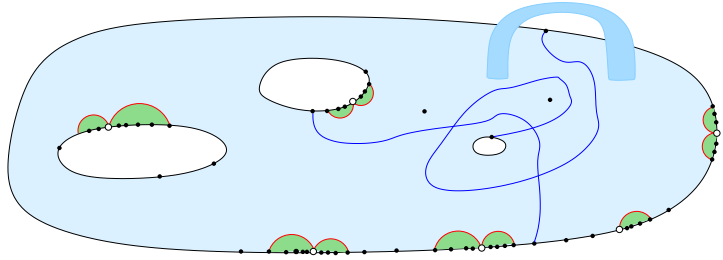,width=0.59\linewidth}
\caption{Removing left and right neighbourhoods of accumulation points gives rise to  a finite surface. } 
\label{two arcs in finite S}
\end{center}
\end{figure}

Propositions~\ref{Prop:FinManyAcc}~and~\ref{Prop:ArcsFinInter} imply the following corollary.

\begin{corollary} \label{Cor:FiniteLimArc}
For any arc $\gamma\in\cals$ and a triangulation $T$ of $\cals,$ there are only finitely many crossings of $\gamma$ with limit arcs of $T.$

\end{corollary}
\subsection{Infinite mutation sequences}

As in~\cite{FST}, we are going to use flips of arcs to change a triangulation (see Fig.~\ref{flip}). Our aim is to be able to transform every triangulation of an infinite surface $\cals$ to any other triangulation of $\cals$. Observe that a limit arc cannot be flipped, and moreover, one cannot make it flippable in finitely many steps.  
To fix this, we will need to introduce {\it infinite mutations} (see Definition~\ref{Def:Mutation}). However, even that would not be enough for transitivity of the action of our moves on all triangulations of $\cals$, and therefore we will also need to introduce {\it infinite sequences of infinite mutations}.

\begin{figure}[!h]
\begin{center}
\psfrag{g}{\small $\gamma$}
\psfrag{g'}{\small $\gamma'$}
\psfrag{m}{$\mu$}
\epsfig{file=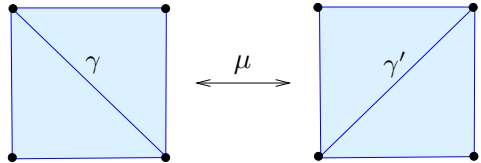,width=0.39\linewidth}
\caption{Flip of an arc $\gamma\in T$: $\mu(\gamma)=\gamma'$, $\mu(\alpha)=\alpha \quad \forall \alpha\in T$, $\alpha\ne \gamma$.} 
\label{flip}
\end{center}
\end{figure}

\begin{definition}[Infinite mutation, infinite sequence of infinite mutations] 
\label{Def:Mutation}
Consider a triangulation $T=\{\gamma_i\mid i\in\bN\}$  of an infinite surface $\cals$.
\begin{enumerate}
\item An \emph{elementary mutation} $\mu$ is a flip of an arc in $T$, see Fig.~\ref{flip}.
\item A \emph{finite mutation} $\mu_{n}\circ\ldots\circ\mu_{1}$ is a composition of finitely many flips for some $n\in\bN$. We will use the notation $\mu^{(0)}$ to refer to  a finite mutation.
\item An \emph{infinite mutation} is the following two step procedure:
\begin{itemize}
\item[-]  apply an \emph{admissible} infinite sequence of elementary mutations $\ldots  \circ \mu_{2} \circ \mu_{1}$,  where a sequence is admissible if for every $\gamma\in T$ there exists $n=n(\gamma)\in \bN$ such that $\forall i>n,$ we have $\mu_i\circ\ldots\circ \mu_2\circ \mu_1(\gamma)=\mu_n\circ\ldots \circ\mu_2\circ \mu_1(\gamma)$;
\item[-] complete the resulting collection of arcs by  all limit arcs (if there are any).
\end{itemize}
We will use the notation $\mu^{(1)}$ to specify an infinite mutation.
\item A \emph{finite sequence of infinite mutations} $\mu^{(1)}_n\circ\ldots\circ\mu^{(1)}_1$ is a composition of finitely many admissible infinite mutations $\mu^{(1)}_i$ for $i=1,\ldots,n$ for some $n \in \bN$. We will use the notation $\mun$ to specify a sequence of $n$ infinite mutations.
\item An \emph{infinite sequence of infinite mutations} is the following two step procedure:
\begin{itemize}
\item[-]  apply an \emph{admissible} infinite composition of infinite mutations $\ldots\circ\muone_{2}\circ\muone_{1}$,  where a composition is admissible  if the {\it orbit} of every individual arc converges, i.e.
the sequence of arcs $\{\muone_{i+j}\circ\ldots\circ\muone_{i+2}\circ\muone_{i+1}(\gamma) \}$, for $j \geq 1$, converges for every arc \\
$\gamma\in \bigcup\limits_{i=0}^\infty\left[\muone_i\circ\ldots\circ\muone_{2}\circ\muone_{1}(T)\right]$, for $i \geq 0$;
\item[-]  \emph{complete} the resulting collection of arcs by  all limit arcs (if there are any).
\end{itemize}
\noindent
We will use the notation $\muinf$ to specify an infinite sequence of infinite mutations.
\end{enumerate}

\end{definition}

If an arc $\gamma'$ is obtained from $\gamma$ by a mutation sequence $\mu$ (defined above) we will say that $\gamma'$ lies in the {\it orbit}  of $\gamma$ for the mutation sequence $\mu$.

\begin{remark}
In Definition~\ref{Def:Mutation}, when completing collections of arcs by limit arcs we only need to add those that are not already in the collection before the completion. 
\end{remark}

\begin{figure}[!h]
\begin{center}
\epsfig{file=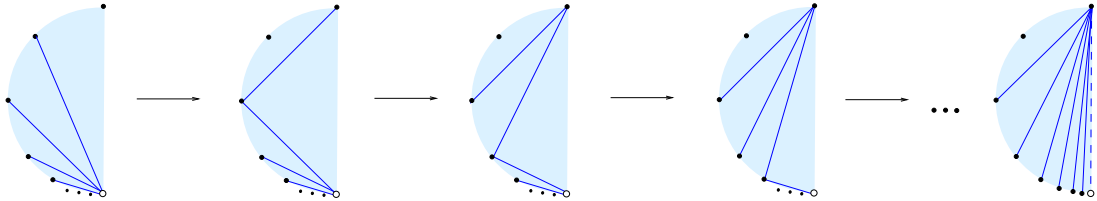,width=0.9\linewidth}
\caption{Switching between two triangulations (more precisely, in terms of Definition~\ref{Def:InfFanZigZag} below,
it is transforming an elementary outgoing fan to an  elementary incoming fan).}
\label{Fig:switch(outgoing-->incoming)}
\end{center}
\end{figure}

\begin{example} 
\label{Example: mutations}
\begin{itemize}
\item[(a)] In Fig.~\ref{Fig:switch(outgoing-->incoming)},~\ref{Fig:shiftsource}~and~\ref{Fig:switch(zigzag-->fan}, we see examples of finite sequences of infinite mutations. 
\item[(b)] The infinite mutation in Fig.~\ref{Fig:shiftsource} allows us to shift the source of a fan in one direction, namely in the direction of the accumulation point. However, 
to shift it back we would need to apply an  infinite sequence of infinite mutations (this will follow from  Theorem~\ref{no mun from stronger}).
\item[(c)] Applying the shift as in  Fig.~\ref{Fig:shiftsource} infinitely many times we get an infinite sequence of infinite mutations which may serve as an ``inverse'' to the infinite mutation  shown in Fig.~\ref{Fig:switch(outgoing-->incoming)}
(again,   Theorem~\ref{no mun from stronger} shows that there is no way to  mutate back with a finite sequence of infinite mutations).
Notice that in this example the orbit of every arc of the initial triangulation converges to the accumulation point.
\end{itemize}
\end{example}

\begin{figure}[!h]
\begin{center}
\epsfig{file=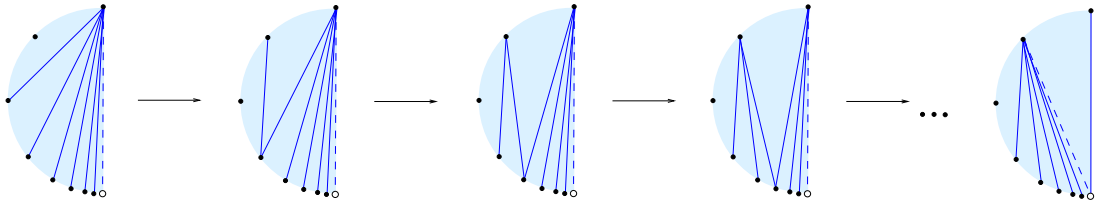,width=0.9\linewidth}
\caption{Shifting the source of a fan. Notice that the limit arc in the first triangulation becomes a regular arc at the end of infinite mutation and a new limit arc is created.}
\label{Fig:shiftsource}
\end{center}
\end{figure}

\begin{remark}
\label{compareBG}
\begin{itemize}
\item Our notion of infinite mutation is similar to the notion of {\it mutation along admissible sequences} in~\cite{BG}. More precisely, 
an infinite mutation is a  {\it mutation along admissible sequences} completed by all limit arcs. However, this does not coincide precisely with the notion of {\it completed mutation} in~\cite{BG}, as we only add limit arcs while \cite{BG} sometimes adds more arcs.

\item Our notion of infinite sequence of infinite mutations is a bit more
restrictive than the notion of {\it transfinite mutation} in~\cite{BG}: we require every individual orbit to converge, while \cite{BG} considers a collection of orbits which stabilise allowing the others to diverge.
\end{itemize}
\end{remark}

\begin{figure}[!h]
\begin{center}
\psfrag{1}{\tiny 1}
\psfrag{2}{\tiny 2}
\psfrag{3}{\tiny 3}
\psfrag{4}{\tiny 4}
\psfrag{5}{\tiny 5}
\psfrag{6}{\tiny 6}
\psfrag{7}{\tiny 7}
\psfrag{8}{\tiny 8}
\epsfig{file=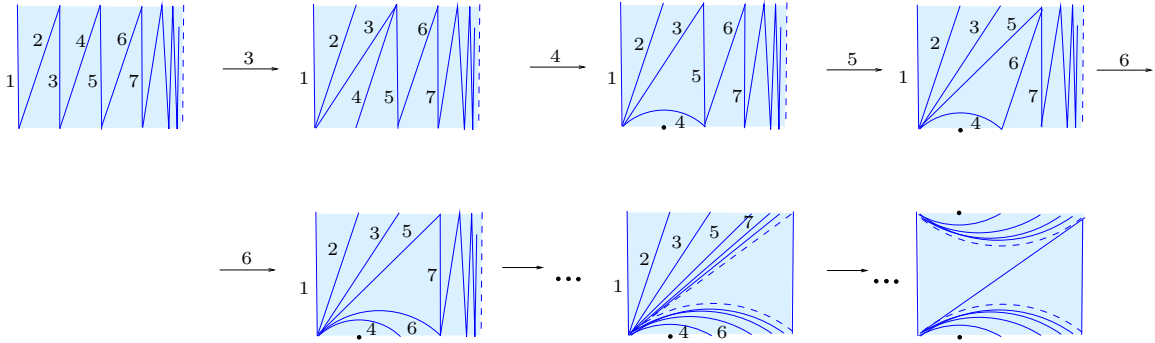,width=0.95\linewidth}
\caption{Transforming an infinite zig-zag into a  union of two  fans via composition of two infinite mutations (where the second infinite mutation coincides with the one in Fig.~\ref{Fig:shiftsource}).}
\label{Fig:switch(zigzag-->fan}
\end{center}
\end{figure}

\begin{remark} 
\begin{itemize}
\item[(a)] Note that \emph{not} every infinite mutation  (even an admissible one!) gives rise to a complete triangulation, see Fig.~\ref{non-complete-trian} for an example of admissible infinite mutation not leading to a triangulation.
\item[(b)] Notice also that if a (finite or infinite) sequence of infinite  mutations is applied to a compatible collection of arcs not forming a triangulation then the result will not be a triangulation either.
\end{itemize}
\end{remark}

\begin{figure}[!h]
\begin{center}
\psfrag{1}{\tiny 1}
\psfrag{2}{\tiny 2}
\psfrag{3}{\tiny 3}
\psfrag{4}{\tiny 4}
\psfrag{5}{\tiny 5}
\psfrag{6}{\tiny 6}
\psfrag{7}{\tiny 7}
\psfrag{8}{\tiny 8}
\psfrag{64654}{\tiny $6\!\circ\! 4\!\circ\! 6\!\circ\! 5\!\circ\! 4$}
\psfrag{86876}{\tiny $8\!\circ\! 6\!\circ\! 8\!\circ\! 7\!\circ\! 6$}
\epsfig{file=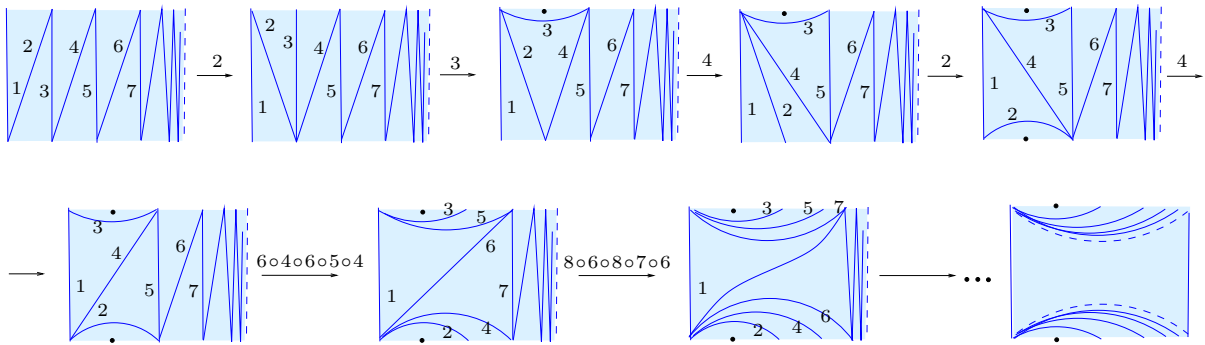,width=0.99\linewidth}
\caption{Mutating an infinite zig-zag to an incomplete triangulation: there is no diagonal in the resulting triangulation as every individual arc moves to a lower fan (for even indices) or to an upper fan (for odd indices) and we cannot add a diagonal at the end since it does not appear as a limit arc.}
\label{non-complete-trian}
\end{center}
\end{figure}

\subsection{Elementary domains}

For the proofs of main results of this section, we will need to cut the surface into smaller pieces. Each of the pieces will be a disc triangulated in one of the five ways classified according to the local behaviour around an accumulation point.

In what follows, by disc we always mean an unpunctured disc with finitely or infinitely many boundary marked points.

\begin{definition}[Elementary domains]
\label{Def:InfFanZigZag}
Triangulations of discs combinatorially equivalent to the ones shown in Fig.~\ref{Fig:local} will be called {\it elementary domains}. Moreover, the elementary domains shown in Fig.~\ref{Fig:local}(a)-(e) will be called
\begin{itemize}
\item[(a)] {\it finite fan} (may contain any finite number of triangles);
\item[(b)] {\it infinite incoming fan};
\item[(c)] {\it infinite outgoing fan};
\item[(d)] {\it infinite zig-zag} around an accumulation point;
\item[(e)] {\it infinite zig-zag} converging to a limit arc.
\end{itemize} 

\end{definition}

\begin{figure}[!h]
\begin{center}
\psfrag{a}{\scriptsize (a)}
\psfrag{b}{\scriptsize (b)}
\psfrag{c}{\scriptsize (c)}
\psfrag{d}{\scriptsize (d)}
\psfrag{e}{\scriptsize (e)}
\epsfig{file=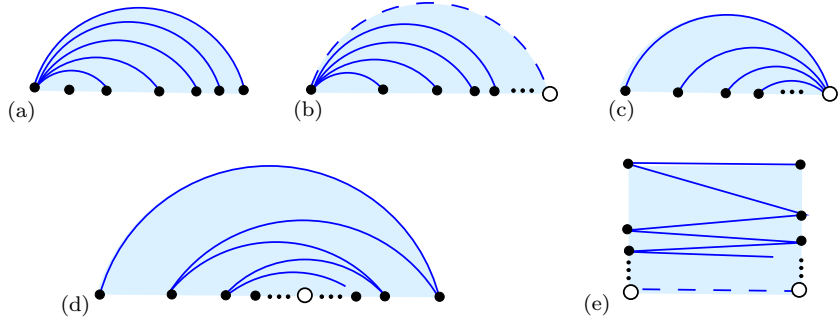,width=0.69\linewidth}
\caption{Five types
 of elementary domains: (a) finite fan, (b) incoming fan, (c) outgoing fan, (d) zig-zag around an accumulation point, (e) zig-zag converging to a limit arc. Notice that the incoming and outgoing fans may lie on the right or on the left of the accumulation point.}
\label{Fig:local}
\end{center}
\end{figure}

\begin{remark} In the literature, infinite fan triangulations are also called {\it fountains}.
Zig-zag triangulations around an accumulation point are also called {\it leap-frogs}. 
\end{remark}

\begin{remark}
Any domain in Definition~\ref{Def:InfFanZigZag}  has at most one left accumulation point and at most one right  accumulation point, and these points may coincide.

If  the two limit points of an  infinite zig-zag of type (e) coincide in such a way that the limit arc is contractible to a point and therefore vanishes, we  obtain exactly an infinite zig-zag of type (d). Therefore, we will understand a zig-zag around an accumulation point as a zig-zag converging to a (vanishing) limit arc.

Similarly, when a limit arc of an incoming fan is contracted and vanishes, we obtain an outgoing fan.
Hence, in many cases it makes sense to speak about fans in general, without specifying the type.
\end{remark}

There are triangulations of a disc with at most two accumulation points which are not exactly of the form given as elementary domains, but  have very similar underlying behaviour, see Fig.~\ref{Fig:local_AlmostElementary}. This leads us to the following definitions. 

\begin{definition}[Total number of accumulation points $\acc(\cals)$ of $\cals$] \label{Defn:Acc(S)}
By the {\it total number of accumulation points in $\cals$} we mean the total number of left and right accumulation points on $\cals$ (where two-sided accumulation points are counted twice). Denote this number by $\acc(\cals)$. 
\end{definition}

\begin{definition}[Almost elementary domains] 
\label{Def:UnderlyingTriangulation} 
Let $\cald$ be a disc with at least one accumulation point, let $T$ be a triangulation of $\cald$. Suppose that it is possible to remove an infinite sequence of marked points from $\cald$ together with all arcs incident to them so that 
\begin{itemize}
\item[(a)] the total number of accumulation points of the obtained surface $\cals'$ is the same as the one of $\cals$, i.e. $\acc(\cals')=\acc(\cals)$;
\item[(b)]  the obtained triangulation is an infinite elementary domain.
\end{itemize}  
Then we say that $\cald$ together with $T$ is an {\it almost elementary domain $(\cald,T)$.} The type of an 
 almost elementary domain is determined according to the type of the corresponding   elementary domain (i.e. {\it almost elementary incoming/outgoing fan},  {\it almost elementary zig-zag}).
\end{definition}

\begin{figure}[!h]
\begin{center}
\epsfig{file=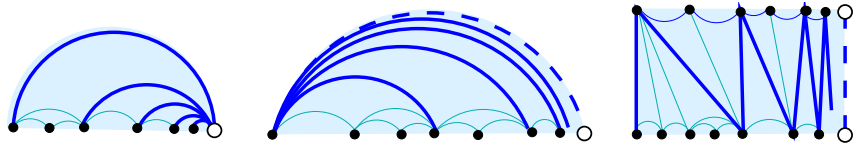,width=0.7\linewidth}
\caption{Examples of almost elementary domains.} 
\label{Fig:local_AlmostElementary}
\end{center}
\end{figure}

\begin{definition}[Source of fan, base of fan, bases of zig-zag]
\label{base}
{\color{white} A}
\begin{itemize}
\item By the {\it source} of a fan $\calf$ we mean the marked point incident to infinitely many arcs of $\calf$.
\item Let $\gamma$ be the (possibly vanishing) limit arc of a fan $\calf$ and let $\alpha\neq\gamma$ be another boundary arc of $\calf$ incident to the source of $\calf$.
Then  by the {\it base} of $\calf$ we mean $\partial \calf\setminus \{\gamma\cup \alpha\}$ 
(see the horizontal part of the boundary of the fan in Fig.~\ref{elem_bases}).
\item Similarly, for an almost elementary zig-zag $\mathcal Z$, let $\gamma$ be the (possibly vanishing) limit arc and let $\alpha$ be the first arc of the zig-zag. Then  by the {\it bases} of $\calz$ we mean the connected components of $\partial \calz\setminus \{\gamma\cup \alpha\}$ 
(see the two horizontal parts of the boundary of the zig-zag in Fig.~\ref{elem_bases}).
\end{itemize}

\end{definition}

\begin{figure}[!h]
\begin{center}
\psfrag{a}{\small $\alpha$}
\psfrag{g}{\small $\gamma$}
\epsfig{file=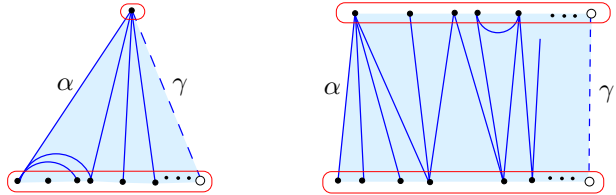,width=0.5\linewidth}
\caption{Source of a fan, base of a fan and  bases of a zig-zag.} 
\label{elem_bases}
\end{center}
\end{figure}

Our next aim is to prove Proposition~\ref{Lem:FinArcs-->Discs}  which shows that almost elementary domains may serve as building blocks of triangulated infinite surfaces and that every triangulation requires finitely many of them. For this we first need a technical lemma giving a characterisation of almost elementary domains.

\begin{lemma}
\label{l-almost elementary}
Let $D$ be a disc with at least one accumulation point of marked points on the boundary.
Let $T$ be a triangulation of $D$ such that it contains
\begin{itemize}
\item[(1)] no internal limit arc, and
\item[(2)]
no arc $\gamma$ cutting $D$ into two discs $D_1$ and $D_2$  such that 
$\acc(D_i)<\acc(D)$ for $i=1,2$. 
\end{itemize}
Then $(D,T)$ is an almost elementary fan or an almost elementary zig-zag.
\end{lemma}

\begin{proof}
We will consider two cases: either $D$ has at least two accumulation points or $D$ has a unique accumulation point.

\medskip
\noindent
\underline{\bf Case  1:}
Suppose that $D$ contains at least 2 accumulation points. We will identify the disc $D$ with the upper half-plane $\{z\in \bC \mid \textup{Im }z> 0 \}\cup \{\infty\}$, and we will assume that
 $0$ and $\infty$ are successive accumulation points, i.e. there is no accumulation point at any $r>0$, $r\in\mathbb R$. We will also assume (by symmetry) that $\infty$ is an ascending limit, i.e. that there is an infinite sequence of marked points in $\{ r>0, r \in \mathbb{R}\}$.  Furthermore, there are two possibilities: either there is a sequence of marked points decreasing to 0 or not. We will consider these cases.

\smallskip
\noindent
{\bf Case 1.a:} Suppose $0$ is not an accumulation point from the right. The marked points on the positive ray can be labelled by $q_i$, $i\in \bN$, so that $q_i<q_j$ if $i<j$, $q_i\to \infty$ as $i\to \infty$ (see Fig.~\ref{almost-pf}, left). Assumption~(2) implies that there is no arc of $T$ connecting $q_i$ to any point $s<0$; and moreover, there is no arc from $q_i$ to $\infty$. Similarly, there is no arc from $0$ to any point $s\le 0$ (since $0$ is an accumulation point from the left). 

Consider the point $q_1$. Denote $u_1=q_1$. The triangle $t_1$ in $T$ containing the boundary arc $0q_1$ has a third vertex. From the discussion above we derive that it is one of  $q_i$ ($i>1$). Denote that vertex $u_2=q_i$, and let $\gamma_{0,2}$ be the arc connecting $0$ with $u_2$. There is another triangle $t_2\ne t_1$ containing the arc $\gamma_{0,2}$. Let $u_3=q_j,$ where $j>i,$ be the third vertex of $t_2$. Denote by $\gamma_{0,3}$ the arc connecting   $0$ with $u_3$. We may continue in this way constructing an infinite growing sequence of the points $u_i$, an infinite sequence of triangles $t_i$ all having $0$ as a vertex, and an infinite sequence of arcs $\gamma_{0,i}\in T$. However, this implies that the arc connecting $0$ with $\infty$ is a limit arc, which contradicts assumption~(1).     

\begin{figure}[!h]
\begin{center}
\psfrag{0}{$0$}
\psfrag{inf}{$\infty$}
\psfrag{=}{\scriptsize $=$}
\psfrag{1}{\scriptsize $q_1$}
\psfrag{2}{\scriptsize $q_2$}
\psfrag{3}{\scriptsize $q_3$}
\psfrag{4}{\scriptsize $q_4$}
\psfrag{5}{\scriptsize $q_5$}
\psfrag{6}{\scriptsize $q_6$}
\psfrag{u1}{\scriptsize $u_1$}
\psfrag{u2}{\scriptsize $u_2$}
\psfrag{u3}{\scriptsize $u_3$}
\psfrag{t0}{\scriptsize $t_0$}
\psfrag{q0}{\scriptsize $q_0$}
\psfrag{t1}{\scriptsize $t_1$}
\psfrag{t2}{\scriptsize $t_2$}
\psfrag{t3}{\scriptsize $t_3$}
\psfrag{a0}{\scriptsize $a_0$}
\psfrag{b0}{\scriptsize $b_0$}
\psfrag{b1}{\scriptsize $b_1$}
\epsfig{file=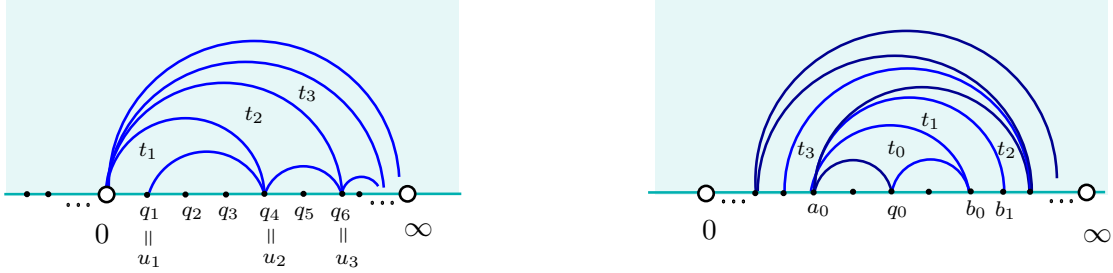,width=0.98\linewidth}
\caption{To the proof of Lemma~\ref{l-almost elementary}: Case~1.a (left) and Case~1.b (right).} 
\label{almost-pf}
\end{center}
\end{figure}

\smallskip
\noindent
{\bf Case 1.b:} Suppose $0$ is an  accumulation point from the right. The marked points on the positive ray can be labelled by $q_i$, $i\in \bZ$, so that $q_i<q_j$ if $i<j$, $q_i\to \infty$ as $i\to \infty$ and $q_i\to 0$ as $i\to -\infty$ (see Fig.~\ref{almost-pf}, right). Assumption~(2) implies that every arc emanating from $q_i$ have the other endpoint at some $q_j$ (as it cannot terminate at any point $r\leq 0$ or $r=\infty$).  

Let $u_0=q_0$. There are finitely many triangles of $T$ incident to $q_0$ (as $q_0$ is not an accumulation point) and the leftmost of these triangles has at least one vertex in the interval $(0,q_0)$, while the rightmost triangle has a vertex in $(q_0, \infty)$. Moreover, the arcs (and boundary arcs) incident to $q_0$ split into a family of consecutive arcs landing in $(0,q_0)$ and a family of consecutive arcs landing in $(q_0,\infty)$. Hence, exactly one triangle incident to $q_0$ has vertices $a_0q_0b_0$ such that $0<a_0<q_0<b_0<\infty$. Denote this triangle by $t_0$. The side $a_0b_0$ of $t_0$ belongs to some other triangle $t_1\in T$, with the third vertex either at a point $a_1<a_0$ or at a point $b_1>b_0$. Denote the leftmost vertex of $t_1$ by $a_1$ and the rightmost vertex of $t_1$ by $b_1$ (with either $a_1=a_0$ or $b_1=b_0$).
 Similarly, this new side $a_1b_1$ of $t_1$ also belongs to the next triangle $t_2$. Proceeding in this way we obtain a sequence of triangles $t_i$ and a sequence of arcs $a_ib_i$. Notice that all $a_i,b_i$ are positive, the sequence  $\{a_i\}$ is monotone decreasing and the sequence $\{b_i\}$ is monotone increasing. Hence both sequences are converging, which implies that there exists a limit arc $\gamma$ (a limit of the arcs $a_ib_i$). This contradicts assumption~(1) unless $\gamma$ is actually a boundary segment with endpoints at $0$ and $\infty$. In the latter case, we see that there are no marked points in the negative ray, and the arcs $a_ib_i$ define an almost elementary zig-zag of $D$. Moreover, the set of vertices $\{a_i,b_i\}$ is obtained from the set $\{q_i\}$ by removing some marked points (so that the total number of accumulation points in the disc remains equal to 2), removing some further vertices from the set   $\{a_i,b_i\}$ we can obtain an elementary zig-zag. By definition, this implies that the original triangulation of $D$ is  an almost elementary zig-zag.

This completes the consideration of Case~1.

\medskip
\noindent
\underline{\bf Case  2:}
Suppose that $D$ contains a unique accumulation point, say at $\infty$.
If it has a two-sided accumulation point, then the case is considered exactly in the same way as in Case~1.b.
So, suppose it is a one-sided limit, say from the right. Then we may assume that the marked points $q_i, i\in\bN,$ on $D$ satisfy $q_1<q_2<q_3<\dots$. 
Using the same construction as in Case~1.a,  we arrive to an almost elementary incoming or outgoing fan.  
Notice that in this setting there may be an arc from $q_i$ to $\infty$, and if there are infinitely many of such arcs, the fan under consideration is outgoing.
\end{proof}

\begin{prop}
\label{Lem:FinArcs-->Discs} 
Fix a triangulation $T$ of $\cals.$ 
There exists a finite set of arcs  $\gamma_i$ in $T$ with $i=1,\ldots,n$ such that $\cals \backslash \{\gamma_1,\ldots,\gamma_n\}$ is a  finite union  of almost elementary domains.
\end{prop}

\begin{proof} 
For each puncture on $\mathcal S,$ we cut along an arc coming to this puncture to get an unpunctured surface.
Then, we  cut along each of the finitely many limit arcs to obtain finitely many surfaces. 
On each connected component, there may be $3$  different types of arcs, see Fig.~\ref{arcs}:
\begin{itemize}
\item[A.]  arcs connecting two boundary components;
\item[B.]  arcs with two endpoints on the same boundary component, which cut a disc from $\cals$ (in other words, these arcs are contractible to the boundary of $\cals$);
\item[C.] a non-trivial arc with endpoints at the same boundary but not contractible to this boundary.
\end{itemize}

\begin{figure}[!h]
\begin{center}
\epsfig{file=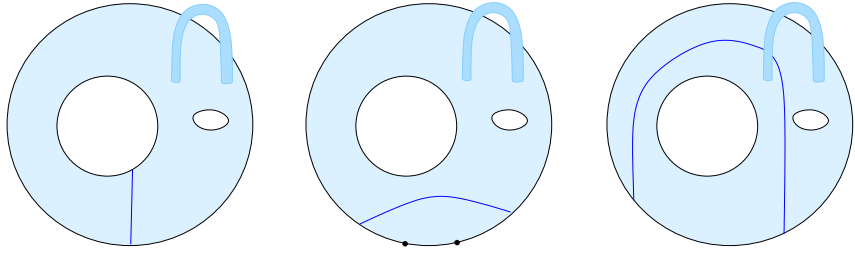,width=0.7\linewidth}
\caption{Examples of arcs of types A (left), B (middle), C (right).}
\label{arcs}
\end{center}
\end{figure}

If there is an arc of type A, we cut along this arc and reduce the number of boundary components. If there is an arc of type C, we cut along it  either to reduce the genus or  to split the surface into two connected components each having either  smaller genus or smaller number of boundary components. If there is no arc of type A or C, then the connected component is a disc.

Now we are left with finitely many discs  $D_i$ each of them having finitely many accumulation points.
The discs having no accumulation points on the boundary may be cut into finitely many separate triangles. Hence, it is left to consider a disc $D_i$ containing at least one accumulation point. If there is an arc $\gamma\in T$  cutting  $D_i$ into two parts $D^+_i$ and $D^-_i$ such that $\acc(D^+_i)<\acc(D_i)$ and $\acc(D^-_i)<\acc(D_i),$ then we cut $D_i$ along $\gamma$. Otherwise, Lemma~\ref{l-almost elementary} implies that $D_i$ is an almost elementary  domain (see Fig.~\ref{Fig:Cut_IsolateDics}). Clearly, this process terminates in finitely many steps and results in finitely many discs all triangulated as almost elementary domains, as required. 
\end{proof}

\begin{figure}[!h]
\begin{center}
\epsfig{file=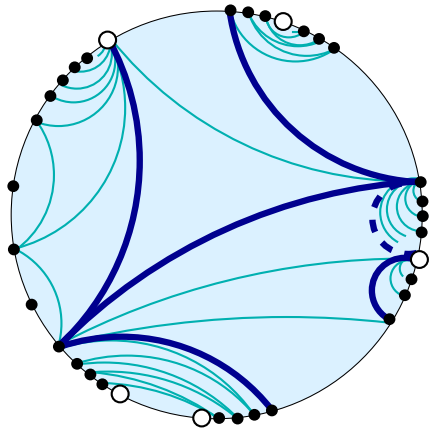,width=0.35\linewidth}
\caption{Example of cutting along thick arcs into almost elementary fans and zig-zags.} 
\label{Fig:Cut_IsolateDics}
\end{center}
\end{figure} 

\begin{remark}
\label{non-uniquenes of elementary domains}
The choice of cuts splitting  $\cals$ into almost elementary domains in Proposition~\ref{Lem:FinArcs-->Discs} is not unique. In particular, one can always cut a finite number of triangles from any infinite domain.
\end{remark}
\subsection{Transitivity of infinite mutations}
The main result of this section is that any two infinite triangulations of $\cals$ are connected by a sequence of mutations.
This generalises the result of~\cite{BG} which is obtained in a slightly different setting and which establishes transitivity of transfinite mutations for the case of {\it completed} infinity-gon.

\begin{theorem} 
\label{Thm:T-->T'}
For any two triangulations $T$ and $T'$ of $\cals,$ there exists a mutation sequence $\mubul$ such that $\mubul(T)=T'$, where $\mubul$ is a finite mutation,  a finite sequence of infinite mutations or an infinite sequence of infinite mutations. 
\end{theorem}

We will first prove a couple of technical lemmas and the proof of the theorem will be presented at the end of this section.

\begin{lemma} 
\label{Rem:Almostelementary-->Elemenatry}
\begin{enumerate}
\item Let $\calf$ be an almost elementary fan. Then there exists an infinite mutation $\muone$ such that $\muone(\calf)$ is an elementary fan. 
\item Let $\calz$ be an almost elementary zig-zag. Then there exists an infinite mutation $\muone$ such that  $\muone(\calz)$ is an elementary zig-zag with (possibly infinitely many) finite polygons  attached along its bases.
\end{enumerate}
\end{lemma}

\begin{proof} Let $\cald$ be an almost elementary fan or zig-zag. By Definition~\ref{Def:UnderlyingTriangulation}, one can remove   a (possibly infinite) number of boundary marked points (together with the arcs incident to them)  from $\cald$,  so that the resulting surface $\cald'$ is an elementary domain and $\acc(\cald')=\acc(\cald)$. From each triangle $t_i$ of $\cald'$, we have removed at most finitely many marked points (otherwise we would get a contradiction to condition (a) of Definition~\ref{Def:UnderlyingTriangulation}). Hence, each triangle $t_i$ of $\cald'$ is subdivided in $\cald$ into finitely many triangles constituting a finite triangulated polygon $\tilde t_i$ in $\cald$. We can transform the triangulation of $\tilde t_i$ to a finite fan as shown in Fig.~\ref{fig:almost elem-->elem} in finitely many flips. Doing so for $\tilde t_i, i\in\bN,$  successively, we will get an infinite mutation transforming the triangulation of $D$ to the required pattern.  

\end{proof}

\begin{figure}[!h]
\begin{center}
\epsfig{file=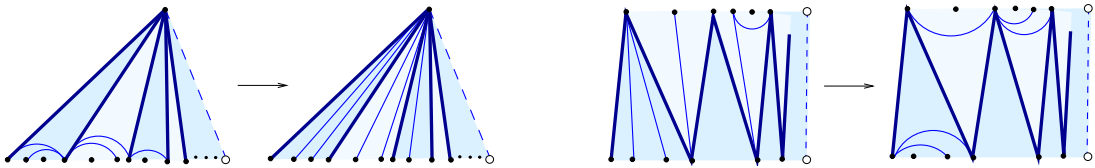,width=0.9\linewidth}
\caption{To the proof of Lemma~\ref{Rem:Almostelementary-->Elemenatry}.} 
\label{fig:almost elem-->elem}
\end{center}
\end{figure}

\begin{definition}[Domain of $\gamma$ in $T$,  almost elementary domains  of $\gamma$] 
\label{Def: domain of gamma}
Let $\gamma\in\cals$ be an arc and $T$ be a triangulation of $\cals$.
\begin{itemize}
\item[(1)] A {\it domain}  $\cald^T_{\gamma}$   of $\gamma$ in $T$  is an open set which consists of triangles intersected by $\gamma$
given by 
$$\cald^T_{\gamma}=(\bigcup\limits_{i\in I}  \Delta\strut^\mathrm{o}_i) \bigcup (\bigcup\limits_{j\in J} \gamma_j),$$
where $J$ is the index set of all arcs in $T$ intersected by $\gamma$,  $I$ is the index set of triangles of $T$ having at least one of their boundary arcs in the set  $\{\gamma_j \ | \ j\in J\}$, and  $\Delta\strut^\mathrm{o}$ is the interior of a triangle $\Delta$. 

\item[(2)] As it is the case for any surface, the domain $\cald^T_{\gamma}$   of $\gamma$ may be cut into almost elementary domains (see Proposition~\ref{Lem:FinArcs-->Discs}). These domains for  $\cald^T_{\gamma}$  will be called 
 {\it almost elementary domains} of $\gamma$.
\end{itemize}
\end{definition}

\begin{figure}[!h]
\begin{center}
\psfrag{a}{\small (a)}
\psfrag{b}{\small (b)}
\psfrag{c}{\small (c)}
\psfrag{g}{{\color{red} $\gamma$}}
\epsfig{file=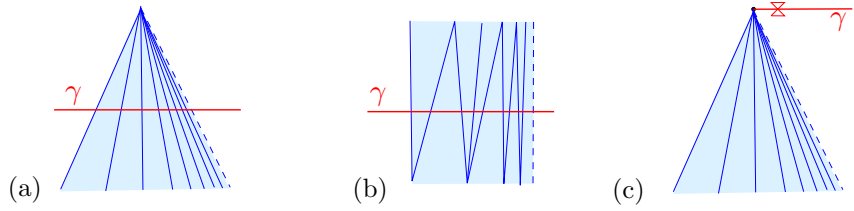,width=0.7\linewidth}
\caption{ Almost elementary domains of $\gamma$: (a) elementary fan, (b)~(almost) elementary zig-zag, (c) in the punctured case, a fan with the source in a puncture $p$ is an elementary domain for an arc $\gamma$ incident to $p$ if $\gamma$ is tagged oppositely at $p$  with respect to the tagging of $T$.}
\label{elem_crossings}
\end{center}
\end{figure}

Examples of elementary domains of an arc $\gamma$ are shown in Fig.~\ref{elem_crossings}.

\begin{remark}
In the case of a punctured surface we need to use the notion of crossing of tagged arcs, see~\cite[Definition~7.4]{FST}.  
In particular, the arc $\gamma$ in Fig.~\ref{elem_crossings}(c) (having an endpoint at a puncture $p$  and  tagged   oppositely at $p$ with respect to the tagging of $T$)  crosses every arc of the fan, so the whole fan will be an almost elementary domain for $\gamma$.

\end{remark}

\begin{remark}
\label{rem: domain of gamma=in+zig}
The domain of $\gamma$ in $T$ is a finite union of almost elementary incoming fans and almost elementary zig-zags.
Indeed, it cannot contain an almost elementary outgoing fan, as no arc can  cross infinitely many arcs of an outgoing fan.
(Also, the union is finite as the surface is a finite union of almost elementary domains).

\end{remark}

\begin{remark}
Almost elementary domains of $\gamma$ in $T$ are not uniquely defined  (compare to Remark~\ref{non-uniquenes of elementary domains}).
\end{remark}

\begin{remark} 
\label{fan is elementary}
If $D$ is an almost elementary fan domain of an arc $\gamma$, then (after removing finitely many triangles) $D$ is actually an elementary fan domain. Indeed,
as  $\gamma$ intersects every triangle in  the domain $\cald^T_{\gamma}$,  $\gamma$ crosses only the arcs of the triangulation of $D$ incident to the source of the fan (except for finitely many triangles at the end). 
\end{remark}

\begin{lemma}
\label{split domain}
For any arc $\gamma\in\cals$ and a triangulation $T$, the domain $\mathcal D_\gamma^T$ of $\gamma$ is a finite union of  disjoint almost elementary domains $R_1,...,R_k$ such that for every domain $R_i$ one of the following holds:
\begin{itemize}
\item[-] either $R_i$ is a single triangle,
\item[-] or all parts of $\gamma$ crossing $R_i$ are parallel to each other  (see Fig.~\ref{pencils}).
\end{itemize}
\end{lemma} 

\begin{figure}[!h]
\begin{center}
\psfrag{g}{{\color{red} $\gamma$}}
\epsfig{file=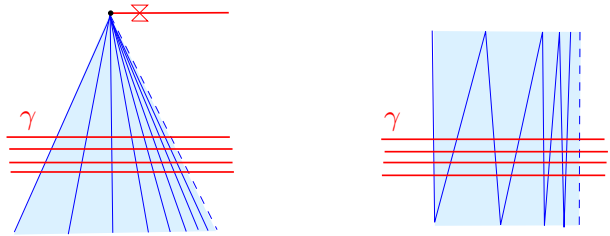,width=0.5\linewidth}
\caption{Parallel  crossings of elementary fan and zig-zag domains.}
\label{pencils}
\end{center}
\end{figure}

\begin{proof}
Applying Proposition~\ref{Lem:FinArcs-->Discs} to  $\mathcal D_\gamma^T$ we can find a finite set of arcs $\alpha_1,\dots,\alpha_k\in T$ such that $\mathcal D_\gamma^T\setminus \{\alpha_1,\dots,\alpha_k \}$ is a  finite union of unpunctured discs $R_i$  with at most two accumulation points. In each of these discs  the triangulation is an almost elementary fan or almost elementary zig-zag.
There are only finitely many pieces $\gamma_j^i$ of the arc $\gamma$ in each $R_i$ since the discs $R_i$ are obtained by finitely many cuts,   each cutting $\gamma$ into finitely many pieces (as any two arcs on  $\mathcal D_\gamma^T$  have finitely many intersections by Proposition~\ref{Prop:ArcsFinInter}).

Let  $\gamma_j^i\in R_i$ be a piece of $\gamma$ crossing finitely many triangles in $R_i$. Then we can cut these triangles out of $R_i$ and consider each of them as an individual elementary domain.
Therefore, we can assume that $R_i$ is an almost elementary fan or zig-zag and every part  $\gamma_j^i\in R_i$ of $\gamma$ crosses the limit arc (or terminates at the accumulation point in the case of a zig-zag around an accumulation point).

Furthermore, as there are finitely many  pieces  $\gamma_j^i$ in $R_i$, there are finitely many endpoints of   $\gamma_j^i$
leaving $R_i$ through the base/bases.
Thus, we can also remove finitely many triangles containing these endpoints, see Fig.~\ref{cuts}. This will cut $R_i$ into finitely many subdomains (one infinite and several finite ones). We will cut the finite domains into finitely many triangles. The remaining infinite domain then is crossed in a parallel way as there is a unique way to cross a fan or a zig-zag not crossing the bases. 

Therefore, $\mathcal D_\gamma^T$   is a finite union of disjoint triangles, infinite fan domains and infinite zig-zag domains, and each infinite domain is crossed by pieces of $\gamma$ in a parallel way.
\end{proof}

\begin{figure}[!h]
\begin{center}
\psfrag{g}{{\color{red} $\gamma$}}
\epsfig{file=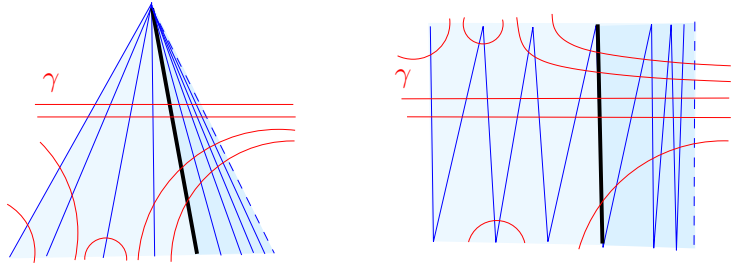,width=0.6\linewidth}
\caption{Cutting fan and zig-zag domains to isolate parallel crossings by $\gamma$.}
\label{cuts}
\end{center}
\end{figure}

\begin{lemma}
\label{Lem:ChangeMutatforGamma}
For any arc $\gamma\in\cals$ and triangulation $T$ of $\cals,$ there exists a finite sequence of infinite mutations $\mun$ of $T$ such that $T'=\mun(T)\in \mathbb T, \gamma\in T'$ and 
$T|_{\cals \backslash \cald^T_{\gamma}} =T'|_{\cals \backslash \cald^{T}_{\gamma}}.$
Moreover, $\mun$ can be chosen so that no elementary mutation inside $\mun$ flips any arc $\alpha\in \cals \backslash \cald^{T}_{\gamma}$.
\end{lemma}

\begin{proof}
By Lemma~\ref{split domain}, the domain  $\mathcal D_\gamma^T$ of $\gamma$ is a finite union of  disjoint almost elementary domains $R_1,...,R_k$ such that for every infinite domain $R_i$ the crossing of $R_i$ by $\gamma$ looks like a parallel pencil. More precisely, among  $R_i,...,R_k$ there are finitely many triangles, finitely many elementary fans and finitely many almost elementary zig-zags. 

Given an elementary fan domain $R_i$, let $\mu^{(1)}_i$ be an infinite mutation shifting the source of the fan as in the left of
Fig.~ \ref{intermediate1} 
(see Fig.~\ref{Fig:shiftsource} for the construction of  $\mu^{(1)}_i$). Denote $\mu_{R_i}:=\mu^{(1)}_i$.

\begin{figure}[!h]
\begin{center}
\psfrag{i}{$\mu_{R_i}$}
\psfrag{j}{$\mu_{R_j}$}
\epsfig{file=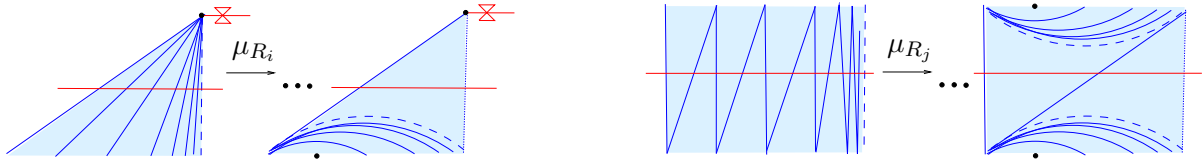,width=0.99\linewidth}
\caption{ $\mu_{R_i}$: resolving infinitely many crossings in a fan and zig-zag domains.}
\label{intermediate1}
\end{center}
\end{figure}
 
Similarly, for every almost elementary zig-zag domain $R_j$, we first apply an infinite mutation  $\mu^{(1)}_j$ to turn it to an elementary zig-zag (with finite polygons attached along its boundary arcs, see Lemma~\ref{Rem:Almostelementary-->Elemenatry}). Then we apply a composition of two infinite mutations $\mu^{(2)}_j$ which turns the zig-zag into two fans as in Fig.~\ref{Fig:switch(zigzag-->fan}. Denote $\mu_{R_j}:=\mu^{(2)}_j \circ \mu^{(1)}_j$, see Fig.~\ref{intermediate1}~(right).

Notice that as the domains  $R_1,...,R_k$ are disjoint, the above mutations $\mu_{R_i}$ do not interact, i.e. they commute. Applying  the composition 
$ \mun=\mu_{R_k} \circ \dots \circ  \mu_{R_1}$  to $T$, we obtain a triangulation $T^*:=\mun(T)$ having finitely many crossings with $\gamma$ inside every $R_i$. This means that $\cald^{T^*}_{\gamma}$ is a finite surface, and in view of \cite{H} there exists a finite mutation $\mu^{(0)}$ such that $\gamma\in\mu^{(0)}(T^*)\in\bT.$ Hence, $\mu^{(0)}\circ\mun$ is a finite sequence of infinite mutations of the required form.

Since we only consider  arcs lying inside the domain  $\cald^{T}_{\gamma}$ of $\gamma$ in the argument of the proof,
 no elementary mutation inside $\mu^{(0)}\circ\mun$ flips any arc $\alpha\in \cals \backslash \cald^{T}_{\gamma}$.
\end{proof}

\begin{remark}
Applying Lemma~\ref{Lem:ChangeMutatforGamma}, we can try to prove Theorem~\ref{Thm:T-->T'} (transitivity of sequences of infinite mutations)
in the following way: 
\begin{itemize}
\item[(1)] label the arcs of $T'$ with positive integers, so that $T'=\{\gamma_i \mid i\in \bN \}$; 
\item[(2)] apply   Lemma~\ref{Lem:ChangeMutatforGamma} repeatedly, to obtain first a triangulation $T_1$ containing $\gamma_1$,
then a triangulation $T_2$ containing $\gamma_1,\gamma_2$; after $n$ steps we get  a finite sequence of infinite mutations transforming $T$ to a triangulation $T_n$ containing $\gamma_1,\dots,\gamma_n$; 
\item[(3)] applying  Lemma~\ref{Lem:ChangeMutatforGamma} infinitely many times we get a triangulation containing all arcs $\gamma_i\in T'$, i.e. we obtain $T'$.
\end{itemize}

The only problem with this algorithm is that in general it is not clear whether the infinite sequence of infinite mutations is admissible (i.e. whether the orbit of every arc converges).

\bigskip

To prove Theorem~\ref{Thm:T-->T'}, we will first show the statement for almost elementary domains in Lemma~\ref{mutating to elementary}.
Both the proof of  Lemma~\ref{mutating to elementary}  and the proof of Theorem~\ref{Thm:T-->T'}
will be based on Lemma~\ref{Lem:ChangeMutatforGamma} and Observation~\ref{not in domain}.
\end{remark}

\begin{observation}
\label{not in domain}
Let $\cals$ be an infinite surface, $T$ be a triangulation of $\cals $, let $\alpha\in T$ and $\gamma\notin T$ be two arcs on $\cals$.
If $\alpha\cap \gamma=\emptyset$ then $\alpha\cap \cald_\gamma^T=\emptyset$.

\end{observation}

\begin{proof}
By Definition~\ref{Def: domain of gamma} of  $\cald_\gamma^T$, an arc $\alpha$ have common points with  $\cald_\gamma^T$ if and only if
$\alpha\cap \gamma\ne\emptyset$. 

\end{proof}

\begin{lemma}
\label{mutating to elementary}
Let $(\cald,T')$  be an almost elementary domain, and let $T$ be another triangulation of $\cald$. Then there exists a   mutation sequence $\mubul$ such that $\mubul(T)=T'$,  where $\mubul$ is a finite mutation,  a finite sequence of infinite mutations or an infinite sequence of infinite mutations.

\end{lemma} 

\begin{proof}
First, suppose that $\cald$ is an elementary domain. 
Label the arcs of $T'$ by $\gamma_i$, $i\in \bN$ so that the sequence $\{\gamma_i\}$ converges as $i\to \infty$ 
(in other words, label the arcs from top to bottom in Fig.~\ref{Fig:local}(c)-(e), and from bottom to top in  Fig.~\ref{Fig:local}(b)). By   Lemma~\ref{Lem:ChangeMutatforGamma} we can find a finite sequence of infinite mutations  $\mu_1^{(n)}$ such that 
$\gamma_1\in T_1:=\mu_1^{(n)}(T)$, then we find a sequence  $\mu_2^{(n)}$ such that 
$\gamma_2\in T_2:=\mu_2^{(n)}\circ\mu_1^{(n)}(T)$. As $\gamma_2$ does not intersect $\gamma_1$ (lying in the same triangulation $T'$), we see from Observation~\ref{not in domain} that $\gamma_1$ is disjoint from $\cald_{\gamma_2}^{T_1}$, which implies that 
$\gamma_1$ is not touched by any elementary mutation inside $\mu_2^{(n)}$. In particular, we have $\gamma_1,\gamma_2\in T_2$.
Repeatedly applying  Lemma~\ref{Lem:ChangeMutatforGamma}  for $\gamma_i$, we obtain a finite sequence of infinite mutations
$\mu_i^{(n)}\circ\dots\circ\mu_2^{(n)}\circ\mu_1^{(n)}$ which transforms $T$ to $T_i$ such that $\gamma_1,\dots,\gamma_i\in T_i$. 
Doing this infinitely many times, we obtain a set of arcs on the surface containing all arcs $\gamma_i$ of $T'$, hence we obtain the triangulation $T'$. 

To finish the proof, we are left to show that the composition  $\mu=\ldots\circ\mu_2^{(n)}\circ\mu_1^{(n)}$ is an admissible sequence of infinite mutations, i.e. that the orbit of every arc of $T$ or $T_i$  either stabilises or converges.
We will first show this for an arc $\alpha$  crossing finitely many arcs of $T'$ and then we will extend the proof for all other arcs.

\medskip
\noindent
{\bf 1: arcs with finitely many crossings.} 
Let $\alpha\in T$ be an arc. Notice that if $\alpha\cap \gamma_i=\emptyset$ then we obtain that
$\mu_j^{(n)}\circ\dots\circ\mu_2^{(n)}\circ\mu_1^{(n)}(\alpha)\cap \gamma_i=\emptyset$
(this follows by induction from  Observation~\ref{not in domain} and the fact that
 $\mu_i^{(n)}$ acts in the domain $\cald_{\gamma_i}^{T_{i-1}}$).
Hence, the orbit of every arc $\alpha\in T$  crossing finitely many arcs of $T'$ stabilises.

\medskip
\noindent
{\bf 2: arcs with infinitely many crossings.} 
Notice that the arcs $\gamma_1,\dots,\gamma_i$ are not touched by any of the mutations $\mu_j^{(n)}$ with $j>i$.
This implies that as $j$ grows, all elementary mutations contained in  $\mu_j^{(n)}$  are performed inside smaller and smaller regions, which shrink in the limit to the accumulation point (when $\cald$ is as in Fig.~\ref{Fig:local} (c) or (d)) or to the limit arc (in cases when $\cald$ is as in Fig.~\ref{Fig:local} (b) or (e)). In the former case this immediately implies that the orbit of every arc either stabilises or converges.
The latter case (i.e. the cases of an incoming fan and of a zig-zag converging to a limit arc) require more work as the orbit of an arc may have a subsequence converging to the limit arc and another subsequence converging to one of the accumulation points. 

{\bf 2.a:}
 {\it $(\cald,T')$ is an  incoming fan.} In this case an arc $\alpha\in T$ crossing infinitely many arcs of the incoming fan $T'$ have one of its endpoints at the accumulation point and another at some marked point in the base of the fan.
If after several mutations the arc $\alpha$ transforms to an arc $\alpha'$ 
connecting the source of the fan to the base, then 
none of the further mutations  will change $\alpha'$ anymore (since $\alpha'$ does not cross any of $\gamma_i$ and so does not lie in any of domains $\cald_{\gamma_{i+1}}^{T_i}$), hence the orbit of $\alpha$ will stabilise.
If every arc in the orbit of $\alpha$ has one end at the accumulation point and another end at the base,
then the orbit of $\alpha$ converge to the accumulation point.
  
{\bf 2.b:}
{\it $(\cald,T')$ is a zig-zag converging to a limit arc.} In this case, an arc  $\alpha\in T$ crossing infinitely many arcs of $T'$ connects a marked point $p_i$ in one of the bases of the zig-zag to an accumulation point $q$ lying either on the other  or on the same base of the zig-zag (see Fig.~\ref{convergence} (a) and (b)). 
 We may assume that the mutations  $\mu_1^{(n)},\dots,\mu_i^{(n)}$ are already applied, so that all arcs $\gamma_j=p_jp_{j+1}$ for $j\le i$ are already in the triangulation $T_i$. Also, we assume $\alpha=p_iq\in T_i$.
We need to  construct the mutation $\mu_{i+1}^{(n)}$  clearing the arc $\gamma_{i+1}=p_{i+1}p_{i+2}$ from the intersections. 
 In both cases of 
Fig.~\ref{convergence} (a) and (b), all almost elementary domains of $\gamma_{i+1}$ in ${T_i}$ are fans (as the the arc $\alpha=p_iq$ separates the left limit from the right one in $T_i$, so  the zig-zags are impossible). We will first remove infinite number of intersections (if they exist) by shifting the source of fans as in Fig.~\ref{Fig:shiftsource}, then we remove all the other intersections in order of their appearance on the (oriented) path $p_{i+1}p_{i+2}$. Then the arc $\alpha=p_iq$ will be flipped either to an arc having no endpoints at an accumulation point (which would imply its orbit will stabilise as shown in  Case~1) or to an arc connecting $p_{i+1}$ with the accumulation point on the other base. The latter  may only occur in the case of Fig.~\ref{convergence} (a), and if this configuration repeats infinitely many times, then the orbit of $\alpha$ converges to the limit arc.

\begin{figure}[!h]
\begin{center}
\psfrag{a}{\small (a)}
\psfrag{b}{\small (b)}
\psfrag{al}{\color{red} \scriptsize $\alpha$}
\psfrag{A}{\scriptsize $q$}
\psfrag{1}{\scriptsize $p_1$}
\psfrag{2}{\scriptsize $p_2$}
\psfrag{3}{\scriptsize $p_3$}
\psfrag{i-1}{\scriptsize $p_{i-1}$}
\psfrag{i}{\scriptsize $p_i$}
\psfrag{i+1}{\scriptsize $p_{i+1}$}
\psfrag{i+2}{\scriptsize $p_{i+2}$}
\epsfig{file=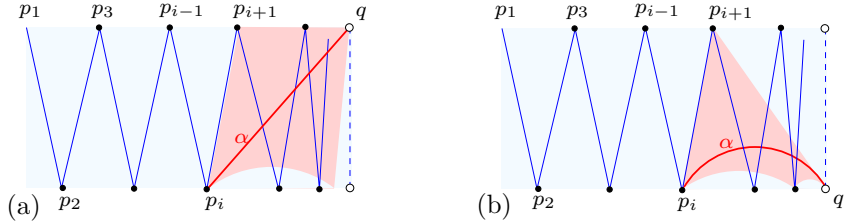,width=0.69\linewidth}
\caption{Arcs crossing infinitely many arcs of a zig-zag: if an arc $\alpha=p_iq$ is the first arc crossed by the (oriented) arc $p_{i+1}p_{i+2}$ then $\alpha$ flips to an arc crossing only finitely many arcs of the zig-zag.}
\label{convergence}
\end{center}
\end{figure}

So, in each of the cases the orbits of all arcs either stabilise or converge, hence the infinite mutations $\mu_i^{(n)}$ compose an admissible infinite sequence of infinite mutations.

Finally, if $\cald$ is not an elementary domain but an {\it almost} elementary one, then exactly the same reasoning works
for appropriate enumeration of arcs of $T'$ (i.e. for the enumeration where the arcs forming the underlying elementary domain triangulation have increasing numbers and additional arcs lying ``between'' $\gamma_i$  and $\gamma_j$  have the numbers between $i$ and $j$).

\end{proof}

\begin{remark}
We expect that  Lemma~\ref{mutating to elementary} can be derived from the results of~\cite{BG}, but it is not a straightforward task due to differences in the definitions, see Remark~\ref{compareBG}.

\end{remark}

\begin{corollary}
\label{Cor: mutating to elementary}
Let $\cald$ with triangulation $T'$ be an almost elementary domain, and let $T$ be another triangulation of $\cald$.
If every arc of $T'$ intersects only finitely many arcs of $T$,  then there exists an infinite mutation $\muone$ such that $\muone(T)=T'$.

\end{corollary}

\begin{proof}
We construct the sequence of mutations as in the proof of Lemma~\ref{mutating to elementary}, but notice that when each arc of $T$ intersects only finitely many arcs of $T'$, one can substitute each of  $\mu_i^{(n)}$ by a finite mutation.
Moreover, the proof of Lemma~\ref{mutating to elementary} also shows that the orbit of every arc  stabilises. 

\end{proof}

We are now ready to prove  the main theorem of this section.

\begin{proof}[Proof of Theorem~\ref{Thm:T-->T'}]
We will start by choosing arcs $\gamma_1, \ldots, \gamma_s$ in $T',$ $s\in\bN,$ as in Proposition~\ref{Lem:FinArcs-->Discs}, which split the surface $\cals$ into finitely many almost elementary domains $\cald_1,\ldots,\cald_{k}$ for some $k\in\bN$. 
Then we  apply Lemma~\ref{Lem:ChangeMutatforGamma} repeatedly for each of  $\gamma_1\dots,\gamma_s$ in the same way as we did in the proof of Lemma~\ref{mutating to elementary}. We obtain a finite sequence of infinite mutations $\mun$ such that $T^*:=\mun(T)$ is a triangulation containing the arcs  $\gamma_1, \ldots, \gamma_s$.

As  $\cald_1,\ldots,\cald_{k}$  are almost elementary domains of $T'$,  Lemma~\ref{mutating to elementary} implies that there exist finite or infinite sequences of infinite mutations $\mu_i^{(n)}$ transforming $T^*|_{\cald_i}$ to $T'|_{\cald_i}$. All mutations in these sequences only flip the arcs in the corresponding almost elementary domains, so the infinite mutations contained in   $\mu_i^{(n)}$ and  $\mu_j^{(n)}$ commute for $i\ne j$. Which means that we can apply the first mutations from each of  $\mu_1^{(n)},\dots \mu_k^{(n)}$, then the second mutations from each of  $\mu_1^{(n)},\dots \mu_k^{(n)}$, then the third ones and so on, forming an infinite sequence of infinite mutations and obtaining all arcs of $T'$ in the limit (and this is an admissible  sequence of infinite mutations since  $\mu_i^{(n)}$ is admissible sequence for each $i$).  

\end{proof}

\subsection{Infinite sequences are unavoidable}
In the main result of this section, Theorem~\ref{infinite sequences}, we will show that for every infinite surface there are triangulations $T$ and $T'$ such that it is not possible to transform $T$ to $T'$ using finitely many infinite mutations, i.e. the use of infinite sequences of infinite mutations is necessary  to ensure transitivity of mutations on triangulations of a given surface.

\begin{notation}[Intersection number]
Denote by 
\begin{itemize}
\item[-]
$|\gamma\cap \alpha|$ the minimum number of intersection points of representatives of the arcs $\gamma$ and $\alpha$ on $\cals$;

\item[-]
$|\gamma\cap T|=\sum_{\tau\in T}|\gamma\cap \tau|$ the sum of intersection numbers $|\gamma\cap \tau|$ for all $\tau\in T$ (this may be infinite);
\item[-]
$|T'\cap T|=\sum_{\gamma\in T'}|\gamma\cap T|$ the sum of all intersection numbers of all  arcs  in the triangulation $T$ with the arcs of the triangulation $T'$ (this may also be infinite).
\end{itemize}

\end{notation}

\begin{remark}
\begin{itemize}
\item[(a)] Clearly, $|T'\cap T|=|T\cap T'|$.
\item[(b)] In a punctured surface, the intersection numbers for the tagged arcs are computed as defined in~\cite[Definition~8.4]{FST}.
In particular, for  arcs $\gamma\ne \alpha$   on $\cals$ we have $|\gamma\cap \alpha|<\infty$; also $\gamma\cap \alpha=0$ if and only if $\gamma$ is compatible with $\alpha$.
\end{itemize}
\end{remark}

\begin{definition}[Bad arcs for a given triangulation] 
For an arc $\gamma\in\cals$ and a triangulation $T$ of $\cals,$ we say $\gamma$ is a \emph{bad arc for $T$} if $|\gamma\cap T|=\infty.$ We denote by $\bad_T T'$ the set of all arcs of $T'$ which are bad for the triangulation $T$.
We also denote by $|\bad_T T'|$ the cardinality of this set (which may be infinite).

\end{definition}

\begin{definition} Let $p$ be a left (resp. right) accumulation point of boundary marked points of a surface $\cals$ and $T$ be a triangulation of $\cals$. A $\overrightarrow p$-\emph{domain} (resp. $\overleftarrow p$-\emph{domain}) \emph{of} $T$ is an almost elementary domain of $T$ containing infinitely many marked points accumulating to $p$ from the left (resp. right). We will denote a $\overrightarrow p$-domain (resp. $\overleftarrow p$-domain) of $T$ by $\cald^T_{\overrightarrow p}$ or simply by $\cald_{\overrightarrow p}$ (resp. $\cald^T_{\overleftarrow p}$ or $\cald_{\overleftarrow p}$). When there is no need to specify which side the accumulation point $p$ is approached from, we will use the simplified notation $\cald_{\overline{p}}$.
\end{definition}

\begin{remark} Given triangulations $T$ and $T'$ of $\cals$ and an accumulation point $p$, by $\cald_{\overline{p}}$ and $\cald'_{\overline{p}}$ we will mean $\overline{p}$-domains of $T$ and $T'$ where $\overline{p}$ stands either for $\overrightarrow{p}$ or $\overleftarrow{p}$ in both $\cald_{\overline{p}}$ and $\cald_{\overline{p}}'$ simultaneously.
\end{remark}

\begin{lemma}
\label{in-out-zig}
Let $T$ and $T'$ be triangulations of $\cals$ and let $p$ be an accumulation point. Let $\cald_{\overline{p}}$ and $\cald'_{\overline{p}}$ be $\overline{p}$-domains of $T$ and $T'$, respectively.

If $\cald_{\overline{p}}$ is an almost elementary incoming fan or an almost elementary zig-zag and $\cald'_{\overline{p}}$ is an almost elementary outgoing fan, then $|Bad_{T}T'|=\infty$. 

\end{lemma}

\begin{proof}
Every arc of $\mathcal D'_{\overline{p}}$ incident to $p$ crosses infinitely many arcs of $T$, so there are infinitely many bad arcs in $T'$ with respect to $T$.

\end{proof}

\begin{proposition} 
\label{prop:|bad|=>more than one}
For two triangulations $T$ and $T'$ of $\cals,$ if $|\bad_T T'|=\infty$ then $\muone(T)  \neq T'$ for any  infinite mutation $\muone$.

\end{proposition}

\begin{proof}
Suppose that $\muone(T)=T'$. Then every non-limit arc of $T'$ arises after finitely many flips, i.e. intersects with finitely many arcs in $T'$, so, it is not bad for $T$. By Corollary~\ref{Cor:FiniteLimArc}, there are only finitely many limit arcs in $T'.$
This contradicts the assumption that  $|\bad_T T'|=\infty$. 

\end{proof}

\begin{theorem}
\label{infinite sequences}
Let $\cals$ be an infinite surface.
Then there exist triangulations $T$ and $T'$ such that $\mun(T)\ne T'$ for any finite sequence of infinite mutations $\mun$.

\end{theorem}

\begin{proof}
Given a triangulation $T$ and an accumulation point $p\in\cals$, we introduce a function $\mathsf{ Type}_{\overline p}(T):T\to \{\mathsf{In}, \mathsf{Out}, \mathsf{Zig}\}$ defined by
$$
\mathsf{Type}_{\overline p }(T)=
\begin{cases}
\mathsf{In},& \text{if $\mathcal D_{\overline p}$ is an almost elementary incoming fan,} \\
\mathsf{Out},& \text{if $\mathcal D_{\overline p}$ is an almost elementary outgoing fan,}\\
\mathsf{Zig},& \text{if $\mathcal D_{\overline p}$ is an almost elementary zig-zag.}
\end{cases}
$$

Now,  choose  triangulations $T$ and $T'$  so that $\mathsf{Type}_{\overline p}(T)=\mathsf{In}$ and  $\mathsf{Type}_{\overline p}(T')=\mathsf{Out}$.
Suppose there is a finite sequence of infinite mutations $\mun$ such that $\mun(T)=T'$. Denote 
$\mun=\muone_n\circ\dots\circ\muone_1$ and $T_i:=\muone_i\circ\dots\circ\muone_1(T)$.
Choose the minimal $k$, $1\leq k\leq n$, satisfying $$\mathsf{Type}_{\overline p}(T_k)=\mathsf{Out}$$ (note that  such $k$ exists as $\mathsf{Type}_{\overline p}(T)=\mathsf{In}$ and $\mathsf{Type}_{\overline p}(T')=\mathsf{Out}$).
Then $\mathsf{Type}_{\overline p}(T_{k-1})$ is either $\mathsf{In}$ or $\mathsf{Zig}$.
By Lemma~\ref{in-out-zig}, this implies that $|\bad_{T_{k-1}}{T_k}|=\infty$. So, the existence of $\muone_k$ such that 
$\muone_k(T_{k-1})=T_k$ is in contradiction with Proposition~\ref{prop:|bad|=>more than one}.

\end{proof}

\subsection{Existence of finite sequences of infinite mutations}
\label{sec hierarchy}

Theorem~\ref{Thm:T-->T'} shows that for any two triangulations $T$ and $T'$ of a surface $\cals$ there exists some mutation sequence $\mu^\bullet$ transforming $T$ to $T'$.
This sequence can be
\begin{itemize}
\item[-] a finite sequence of elementary mutations $\mu^{(0)}$;
\item[-] a finite sequence of infinite mutations $\mun$;
\item[-] an infinite sequence of infinite mutations  $\mu^{(\infty)}$.
\end{itemize}
In this section we discuss the following questions: 

\medskip

{\it 
Given two triangulations $T$ and $T'$ of a surface $\cals$, 
\begin{itemize}
\item[-] is there a finite sequence of elementary mutations $\mu^{(0)}$ satisfying  $\mu^{(0)}(T)=T'$? 
\item[-] if not, is there a finite sequence of infinite mutations $\mun$ satisfying $\mun(T)=T'$?  
\end{itemize}
}

\medskip

The first of these questions is easy:

\begin{proposition} 
\label{finite}
Two triangulations $T$ and $T'$ of $\cals$ are related by a finite sequence of elementary mutations if and only if  $|T\cap T'|<\infty$.

\end{proposition} 

\begin{proof}
Resolve each of the finitely many crossings one by one to switch from one triangulation to the other.
Conversely, a finite mutation is completely contained inside a finite union of finite polygons, so it cannot create infinitely many crossings.

\end{proof}

We will not be able to give a complete answer to the second question but we will give a sufficient condition for the existence of $\mun$
(Theorem~\ref{bad are all limit}) and a sufficient condition for the non-existence (see Theorem~\ref{no mun from stronger}). The proofs of existence statements will be mainly based on the notion of {\it bad arcs}, while the proof for non-existence will use the function introduced in the proof  of Theorem~\ref{infinite sequences} and based on the types $\mathsf{In}, \mathsf{Out}, \mathsf{Zig}$ of infinite almost elementary domains.

\begin{proposition}\label{Prop:NoBadArc} For two triangulations $T$ and $T'$ of $\cals$, if $|\bad_T T'|=0,$ (i.e. if no arc of $T'$ intersects infinitely many  arcs of $T$) then there exists an infinite mutation $\mu^{(1)}$ such that $\mu^{(1)}(T)=T'.$
\end{proposition}

\begin{proof}
As in the proof of Theorem~\ref{Thm:T-->T'},
we first use Proposition~\ref{Lem:FinArcs-->Discs} to choose finitely many  arcs $\gamma_1, \ldots, \gamma_s\in T'$, $s\in\bN,$  splitting the surface $\cals$ into finitely many almost elementary domains $\cald_1,\ldots,\cald_{k}$ for some $k\in\bN$. 
As  $|\bad_T T'|=0$, we can find a finite sequence $\muze$ of elementary mutations first transforming  $T$  to a triangulation 
containing $\gamma_1$, then to a triangulation containing $\gamma_1,\gamma_2$, and finally to a triangulation $\widetilde T$ containing all of the arcs $\gamma_1,\ldots,\gamma_s$.

Now, when  $\gamma_1,\ldots,\gamma_s$ are in the triangulation $\widetilde T=\muze(T)$, we are left to sort the question in each of the almost elementary domains  $\cald_1,\ldots,\cald_{k}$ separately.
By Corollary~\ref{Cor: mutating to elementary}, there exists an infinite mutation $\muone_i$ such that $\muone_i(\widetilde T|_{\cald_i})=T'|_{\cald_i}$. Since mutations in separate domains commute, we can
 compose the first elementary mutations of $\muone_1,\ldots,\muone_s$, then the second elementary mutations of $\muone_1,\ldots,\muone_s$, then the third ones an so on, to obtain an infinite composition $\tilde{\mu}^{(1)}$ of elementary mutations transforming 
$\widetilde T$ to $T'$. Moreover, this composition is an admissible infinite mutation, i.e. the orbit of every arc stabilise,
since each of $\muone_i$ is an infinite mutation and distinct mutations act in distinct elementary domains $\cald_i$.
Hence, $\tilde{\mu}^{(1)}\circ\muze$ is an infinite mutation transforming $T$ to $T'$.
\end{proof}

The following observation is easy to check using the definition of bad arcs.

\begin{observation}
\label{crossing limit arc}
Let $T$ and $T'$ be two triangulations of an infinite surface $\cals$. Then
\begin{itemize}
\item[(a)] if $|\alpha\cap \gamma|\ne 0$, where $\gamma\in T'$ and  $\alpha\in T$ is a limit arc,  then $\gamma\in \bad_T T'$;

\item[(b)] if $|\alpha\cap \gamma|\ne 0$ where  $\alpha\in T$ is a limit arc  and  $\gamma\in T'$ is a limit arc, then $|\bad_TT'|=\infty$;

\item[(c)] if $T$ contains a zig-zag around an accumulation point $p$ (as in Fig.~\ref{Fig:local}.d), then every arc $\gamma\in T'$ incident to $p$ 
belongs to $\bad_T T'$.
\end{itemize}

\end{observation}

\begin{lemma}
\label{bad limit arc}
 For two triangulations $T$ and $T'$ of $\cals$, if $0<|\bad_T T'|< \infty$ then there exists a limit arc  $\gamma\in T'$  satisfying
$\gamma\in \bad_T T'$.

In particular, if  $|Bad_{T}T'|=1$, then the only bad arc is a limit arc of $T'$.

\end{lemma}

\begin{proof}
Choose an arc $\beta\in \bad_T T'$. If $\beta$ is a limit  arc of $T'$ denote $\gamma:=\beta$.
Otherwise, consider the domain $\mathcal D_\beta^T $ of $\beta$ with respect to $T$. By  Remark~\ref{rem: domain of gamma=in+zig},  $\mathcal D_\beta^T $ is a finite union of almost elementary incoming fans and almost elementary zig-zags (and as $\beta$ is a bad arc, at least one of these almost elementary domains is infinite). Let $p$ be an accumulation point in $\cald_{\beta}^T$ and $\cald_{\overline{p}}$ and $\cald'_{\overline{p}}$ be the $\overline{p}$-domains of $T|_{\cald_{\beta}^T}$ and $T'|_{\cald_{\beta}^T}$, respectively. Since $|\bad_T T'|<\infty$, Lemma~\ref{in-out-zig} implies that $\mathcal D'_{\overline{p}}$ cannot be an outgoing fan.
So, each of $\mathcal D_{\overline{p}}$ and  $\mathcal D'_{\overline{p}}$ is either an incoming fan, or a zig-zag converging to a limit arc, or a zig-zag around $p$.
Below we will see that for any of these combinations either  $\mathcal D'_{\overline{p}}$ has a bad limit arc (as required) or  $\mathcal D'_{\overline{p}}$ contains infinitely many bad arcs in contradiction to the assumption. 

More precisely, 
 suppose first that $\mathcal D'_{\overline{p}}$ has a non-vanishing limit arc $\gamma$. Then $\gamma$ is incident to $p$ and either $\gamma$ crosses infinitely many arcs of $\mathcal D_{\overline{p}}$ (and hence is a bad limit arc as required), or infinitely many arcs approaching $\gamma$ in $\mathcal D'_{\overline{p}}$ 
cross the limit arc of  $\mathcal D_{\overline{p}}$ (and hence $|\bad_T T'|=\infty$ in view of Observation~\ref{crossing limit arc}.b).
If   $\mathcal D'_{\overline{p}}$ is a zig-zag around $p$, then  $\mathcal D_{\overline{p}}$ is not a zig-zag around $p$ (as  $\mathcal D'_{\overline{p}}\subset  \mathcal D_\beta^T$),
hence   $\mathcal D_{\overline{p}}$ has a limit arc $\alpha$  incident to $p$, which implies that infinitely many arcs of $\mathcal D'_{\overline{p}}$ cross $\alpha$,
and again,  Observation~\ref{crossing limit arc}.b implies  $|\bad_T T'|=\infty$ in contradiction to the assumption. 
 
\end{proof}

The following lemma will be used in the proof of Theorem~\ref{bad are all limit}.

\begin{lemma}
\label{finite surface}
 Let $S$ be a finite triangulated surface distinct from the sphere with three punctures, let $T$ and $T'$ be triangulations of $S$. Let $\alpha\in T$ and $\gamma\in T'$ be any arcs. There exists a finite sequence of flips $\mu^{(0)}$ which transforms $T$ to $T'$ and takes $\alpha$ to $\gamma$.
\end{lemma}

\begin{proof}
We will start with proving the lemma for the case when $S$ is a polygon and then extend  the proof to the general case.

\medskip
\noindent
\emph{Step 1: Proof for a polygon.}  
First, suppose that $\alpha$ crosses $\gamma$. We will remove (by flipping arcs of $T$) all other crossings of $\gamma$ with the arcs of $T$: to  do this  we flip an arc $\hat \alpha\ne \alpha$  crossing $\gamma$ closest to one of the endpoints of $\gamma$ (we can always do so if $\alpha$ is not  the only arc crossing $\gamma$). When all crossings except for the one with $\alpha$ are resolved, we can flip $\alpha$ to $\gamma$ and then apply finitely many flips (not changing $\gamma$) to obtain the triangulation $T'$.

Now, if  $\alpha$ does not cross $\gamma$, then there exists an arc $\beta\in S$ crossing both $\alpha$ and $\gamma$. We will first mutate $T$ to take $\alpha$ to $\beta$ (obtaining any triangulation $\widetilde T$ containing $\beta$) and then mutate $\widetilde T$ to $T'$ so that $\beta$ is mapped to $\gamma$. The lemma is proved when $S$ is a polygon.

\medskip
\noindent
\emph{Step 2: General case.}  
Now, let $\cals$ be any finite surface. If $\cals$ is a punctured surface, we understand the triangulation as ideal triangulation (and then the result would hold for the corresponding tagged triangulation).

As before, first we assume that $\alpha$ crosses $\gamma$. We choose one of the intersection points and resolve all other crossings of $\gamma$ with arcs of $T$ (in particular, if $\alpha$ crosses $\gamma$ more than once, we resolve these crossings and then work with the image of $\alpha$ instead of $\alpha$ itself). Then we flip (the image of) $\alpha$ to $\gamma$ and restore the triangulation $T'$ in any way not changing $\gamma$.

Suppose now that $\alpha$ does not intersect $\gamma$. Consider any triangulation $T$ of $\cals$ including $\alpha$
and $\gamma$.  If $T$ contains an arc $\delta$ such that $\delta\notin \{\alpha,\gamma\}$ and $\delta$ does not cut $\cals$ into two connected components, then the surface $\cals\setminus \delta$ has smaller number of arcs in the triangulation than $\cals$ and we may use induction on the size of surface to find a mutation sequence as required. If $T$ contains an arc $\delta$ which cuts $\cals$ into two connected components one of which  contains both $\alpha$ and $\gamma$, then again, we use induction to find the mutation sequence. So, we may assume that every arc of $T$ (distinct from $\alpha$ and $\gamma$) cuts $\cals$ into two connected components $C_1$ and $C_2$, so that $\alpha\in C_1$, $\gamma\in C_2$. If $C_1$ is not a polygon, then  $C_1$ contains an arc $\tau$ which does not cut $C_1$ into two components and is disjoint from $\alpha\cup \gamma$, so $\cals\setminus \tau$ is connected in contradiction to the assumption above. This implies that each of $C_1$ and $C_2$ are polygons. If in addition $\delta$ (the arc which cuts $\cals$ into $C_1$ and $C_2$) has two distinct endpoints, then
$\cals$ is a polygon. If the endpoints of $\delta$ coincide then each of $C_1$ and $C_2$ are discs with one boundary marked point, and, in view of the assumption that every arc on $\cals\setminus \{\alpha, \beta \}$ is separating, 
 each of $C_1$ and $C_2$ is the once punctured disc. Hence, $\cals$ is the sphere with three punctures.
 
\end{proof}

The proof of the following theorem is a bit technical. Neither the proof nor the result will be
used later in the article, so the reader can easily skip it without any harm to the understanding
of the rest of the paper.

\begin{theorem}
\label{bad are all limit}
 Given two triangulations $T$ and $T'$ of $\cals$, suppose that $|\bad_T T'|=k<\infty$ and all bad arcs of $T'$ with respect to $T$ are limit arcs of $T'$. 
 Then there exists a composition of finitely many infinite mutations $\mun$ satisfying $\mun(T)=T'$.
Moreover, $\mun$ can be chosen so that the number of infinite mutations in $\mun$ does not exceed $k+1$.

\end{theorem}

\begin{proof}
The idea of the proof is as follows. If $|\bad_T T'|=0$, the statement is shown in  Proposition~\ref{Prop:NoBadArc}, so we assume $|\bad_T T'|>0$. In Steps~1-4, we will choose an arc
 $\gamma\in\bad_T T'$ and will find a mutation $f_1\circ \muone_1$ such that 
\begin{itemize}
\item[-] $\muone_1$ is an infinite mutation;
\item[-] $f_1$ is an elementary mutation, i.e. a flip;
\item[-]  $f_1\circ \muone_1(T)$ is a triangulation $\widetilde T$  such that 
$\bad_T T'\subset \bad_{\widetilde T} T'$ and $\gamma\in \widetilde T$. 
\end{itemize}
All assumptions of the theorem will hold for triangulations $\widetilde T$ and $T'$ (but with smaller number of bad arcs), so we will be able to repeat  Steps 1--4 finitely many (i.e. at most $k$) times to arrive to a triangulation $\hat T$ such that $|\bad_{\hat T} T'|=0$ (see Step~5). This will produce a finite sequence of infinite mutations $\mun$  taking $T$ to $\hat T$. Finally, in Step~6 we 
will use Proposition~\ref{Prop:NoBadArc} to show the existence of an infinite mutation $\muone_{k+1}$ transforming $\hat T$ to $T'$.
Composing them we will get $\mu=\muone_{k+1}\circ \mun$  satisfying $\mu(T)=T'$. 

To construct the sequence $f_1\circ \muone_1$  transforming $T$ to $\widetilde T$ (see Steps~1--4),  
we will only change arcs lying inside the domain $\mathcal D_\gamma^T$ of $\gamma$ with respect to $T$. 
We will first discuss the structure of $\mathcal D_\gamma^T$ (Step~1) and the crossing pattern of $T'$ with $\cald^T_{\gamma}$ (Step~2). Then we will  construct a special triangulation $\widetilde T$ of $\mathcal D_\gamma^T$ satisfying  $|\bad_{\widetilde T} T'|=0$ (see Step~3).  Then, in Step~4, we will  show that there is a mutation $f_1\circ \muone_1$ transforming $T$ to $\widetilde T$.

\medskip
\noindent
{\it \underline{Step 1}: Structure of  $\mathcal D_\gamma^T$.}
Let $\gamma=p_1p_2$ be a bad arc of $T'$ with respect to $T$.
By assumption, $\gamma$ is a limit arc of $T'$. 
 Although there are infinitely many arcs of $T$ crossing $\gamma$,
    none of these arcs is a limit arc of $T$ (otherwise, Observation~\ref{crossing limit arc}.b would imply  $|\bad_T T'|=\infty$ in contradiction to the assumption).
Since   $\gamma$ does not cross any limit arc of $T$ and $\mathcal D_\gamma^T$ is a union of almost elementary fans and almost elementary zig-zags, the domain  $\mathcal D_\gamma^T$ has to be a union of at most two infinite almost elementary domains $\mathcal D_1$ and $\mathcal D_2$ and some finite surface $S$ connecting them (here,  $\mathcal D_1$ and $\mathcal D_2$ are attached to a finite triangulated surface $S$ along a boundary segment of $S$; also, one  of $\mathcal D_1$ and $\mathcal D_2$ may be absent, or $\mathcal D_1$ may coincide with $\mathcal D_2$). The endpoints $p_1$ and $p_2$ of $\gamma$
lie in  $\mathcal D_1$ and $\mathcal D_2$ respectively (or one of them, say $p_2$ lies in the finite surface $S$). Observe that $\cald_i$ cannot be an outgoing fan for $i=1,2$, since this would contradict the assumption that $|\bad_T {T'}|<\infty$.

Our next aim is to show that each of the discs  $\mathcal D_1$ and $\mathcal D_2$ has a unique accumulation point, namely the endpoint $p_1$ or $p_2$ of $\gamma$. Suppose that  $\mathcal D_1$ has more than one accumulation point. Then  $\mathcal D_1$ is a zig-zag converging to a non-vanishing limit arc $\delta$. As $\delta$ should not cross $\gamma$, one of the endpoints of $\delta$ coincides with $p_1$. Let $q_1$ be another endpoint of $\delta$. Notice that every arc $\tau$  starting from $p_1$ or $q_1$ in $T'$ crosses infinitely many arcs of $T$, and hence, is bad for $T$. As $|\bad_T T'|<\infty$ by assumption, we see that $T'$ contains only finitely many arcs starting from $p_1$ or $q_1$. Consider the triangulation $T'$ in the neighbourhood of $\gamma$ approaching $\gamma$ from the same side as $\delta=q_1p_1$ does. The arc $\gamma$ cannot be a limit arc from that side as there are finitely many arcs starting from $p_1$ and there is no accumulation of boundary marked points towards $p_1$ from that side. Hence, there is a triangle $t$ in $T'$ containing $\gamma$ as one of its sides and having an arc $\tau \ne\gamma$ starting at $p_1$. Then $\tau$ is not a limit arc: it is clearly not a limit arc on the side containing the triangle $t$, but also it cannot be a limit arc on the other side containing $q_1$ by a similar argument as above. So, $\tau$ crosses infinitely many arcs of $T$, but it is not a limit arc. This is a contradiction to the assumption that every bad arc of $T'$ with respect to $T$ is a limit arc of $T'$. The contradiction shows that  each of $\mathcal D_1$ and $\mathcal D_2$ is a disc with at most one accumulation point (one of the endpoints $p_1$ and $p_2$ of $\gamma$).
Note that the same reasoning works for the case when the discs  $\mathcal D_1$ and $\mathcal D_2$ coincide.

Observe also that the triangulation $T$ of  $\mathcal D_\gamma^T$ may be understood  as a limit of  growing finite connected surfaces $S_i\to  \mathcal D_\gamma^T$ as $i\to \infty$, where $S_i$ is a union  of  $i$ triangles of $T$ and $S_{i}\subset S_{i+1}$.

\begin{figure}[!h]
\begin{center}
\psfrag{D1}{ $\mathcal D_1$}
\psfrag{D2}{ $\mathcal D_2$}
\psfrag{S}{ $S$}
\psfrag{g}{\color{red} $\gamma$}
\psfrag{g+}{\color{Plum} $\gamma^+$}
\psfrag{g-}{\color{Plum} $\gamma^-$}
\psfrag{p1}{ $p_1$}
\psfrag{q1}{ $q_1$}
\psfrag{p2}{ $p_2$}
\epsfig{file=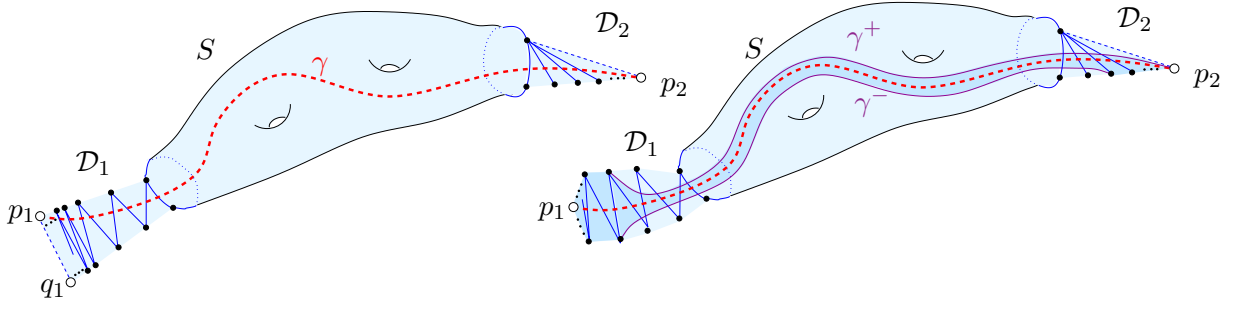,width=0.99\linewidth}
\caption{To the proof of Theorem~\ref{bad are all limit}.}
\label{belt}
\end{center}
\end{figure}

\medskip
\noindent
{\it  \underline{Step 2}: Arcs of $T'$ crossing  $\mathcal D_\gamma^T$.} 
Let  $\alpha\in T'$ be an arc not lying in $\bad_T T'$, i.e. $\alpha$  has finitely many crossings with the arcs of $T$. This means that  if $\alpha$  is crossing the domain  $\mathcal D_\gamma^T$ of $\gamma$, then the whole crossing $X=\alpha\cap \mathcal D_\gamma^T$ lies in some $S_i$ for a big enough $i$. 

From this we will show that the crossing $X$ ``stays far away'' from $\gamma$, by which we mean that there is a thin enough belt $\mathcal B_i$ along $\gamma$ containing no points of $X$. More precisely, if $X\subset S_i$, consider an arc $\gamma^+$ starting at any vertex of  $\mathcal D_1\setminus \{S_i\cup \{p_1\}\}$, then coming close to $\gamma$ inside  $\mathcal D_1$, then following $\gamma$ till almost the other endpoint $p_2$ and landing at any vertex of  $\mathcal D_2\setminus \{S_i\cup \{p_2\}\}$ (if $p_2$ is not an accumulation point of $\mathcal D_\gamma^T$, the arc $\gamma^+$ will have the second endpoint at $p_2$). If $\cald_2$ is an incoming fan, then $\gamma^+$ ends at the source of that fan.
Similarly, we construct an arc  $\gamma^-$ on the other side of $\gamma$ ($\gamma^-$ may coincide with $\gamma$).
Cutting $\mathcal D_\gamma^T$ along $\gamma^+$ and $\gamma^-$ we obtain a belt $\mathcal B_i$ along $\gamma$ (which is a disc with at most two accumulation points $p_1$ and $p_2$) and some finite connected or disconnected surface. Clearly, $\mathcal B_i$ is free of points of $X$. Notice also  that if the belts $\mathcal B_i$ are chosen as maximal possible inside $\mathcal D_\gamma^T\setminus S_i$,  then these belts are  nested: $\mathcal B_{i+1}\subset \mathcal B_i$ for each $i$.
In the case when $\cald_1=\cald_2$, the construction of the belt is very similar but  now  $\gamma^-$ coincides with $\gamma$.

\medskip
\noindent
{\it  \underline{Step 3}: Triangulation $\widetilde T$ of $\mathcal D_\gamma^T$.} 
Now, we will construct a triangulation  $\widetilde T$ of  the domain $\mathcal D_\gamma^T$ 
satisfying  $\bad_{\widetilde T} T'\subset \bad_{ T} T'$ and containing $\gamma$. 
We will choose the  triangulation  $\widetilde T$  as follows:
\begin{itemize}
\item in $\cals \setminus  \mathcal D_\gamma^T$ the new triangulation  $\widetilde T$ coincides with $T$;
\item $\gamma\in\widetilde T$;
\item $\widetilde T$ contains $\gamma^+$ and $\gamma^-$  for a belt $\mathcal B_i$ constructed in Step~2 for some choice of $S_i$, i.e. 
$\widetilde T$ contains $\gamma^+$ and $\gamma^-$, so that  $\mathcal D_\gamma^T\setminus \{\gamma^+ \cup \gamma^- \}$ is a union of a finite surface $S_0$ and a disc $\calb_i$ with at most two accumulation points (the endpoints of $\gamma$);
\item the finite surface $S_0$ is triangulated in any way;
\item each of the two connected components of $ \calb_i \setminus \gamma$ is a disc triangulated as in Fig.~\ref{constructing}, that is as 
\begin{itemize}
\item an elementary zig-zag (if the disc has two one-sided accumulation points);
\item an elementary incoming fan (if the disc has a unique one-sided accumulation point $p_i$ where $i\in \{1,2\}$), the source of this fan 
will be at the vertex closest to $p_j$, where $j\ne i$, $j\in \{1,2\}$.
\item any triangulation of a finite polygon (otherwise).

\end{itemize}

\end{itemize}

\begin{figure}[!h]
\begin{center}
\psfrag{Dj}{ $\mathcal D_j$}
\psfrag{g}{\color{red} $\gamma$}
\psfrag{pi}{ $p_i$}
\psfrag{pj}{ $p_j$}
\epsfig{file=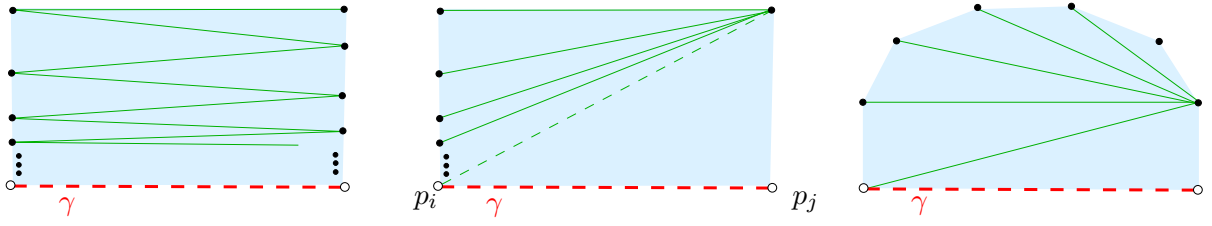,width=0.99\linewidth}
\caption{Constructing triangulation $\widetilde T$.}
\label{constructing}
\end{center}
\end{figure}

Notice that the triangulation $\widetilde T$ contains the arcs $\gamma_i^+$ and $\gamma_i^-$ cutting out the belt $\mathcal B_i$ for some $i\in\mathbb N$. Moreover, by construction it contains a belt lying inside any given belt $\mathcal B_j$, $j>i$.
This implies that for every arc $\alpha\in T'$ such that $\alpha\notin \bad_T T'$, some belt $\mathcal B_j$ is free of $\alpha$ and hence,
$\alpha\notin \bad_{\widetilde T} T'$  (here, we also use that the triangulation  $\widetilde T$ coincides with $T$ outside  $\mathcal D_\gamma^T$).   So, $\bad_T T'\subset \bad_{\widetilde T} T'$.  

Furthermore,  $\widetilde{T}$ has at most three bad arcs with respect to $T$: $\gamma$ and at most two limit arcs of incoming fans. Note that $\gamma$ is not necessarily a limit arc of $\widetilde T$.

\medskip
\noindent
{\it  \underline{Step 4}: Mutation from $T$ to $\widetilde T$.}   
Now, we will construct a mutation $\mu^{(0)}\circ \muone_1$ transforming $T$ to $\widetilde T$ inside $\cald^T_{\gamma}$.

The plan will be as follows.
We will label the non-limit arcs of $\widetilde T$ by natural numbers and will first apply a finite sequence of flips to obtain the first arc of $\widetilde T$, then the second, and so on. By this we will reconstruct all non-limit arcs of $\widetilde T$, and all the limit arcs of $\widetilde T$ will be included in the newly obtained triangulation as limits of existing arcs.

There are two difficulties with this plan:
\begin{itemize}
\item[(A)] the arc $\gamma$ may be a non-limit arc of $\widetilde T$, while $\gamma\in \bad_T \widetilde T$;
\item[(B)] the sequence of mutations described above may turn out to be not admissible 
(recall that a sequence of elementary mutations $\muone_1:=...\circ \mu_2\circ \mu_1$
 is admissible, if
 for every $\beta\in T$ there exist $n$ such that  $\mu_i\circ\dots \circ \mu_1(\beta)=\mu_n\circ\dots \circ \mu_1(\beta)$, for all $i>n$). 

\end{itemize}

To resolve the first difficulty, notice that if $\gamma$ is not a limit arc in $\widetilde T$ then the arc $\tau$ obtained by flipping $\gamma$ in $\widetilde T$ is not a bad arc with respect to $T$ (see Fig.~\ref{up side down}).  So, we can obtain $\tau$  after a finite mutation $\mu^{(0)}$. Then we will reconstruct one by one all other non-limit arcs of $\widetilde T$ (as they are not bad with respect to $T$). At the end we will flip $\tau$ back to $\gamma$ (i.e. the mutation transforming $T$ to $\widetilde T$ will be an infinite mutation followed by a single flip).


\begin{figure}[!h]
\begin{center}
\psfrag{Dj}{ $\mathcal D_j$}
\psfrag{g}{\color{red} $\gamma$}
\psfrag{pi}{ $p_i$}
\psfrag{pj}{ $p_j$}
\psfrag{t}{\color{Plum} $\tau$}
\epsfig{file=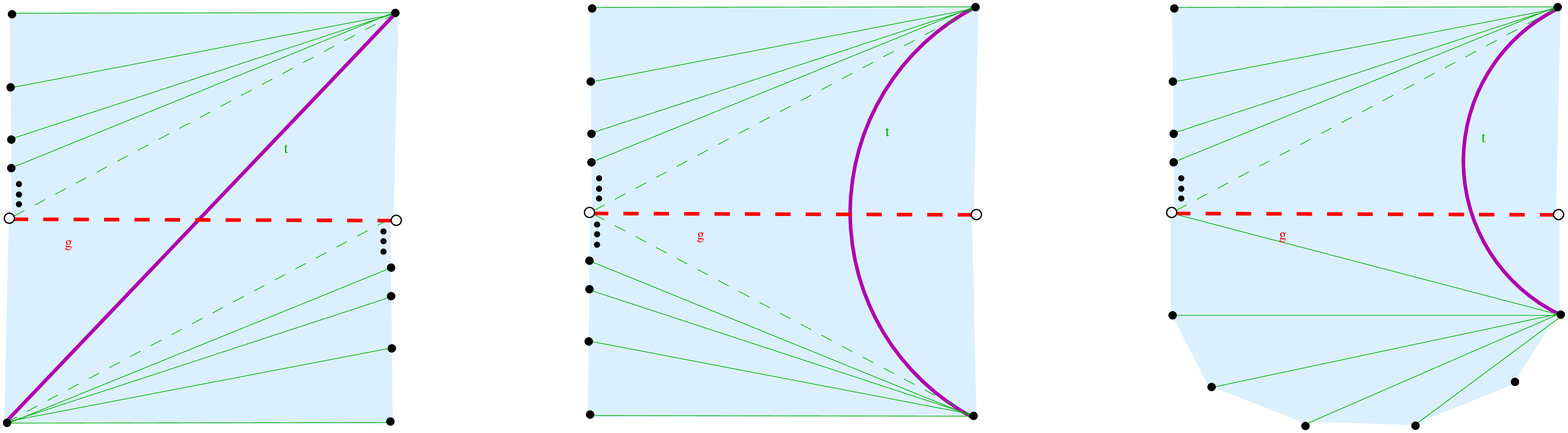,width=0.89\linewidth}
\caption{To Step~4.}
\label{up side down}
\end{center}
\end{figure}



To resolve the second difficulty, i.e. to make sure that the constructed sequence of mutations is admissible, we will be a bit more precise about the order of mutations. 
We start with  choosing some belt $\mathcal B_i$, $i\in \mathbb N$, bounded by arcs of $\widetilde T$. Then we apply a finite mutation $\mu^{(0)}_0$ to get all arcs of $\widetilde T$ lying in the finite part
$ \mathcal D_\gamma^T\setminus \mathcal B_i$  (this is possible as  $\widetilde T$ contains no bad arcs with respect to $T$ lying in $\cald^T_{\gamma}\setminus \calb_i$). Now, our aim is to find an infinite mutation inside the disc $\mathcal B_i$ which will transform  $\mu^{(0)}_0(T)$ to $\widetilde  T$.

To settle the triangulation inside $\mathcal B_i$ we 
 label the  arcs of $\mu^{(0)}_0(T)$ lying inside  $\mathcal B_i$ by  $\{\alpha_j\}$, $j\in \mathbb N$, we will also label   the non-limit arcs of  $\widetilde T$ lying inside  $\mathcal B_i$ by  $\{\beta_j \}$,  $j\in \mathbb N$. The arcs of $\mu^{(0)}(T)$ can be labelled in an arbitrary order. The arcs of $\widetilde T$ will be labelled
``from outside to inside'' so that each of the domains  $\mathcal B_i\setminus \{\beta_1,\dots,\beta_m\}$   remains connected (i.e. remains a disc) for every $m$. 

Given this labelling, we will first
find  a sequence of flips which will take the arc $\alpha_1$ to the arc $\beta_1$.
To achieve this, notice that there is a finite surface $S_i\in T$ containing $\beta_1$ (as $\beta_1$ is not bad with respect to $T$),
so one can expand $S_i$ to obtain a finite surface $\hat S_i\in T$ containing both  $\alpha_1$ and  $\beta_1$. Now, we apply Lemma~\ref{finite surface} to find a finite mutation $\mu^{(0)}_1$ which takes $\alpha_1$ to $\beta_1$ inside  $\hat S_i$.  
 Then we will find a finite sequence of flips which will take  $\mu^{(0)}_1(\alpha_2)$ to $\beta_2$ (here, we use that the arcs of $\widetilde T$ are enumerated so that  $\mathcal B_i\setminus  \{\beta_1\}$ is connected).
Similarly, we proceed for all other arcs inside $\calb_i$ (using that  $\mathcal B_i\setminus  \{\beta_1,\dots,\beta_m\}$ is connected).
The composition $\muone_1=...\circ\mu^{(0)}_2\circ\mu^{(0)}_1\circ\mu^{(0)}_0$ is an admissible infinite sequence of mutations
transforming $T$ to a triangulation which  differs from $\widetilde T$ at most by flipping $\gamma$ to $\delta$.
Composing the flip with $\muone_1$ we obtain a mutation $f_1\circ \muone_1$ taking $T$ to $\widetilde T$.

\medskip
\noindent
{\it  \underline{Step 5}: Repeating Steps~1--4 to remove all bad arcs.}   
The mutation $f_1\circ \muone_1$ constructed in Step~4 transforms $T$ to a triangulation $\widetilde T$ such that
 $\bad_T T'\subset \bad_{\widetilde T} T'$. As $\gamma\in\widetilde T$, this implies that $\bad_{\widetilde T} T'<\bad_T T'$. Also, notice that all bad arcs of $T'$ with respect to $\widetilde T$ are still limit arcs of $T'$. So, all assumptions of the theorem hold for the pair of triangulations $\widetilde T$ and $T'$ (but with smaller number of bad arcs). 
So, we can  apply Steps~1-4 again to find a mutation $f_2\circ\muone_2$ which will reduce the number of bad arcs.
Repeating this at most $k$ times we obtain a finite sequence of infinite mutations $\mun=f_k\circ\muone_k\circ\dots\circ f_1\circ\muone_1$ 
such that $\hat T=\mun(T)$ is a triangulation for which $\bad_{\hat T} T'=0$.

\medskip
\noindent
{\it  \underline{Step 6}: Mutation from $\hat T$ to $T'$.}   
Since $T'$ has no bad arcs with respect to $\hat T$,
by  Proposition~\ref{Prop:NoBadArc} there exists an infinite mutation $\muone_{k+1}$ transforming $\hat T$ to $T'$.
Composing $\muone_{k+1}$ with $\mun$ constructed in Step~5, we obtain  $\mu_{k+1}=\muone_{k+1}\circ \mun$  satisfying $\mu_{k+1}(T)=T'$. Notice, that a flip applied between two infinite mutations can be considered as a first elementary mutation in the next infinite sequence, so these flips do not affect the total number of infinite mutations we need to apply. 
Hence, we will be able to transform $T$ to $ T'$ in at most $k+1$ infinite mutations.
 
\end{proof}

\bigskip 

In a more general situation, i.e. in the presence of  non-limit bad arcs in $T'$, it turns out that the number of bad arcs does not provide the answer to the question whether there exists $\mun$ such that $\mun(T)=T'$ or not. In particular, $|\bad_T T'|<\infty$ does not imply existence of $\mun$ and   
 $|\bad_T T'|=\infty$ does not imply non-existence  of $\mun$ either (see Examples~\ref{ex:countrex1} and~\ref{ex:countrex2} below).

\begin{definition}[Stronger domain]
\label{def stronger domain} Given triangulations $T$ and $T'$ of a surface $\cals$ and an accumulation point $p$, we will say that the $\overline{p}$-domain $\mathcal D_{\overline p}^{}$  is {\it stronger} than the $\overline{p}$-domain $\mathcal D_{\overline p}'$
and write  $\mathcal D_{\overline p}^{} \succ \mathcal D_{\overline p}'$
 if one of the following holds:
\begin{itemize}
\item[(a)]   $\mathcal D_{\overline p}'$ is an outgoing fan while  $\mathcal D_{\overline p}^{}$ is not;
\item[(b)]   $\mathcal D_{\overline p}^{}$ is not an outgoing fan and $\mathcal D_{\overline p}^{}$ contains a limit arc which crosses infinitely many arcs of $\mathcal D_{\overline p}'$.

\end{itemize}
\end{definition}

\begin{remark}
An incoming fan domain always has a limit arc.
A zig-zag domain may  contain no limit arcs (when it is a zig-zag around an accumulation point $p$). In this case one can understand $p$ as the vanishing limit arc
and so Definition~\ref{def stronger domain}~(b) applies.
  
\end{remark}

\begin{remark}
Definition~\ref{def stronger domain} can be rephrased in the following way. For the pair   $(\mathcal D_{\overline p}^{},\mathcal D_{\overline p}')$
we can write whether its components are incoming fans, outgoing fans or zig-zags (denoting them $\mathsf{In},  \mathsf{Out},  \mathsf{Zig}$ respectively).
For example, $(\mathcal D_{\overline p}^{},\mathcal D_{\overline p}')\subset (\mathsf{In},\mathsf{Out})$ would mean that 
$\mathcal D_{\overline p}$ is an almost elementary incoming fan and $\mathcal D_{\overline p}'$ is an almost elementary outgoing fan.

In this notation,  $\mathcal D_{\overline p}^{} \succ \mathcal D_{\overline p}'$ if  $(\mathcal D_{\overline p}^{},\mathcal D_{\overline p}')$ is one of the following:
\begin{itemize}
\item[(a)]   $(\mathsf{In},\mathsf{Out})$ or  $(\mathsf{Zig},\mathsf{Out})$;
\item[(b)]   $(\mathsf{In},\mathsf{In})$, or  $(\mathsf{In},\mathsf{Zig})$, or  $(\mathsf{Zig},\mathsf{In})$, or  $(\mathsf{Zig},\mathsf{Zig})$
 and $\mathcal D_{\overline p}^{}$ contains a limit arc which crosses infinitely many arcs of $\mathcal D_{\overline p}'$.

\end{itemize}

\end{remark}

The next property follows immediately from the definition.

\begin{proposition}
Let $T,T',T''$ be three triangulations of $\cals.$ 
If  $\mathcal D_{\overline p}^{}\succ \mathcal D_{\overline p}'$ and  $\mathcal D_{\overline p}'\succ \mathcal D_{\overline p}''$
then 
 $\mathcal D_{\overline p}^{}\succ \mathcal D_{\overline p}''$.

\end{proposition}


%

\begin{theorem}
\label{no mun from stronger}
Let $T$ and $T'$ be triangulations of $\cals$. 
Suppose that $\mathcal D_{\overline p}^{}\succ \mathcal D_{\overline p}'$ for some one-sided limit at an accumulation point $p$. Then there is no finite sequence of infinite mutations $\mun$ transforming $T$ to $T'$.

\end{theorem}

\begin{proof}
If  $\mathcal D_{\overline p}^{}\succ \mathcal D_{\overline p}'$ then $|\bad_T T'|=\infty$ (every arc of  $\mathcal D_{\overline p}'$ crossing the limit arc $\alpha$ of   $\mathcal D_{\overline p}$ lies in  $\bad_T T'$ by Observation~\ref{crossing limit arc}.a).
By Proposition~\ref{prop:|bad|=>more than one} this implies that if  $\mathcal D_{\overline p}^{}\succ \mathcal D_{\overline p}'$
then $T$ cannot be transformed to $T'$ in one infinite mutation.

In particular, if  $\mathcal D_{\overline p}^{}$ is not an outgoing fan, then it will not turn into an outgoing fan after applying one infinite mutation, so it will not turn into an outgoing fan after two infinite mutations, and so on: it cannot become an outgoing fan after finitely many infinite mutations. This proves the proposition for pairs $\mathcal D_{\overline p}^{}\succ \mathcal D_{\overline p}'$ coming from  Case (a)  
of Definition~\ref{def stronger domain}.

For the pairs $\mathcal D_{\overline p}^{}\succ \mathcal D_{\overline p}'$ coming from  Case (b), notice that 
 applying an infinite mutation $\muone$ either leaves the limit arc of  $\mathcal D_{\overline p}^{}$ unchanged or makes the domain stronger 
(otherwise, we get a contradiction to the first paragraph of the proof).
This means that after a finite sequence of infinite mutations the domain  $\mathcal D_{\overline p}^{}$ turns into a domain $\mathcal D$ which either has the same limit arc or is stronger than  $\mathcal D_{\overline p}^{}$. Since 
$\mathcal D_{\overline p}^{}\succ \mathcal D_{\overline p}'$, this implies that $T'$ cannot be obtained from $T$ by a finite sequence of infinite mutations.

\end{proof}

Theorem~\ref{no mun from stronger} gives a practical tool for deciding whether a triangulation $T$ can be transformed into a triangulation $T'$ in finitely many infinite mutations (in particular, we use it in Examples~\ref{ex:countrex1} and~\ref{ex infinity-gon}).
At the same time, it still cannot be turned into  an ``if and only if'' condition: in Example~\ref{ex:countrex1} we present a triangulation $T$ which cannot be transformed into $T'$  in finitely many infinite mutations, however, one cannot show it by immediate application of  Theorem~\ref{no mun from stronger}.

\bigskip

\begin{remark}
It is shown in~\cite{BG} that mutations along {\it admissible} sequences induce a preorder on the triangulations of a given surface (where $T<T'$ when there exists an admissible sequence of elementary mutations $\mu$ such that $\mu(T)=T'$). 
In our settings of infinite mutations completed with all limit arcs, this property obviously holds for the relation $<_n $  (where $T<_nT'$ if there exists a finite  sequence of infinite mutations $\mun$ such that $\mun(T)=T'$). On the other hand, the relation $<_1$ (where $T<_1T'$ if there exists an infinite mutation $\muone$ such that $\muone(T)=T'$) does not induce a preorder as a composition of two infinite mutations is not necessarily an infinite mutation.

The relation $<_n$ clearly defines a preorder on all triangulations of $\cals$ (as $T<_nT$ for every $T$ and $T_1<_nT_2 <_n T_3$ implies $T_1 <_{n} T_3$).
In the following proposition we will see that for every infinite surface there is a minimal element with respect to the preorder defined by the relation $<_n$.

\end{remark} 

\begin{proposition}
\label{distinguished triangulation}
For every infinite surface $\mathcal S$, there exists a triangulation $T$ such that 
\begin{itemize}
\item[-] for any arc $\gamma\in\cals,$ holds $|\gamma\cap T|<\infty$,  and
\item[-] for any triangulation $T'$ of $\cals$  there exists an infinite mutation $\mu^{(1)} $  satisfying   $T'=\mu^{(1)}(T)$.
\end{itemize}
\end{proposition}

\begin{proof} For each accumulation point $p_i,$ choose a small disc  neighbourhood $D_i$ so that these neighbourhoods do not intersect. Inside each $D_i,$ set an outgoing fan triangulation (see Fig.~\ref{fig: best triang}). Choose any triangulation on the rest of the surface (the surface $\mathcal S\setminus \cup D_i$ has finitely many boundary marked points, so it has a finite triangulation).

Denote the triangulation constructed above by $T$ and observe that any arc $\gamma$ in $\mathcal S$ crosses $T$ only finitely many times (indeed, $\gamma$ crosses at most finitely many arcs from each $D_i$). Thus, for any triangulation $T'$, we have $|\bad_T T'|=0$ which in view of Proposition~\ref{Prop:NoBadArc} proves the result.
\end{proof}

\begin{figure}[!h]
\begin{center}
\psfrag{g}{{\color{red} $\gamma$}}
\epsfig{file=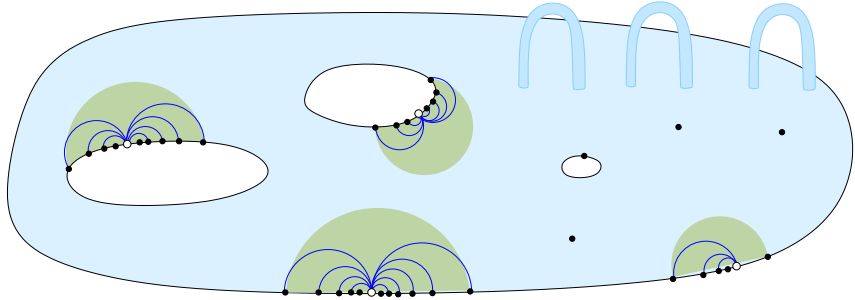,width=0.7\linewidth}
\caption{An example of an outgoing fan triangulation: neighbourhood of  each accumulation point is triangulated as a union of at most two outgoing fans, and the rest of the surface may be triangulated by finitely many  arcs.}
\label{fig: best triang}
\end{center}
\end{figure}

\begin{remark}
Proposition~\ref{distinguished triangulation} states that among all the triangulations there is a class of particularly good ones composed of almost elementary outgoing fans. We will call them  {\it outgoing fan triangulations} (see Definition~\ref{Def:FanTri}).

In Section~\ref{Sec:ClusAlg}, we will  introduce cluster algebras associated with infinite surfaces, and we will see that for an outgoing fan triangulation every cluster variable can be expressed as a Laurent polynomial in the initial cluster, see also Remark~\ref{rem: outgoing fan}. 
\end{remark}

\subsection{Examples} In this section we collect examples illustrating the statements from Section~\ref{sec hierarchy}.

\begin{example}[$|\bad_T T'|=\infty$ does not imply $\mun(T)\ne T'$]%
\label{ex:countrex2}
One could expect that if $|\bad_T T'|=\infty$ for triangulations $T$ and $T'$ of a surface, then there is no finite sequence of infinite mutations transforming $T$ to $T'$. This is not true, as illustrated in Fig.~\ref{countrex2}  by triangulations $T$ and $T'$ such that  $|\bad_T T'|=\infty$,  $|\bad_{T'} T|=\infty$, but $T'=\muone_2\circ\muone_1(T)$ where both $\muone_1$ and $\muone_2$ are shifts of sources of some fans. 

\end{example}

\begin{figure}[!h]
\begin{center}
\psfrag{T}{\color{blue} $T$}
\psfrag{T'}{\color{red} $T'$}
\psfrag{m1}{ $\muone_1$}
\psfrag{m2}{ $\muone_2$}
\epsfig{file=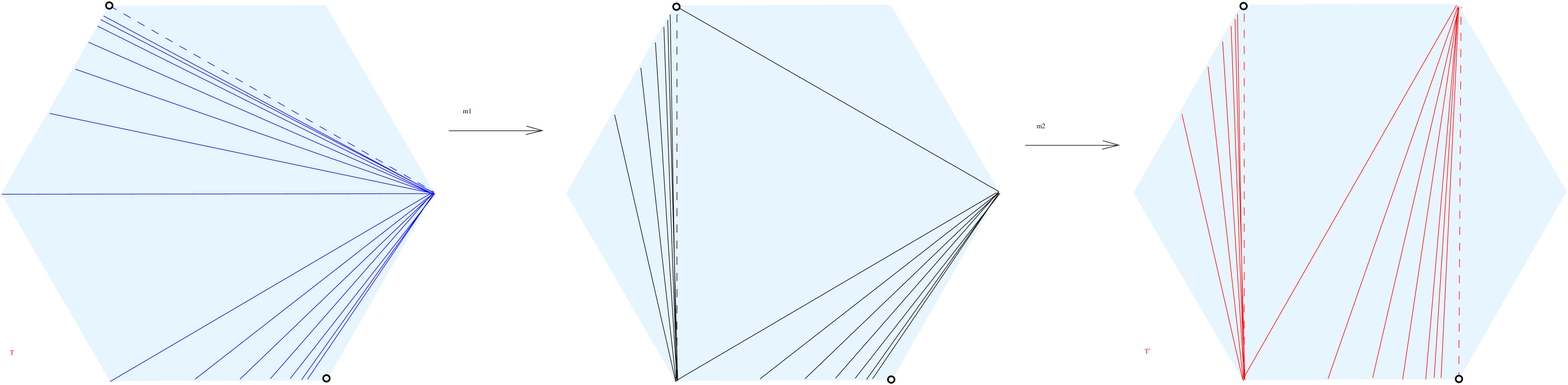,width=0.99\linewidth}
\caption{Triangulations $T$ and $T'$: $|\bad_T T'|=\infty$ and  $T'=\muone_2\circ\muone_1(T)$, see Example~\ref{ex:countrex2}.}
\label{countrex2}
\end{center}
\end{figure}

Theorem~\ref{infinite sequences} above shows that for some triangulations $T$ and $T'$ of the same surface one really needs infinite sequences of infinite mutations to transform $T$ to $T'$. 
The proof  is based on two triangulations such that $|\bad_T T'|=\infty$.
One could hope that the condition  $|\bad_T T'|<\infty$ would imply that there is a finite sequence of infinite mutations transforming $T$ to $T'$. However, the following example shows that it is not always the case.

\begin{example}[$|\bad_T T'|<\infty$ does not imply $\mun(T)=T'$]
\label{ex:countrex1}
Consider the triangulations $T$ and $T'$ shown in Fig.~\ref{countrex1}.
Notice that the set  $|\bad_T T'|$ consists of three arcs of $T'$:  two horizontal limit arcs of $T'$   and the diagonal passing through the centre of the octagon. 
We will show that it is  not possible to obtain $T'$ from  $T$  by applying finitely many infinite mutations.
The idea is very similar to the one in the proof of Theorem~\ref{infinite sequences}.

\begin{proof}
{\it Step 1: no $\muone$ transforms $T$ to $T'$.}
We will start by showing that $T$ cannot be transformed to $T'$ in one infinite mutation $\muone$.
Suppose that $\muone(T)=T'$ for an admissible infinite mutation $\muone$. 
Then every non-limit arc of $T'$  is obtained after finitely many elementary mutations (while the limit arcs can also arise in the process of completion). Consider the diagonal arc $\gamma\in T'$ passing through the centre of the octagon: $\gamma$ crosses infinitely many arcs of $T$, so it cannot arise after finitely many elementary mutations, however, $\gamma$ is not a limit arc of $T'$. This contradicts the existence of $\muone$. 

\medskip

\noindent
{\it Step 2: no $\mun$ transforms $T$ to $T'$.}
Now, suppose that there exists a finite sequence of infinite mutations $\mun$, so that $\mun(T)=T'$. Denote 
$\mun=\muone_n\circ\dots\circ\muone_1$ and $T_i:=\muone_i\circ\dots\circ\muone_1(T)$. We will also assume $T_0:=T$.
Notice that as $T'=T_n$ is a triangulation, each of $T_i$ should be a triangulation.

Let $\overline{p}$ be one of the four accumulating points $A,B,C,D$  of the octagon.
Consider the types of $\overline{p}$-domains $\mathcal D_{\overline{p}}^{(i)}$  of the triangulation $T_i$ at the one-sided limit accumulating to $\overline{p}$. Notice that  for every choice of $\overline{p}\in \{A,B,C,D\}$ the domain  $\mathcal D_{\overline{p}}^{(0)}$ is a zig-zag with the limit arc lying on the boundary of the octagon (as $T_0=T$). 

Let $k$ be  the smallest number $0<k\le n$ such that at least one of the four $\overline{p}$-domains  $\mathcal D_{\overline{p}}^{(k)}$ is not a zig-zag with the limit arc on the boundary. We will now look at the type of this $\overline p$-domain:
\begin{itemize}
\item 
 This domain  $\mathcal D_{\overline{p}}^{(k)}$ cannot be an outgoing fan (this follows from  Theorem~\ref{no mun from stronger}, but also one can directly notice that in this case $|\bad_{T_{k-1}} T_k|=\infty$ which is a contradiction to  $\muone_k(T_{k-1})=T_k$
 by  Proposition~\ref{prop:|bad|=>more than one}). 
\item
Suppose,  $\mathcal D_{\overline{p}}^{(k)}$ is an incoming fan. 
Then  for at least one other accumulation point $q\in \{A,B,C,D\}$, $q\ne p$, the $\overline q$-domain  is not a zig-zag, and hence, is an incoming fan (clearly, $q$ is one of the two neighbours of $p$ in  $ABCD$). Suppose that the limit arc of the incoming fan $\mathcal D_{\overline p}^{(k)}$ lies outside the quadrilateral $ABCD$. Then   $\mathcal D_{\overline p}^{(k)}\succ  \mathcal D_{\overline p}'$, which contradicts the assumption that $T'$ can be obtained from $T_k$ by finitely many infinite mutations (see  Theorem~\ref{no mun from stronger}). 
Hence, the limit arc of the incoming fan $\mathcal D_{\overline p}^{(k)}$ starts at $p$ and passes through the interior of $ABCD$ (it may land at $A, B, C$, or $D$, or cross through a side of $ABCD$).
Similarly, the limit arc of the incoming fan $\mathcal D_{\overline q}^{(k)}$ starts at $p$ and passes through the interior of $ABCD$.
These limit arcs are only compatible in one triangulation if one of them is a diagonal of $ABCD$  and another is a side. However, this would imply that the domains at the other two points $s,t\in \{A,B,C,D \}\setminus \{p,q \}$ are not zig-zags, but  incoming fans. The above reasoning shows that the limit arcs of these two new fans are also a diagonal and a side of $ABCD$,
however, all limit arcs of the four  fans will not be compatible.
The contradiction shows that   $\mathcal D_{\overline{p}}^{(k)}$ is not an incoming fan.

\item
So,  $\mathcal D_{\overline{p}}^{(k)}$ is a zig-zag with a limit arc not lying on the boundary of the octagon. Now, consider the other $\overline{q}$-domains $\mathcal D_{\overline{p}}^{(k)}$ for $\overline{q}\ne\overline{p}$,  $\overline{q}\in \{A,B,C,D\}$. None of them is a zig-zag with the limit arc on the boundary as such a domain is not compatible with the zig-zag  $\mathcal D_{\overline{p}}^{(k)}$. So, by a similar argument as above, all of  $\mathcal D_{\overline{q}}^{(k)}$ are also zig-zags with limit arcs not lying on the boundary of the octagon.
This means that $T_{k-1}$ looks similar to $T_0=T$ while $T_k$ looks similar to $T_n=T'$, thus the argument in Step~1 shows that $T_{k}=\muone_k(T_{k-1})$ is impossible.  

\end{itemize}

\end{proof}

\end{example}

\begin{figure}[!h]
\begin{center}
\psfrag{s}{ $s$}
\psfrag{p}{ $p$}
\psfrag{g}{\color{red} $\gamma$}
\psfrag{d}{\color{red} $\delta$}
\psfrag{T}{\color{blue} $T$}
\psfrag{T'}{\color{red} $T'$}
\psfrag{TT'}{$T\cap T'$}
\epsfig{file=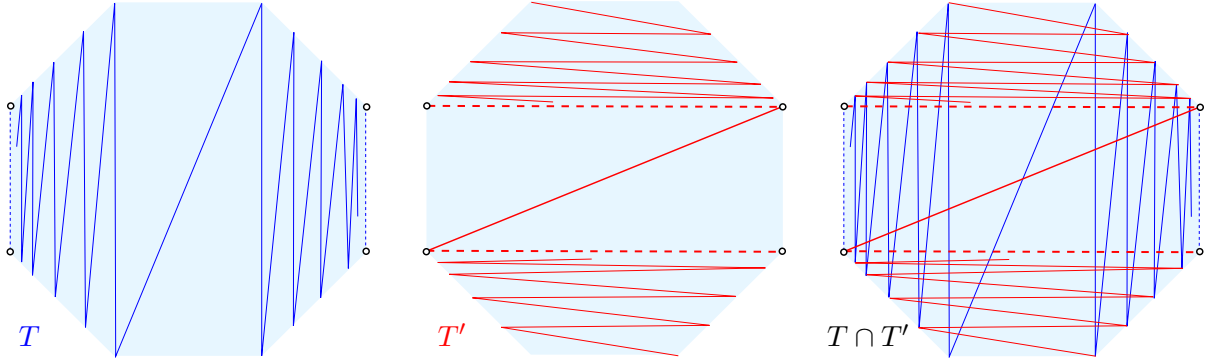,width=0.99\linewidth}
\caption{Triangulations $T$ and $T'$ with $|\bad_T T'|=3$ and no mutation $\mun$ transforming $T$ to $T'$, see Example~\ref{ex:countrex1}.}
\label{countrex1}
\end{center}
\end{figure}

\begin{example}[Infinity-gon]
\label{ex infinity-gon}
In Fig.~\ref{Fig:one-sided Infinity-gon} and \ref{Fig:two-sided Infinity-gon}, we illustrate the relations between various triangulations  
of the one-sided infinity-gon $I_1$ and the two-sided infinity-gon $I_2$, respectively, where
\begin{itemize}
\item[-] solid arrows indicate finite sequences of infinite mutations,
\item[-] dashed arrows indicate infinite sequences of infinite mutations.
\end{itemize}
Moreover, these figures can be considered as ``underlying exchange graphs'' for $I_1$ and $I_2$:
\begin{itemize}
\item[-] vertices of the graphs correspond to classes of triangulations of $I_1$ and $I_2,$ respectively, 
where classes are composed of triangulations having similar combinatorics, 
which roughly speaking means that two triangulations are in the same class if they have the same set of almost elementary domains attached to each other in the same way. 
\end{itemize}
We use direct constructions (as in Fig.~\ref{Fig:switch(outgoing-->incoming)},~\ref{Fig:shiftsource} and~\ref{Fig:switch(zigzag-->fan}) to show existence of  finite sequences of infinite mutations (where they do exist) and
 Theorem~\ref{no mun from stronger} to prove non-existence of finite sequences (otherwise).
The graphs shown in  Fig.~\ref{Fig:one-sided Infinity-gon}~and~\ref{Fig:two-sided Infinity-gon} agree with those presented in Fig.~8 and~9 of~\cite{BG}.

\begin{figure}[!h]
\begin{center}
\psfrag{j>i}{\scriptsize $j>i$}
\psfrag{i}{\tiny $i$}
\psfrag{i1}{\tiny $i+1$}
\psfrag{i2}{\tiny $i+2$}
\psfrag{i3}{\tiny $i+3$}
\psfrag{i4}{\tiny $i+4$}
\psfrag{j}{\tiny $j$}
\psfrag{j1}{\tiny $j+1$}
\psfrag{j2}{\tiny $j+2$}
\psfrag{j3}{\tiny $j+3$}
\epsfig{file=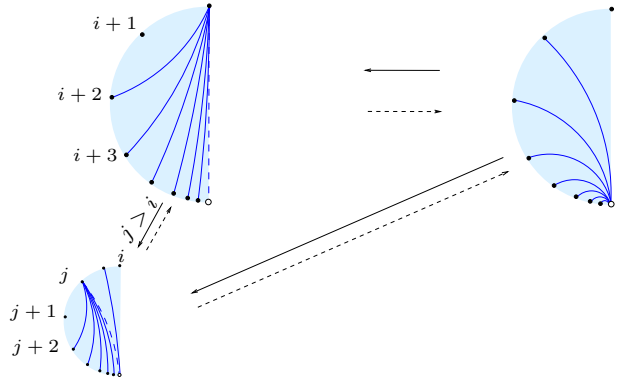,width=0.5\linewidth}
\caption{``Underlying exchange graph'' for the one-sided infinity-gon.}
\label{Fig:one-sided Infinity-gon}
\end{center}
\end{figure}

\begin{figure}[!h]
\begin{center}
\psfrag{i>j}{\tiny $i>j$}
\psfrag{i'<=i}{\tiny  $i'\le i$}
\psfrag{i}{\tiny $i$}
\psfrag{j}{\tiny $j$}
\psfrag{i'}{\tiny $i'$}
\psfrag{j'}{\tiny $j'$}
\psfrag{i+k}{\tiny $i+k$}
\psfrag{i'+k'}{\tiny $i'+k'$}
\psfrag{i'+k'>=i+k}{\tiny $i'+k'\ge i+k$}
\psfrag{k>=0}{\tiny $k\ge 0$}
\psfrag{k'>=0}{\tiny $k'\ge 0$}
\psfrag{j+k}{\tiny $j+k$}
\psfrag{j-k}{\tiny $j-k$}
\psfrag{j'>j}{\tiny $j'>j$}
\psfrag{j'<j}{\tiny $j'<j$}
\psfrag{i+k>=j}{\tiny $i+k\ge j$}
\psfrag{n}{\tiny $n$}
\psfrag{i<=j}{\tiny $i\le j$}
\epsfig{file=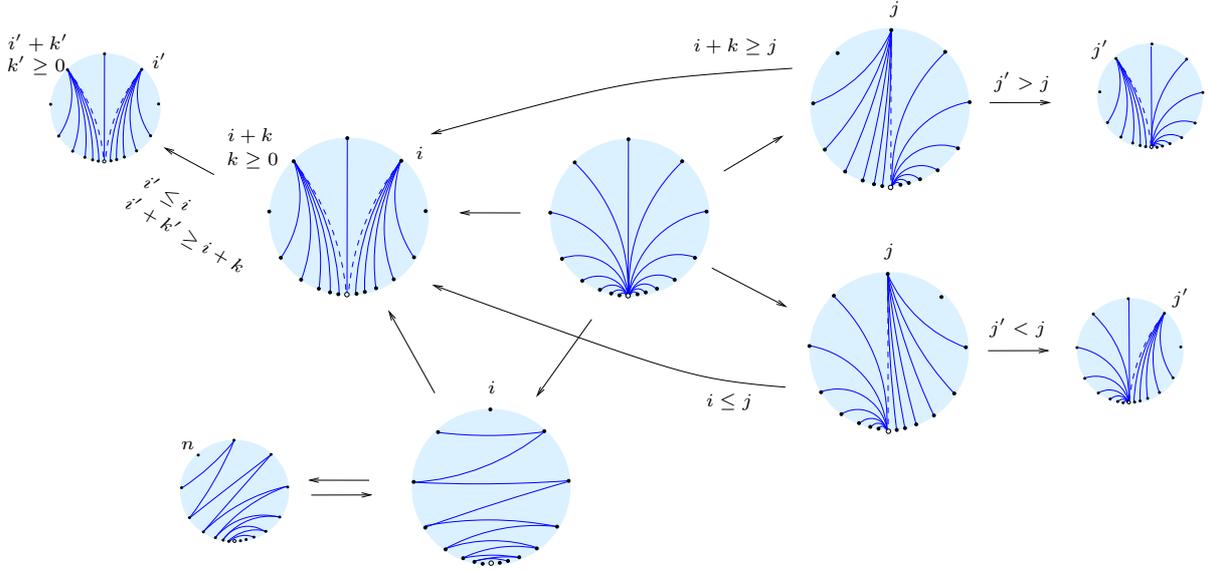,width=0.99\linewidth}
\caption{``Underlying exchange graph'' for the two-sided infinity-gon. We omit dashed arrows representing infinite sequences of infinite mutations, but they are assumed between the states where no solid arrow is shown.}
\label{Fig:two-sided Infinity-gon}
\end{center}
\end{figure}
\end{example}

\begin{remark}
In Fig.~\ref{Fig:one-sided Infinity-gon} and~\ref{Fig:two-sided Infinity-gon} one can find examples of outgoing fan triangulations for the infinity-gon: these are those triangulations shown in the middle of the figure and connected to every other triangulation by a solid arrow. 
\end{remark}
\section{Infinite rank surface cluster algebras} \label{Sec:ClusAlg}

In this section we will use hyperbolic structures on infinite surfaces to define infinite rank surface cluster algebras.
The spirit of this section follows the one of  Fomin and Thurston~\cite{FT}.

\subsection{Hyperbolic structure and converging horocycles}

Consider an infinite surface $\cals$ with a triangulation $T$. We will understand the triangles of $T$ as ideal hyperbolic triangles, i.e. hyperbolic triangles with all three vertices on the boundary of the hyperbolic plane. To each marked point $p_i$, we will assign a horocycle $h_i$  with an additional requirement that if marked points $p_j$, $j\in\mathbb{N},$ converge to an accumulation point $p_*$ then the corresponding horocycles converge to the horocycle $h_*$ at $p_*$ (see Definition~\ref{converging horocycles}). 

Similarly to the case of finite surfaces, hyperbolic triangles together with a choice of horocycles induce a lambda length assigned to each arc of the triangulation.  
We will show in  Section~\ref{compatibility} that these lambda lengths satisfy certain conditions (as long as the horocycles are converging). 
Conversely, in Section~\ref{data->hyperbolic structure} we will see that given a  set of positive numbers associated with  the arcs
of $T$ and  satisfying the above-mentioned conditions, 
there exists a hyperbolic surface with a choice of converging horocycles such that the lambda lengths of the arcs of $T$ coincide with the prescribed numbers,  see Theorem~\ref{Thm:CompatibilityWhole}. 

\subsubsection{Admissible hyperbolic structures}

\begin{definition}[Admissible hyperbolic structure]
By an {\it admissible hyperbolic structure} on an infinite surface $\cals$ 
we mean a locally hyperbolic metric such that
\begin{itemize} 
\item[-] all boundary arcs are complete (i.e. bi-infinite) geodesics, and
\item[-] all interior marked points are  cusps.

\end{itemize}
A surface with admissible hyperbolic structure will be called an {\it admissible hyperbolic surface}.       

\end{definition}

This means that having all vertices of triangles at interior or boundary marked points of a triangulation of $\mathcal S$, every triangle is isometric to a hyperbolic triangle with all vertices at the boundary of the hyperbolic plane.

\begin{lemma}
Let $\mathcal S$ be an infinite surface with an admissible hyperbolic structure.
Then the universal cover of $\mathcal S$ is contained in the hyperbolic plane $\bH^2$ as a proper subset.

\end{lemma}

\begin{proof}
We start by choosing an outgoing fan triangulation $T$ of $\cals$ as in Proposition~\ref{distinguished triangulation}.
Let $t_1$ be a triangle in $T$. We can embed $t_1$ isometrically into the hyperbolic plane $\bH^2$ (uniquely up to hyperbolic isometry) and we denote the obtained triangle $\tilde t_1$. The triangle $t_1$ has at most 3 adjacent triangles $t_2,t_3,t_4$ in $\mathcal S$ (some of them may coincide with others or with $t_1$). The gluing of two adjacent triangles is determined by the choice of hyperbolic structure on $\mathcal S$ and it may be described by shear coordinates as in Fig~\ref{fig: shear} (a): we drop the perpendiculars from the vertices of the adjacent triangles to the common side and look at the hyperbolic distance between the feet of the two perpendiculars. This dictates how the triangles   $t_2,t_3,t_4$ will lift to the universal cover. We add the triangles to the universal cover one by one. Notice that we will  always  be able to embed the next triangle $\tilde t_{n+1}$ (attached to the previous triangle $\tilde t_{n}$ by the common side $\gamma$, see  Fig~\ref{fig: shear} (b))  into the hyperbolic plane without intersecting the previous ones. Indeed, the finite union of already placed triangles $\tilde t_1, \dots, \tilde t_n$ lies on one side from the line $\gamma=\tilde t_{n}\cap\tilde t_{n+1}$, and the  new triangle $\tilde t_{n+1}$ lies on the other side.
Lifting adjacent triangles of each triangle $t_i\in T$  we will reach every triangle of $ T$ in finitely many steps:
indeed $T$ was chosen as an outgoing fan triangulation, in particular, it contains no limit arcs.
The fact that the boundary of $\cals$ is non-empty implies
 that the lifts will not cover the whole hyperbolic plane.

\begin{figure}[!h]
\begin{center}
\psfrag{g}{{\color{red} $\gamma$}}
\psfrag{a}{\small (a)}
\psfrag{b}{\small (b)}
\psfrag{t}{\scriptsize $\tilde t_{n}$}
\psfrag{t1}{\scriptsize $\tilde t_{n+1}$} 
\epsfig{file=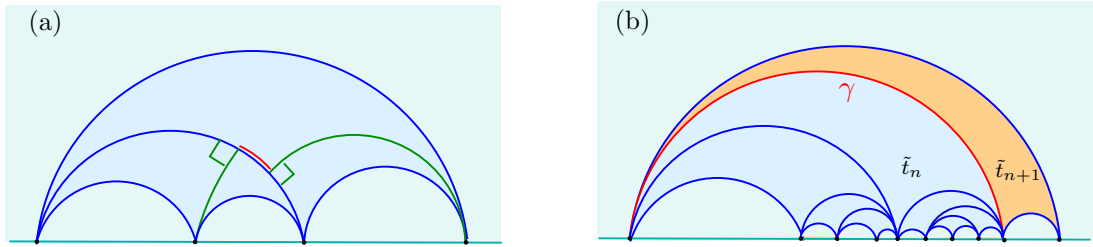,width=0.9\linewidth}
\caption{Lift of $\mathcal S$ to the universal cover: (a) gluing of  adjacent triangles (shear coordinates); (b) attaching the next triangle.}
\label{fig: shear}
\end{center}
\end{figure}

\end{proof}

In this section, we will often consider the surface $\mathcal S$ as lifted to the hyperbolic plane $\bH^2$. More precisely, we will often lift $\mathcal S$ to the {\it upper half-plane model} $\mathcal H$ of the hyperbolic plane $\bH^2$,
i.e. to the set of points $\{z\in \bC \mid    \textup{Im }z>0  \}$ with hyperbolic distance defined by $$\cosh d(z_1,z_2)=1+ \frac{|z_1-z_2|}{2 \text{Im}(z_1) \text{Im}(z_2)}. $$ In this model the lines are represented by half-lines and half-circles orthogonal to $\partial \mathcal H$. The cusps, as well as the boundary marked points, lift to the points on the boundary $\partial \mathcal H$. 

\subsubsection{Converging horocycles}
A horocycle centred at a cusp $p$ is a curve perpendicular to every geodesic incident to $p$. 
Lifted to the universal cover 
$\mathcal H$, a horocycle centred at $p\ne \infty$ is represented by a circle tangent to $\partial \mathcal H$ at $p$.
A horocycle centred at $p=\infty$ is represented by a horizontal line $\textup{Im } z=const$. A horocycle centred at $p$ may be understood as a set of points ``on the same (infinite) distance from $p$'', and it is also a limit of a sequence of circles through a fixed point with the centres converging to $p$. 

\begin{figure}[!h]
\begin{center}
\epsfig{file=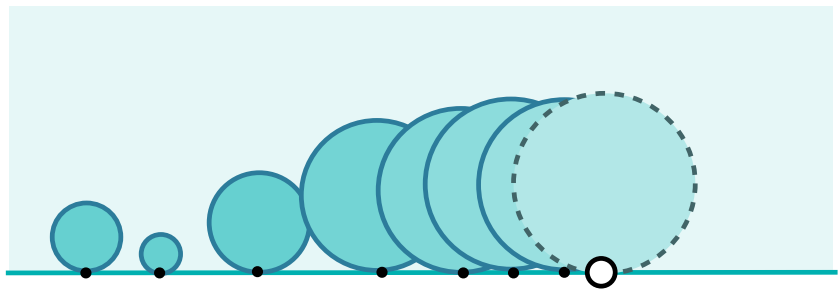,width=0.7\linewidth}
\caption{Converging horocycles in the upper half-plane.}
\label{fig: converging horocycles}
\end{center}
\end{figure}

\begin{definition}[Converging horocycles in $\mathcal H$]
\label{converging horocycles} 
Let $\mathcal H$ be the upper half-plane model of the hyperbolic plane and let $p_i\in \partial \mathcal H$, $i\in \bN,$ be a sequence of points on the boundary. Suppose that the points $p_i$ converge to  some point $p_*$, denote it by $p_i\to p_*$.  Let $h_i$ be a horocycle centred at $p_i$ and let $h_*$ be a horocycle centred at $p_*$. We say that the horocycles $h_i$ {\it converge} to $h_*$ (and denote it by $h_i\to h_*$) if the Euclidean circles representing $h_i$ converge to the Euclidean circle representing $h_*$.  

\end{definition}

\begin{remark}
\label{rem converging}
(a) If a point $p_*$ in Definition~\ref{converging horocycles} is represented  by $\infty$ in $\mathcal H$, we  apply an isometry of the hyperbolic plane (i.e. any linear-fractional transformation $f(z)=\frac{az+b}{cz+d}$ with $a,b,c,d\in \bR$, $ad-bc> 0$) which takes $\infty$ to another point of $\partial \mathcal H$  and then we use the definition above. 

(b) Definition~\ref{converging horocycles}  refers to the Euclidean shapes in the upper half-plane model. So, we need to check the convergence of horocycles (i.e. the property that they can be represented 
 by a converging sequence of Euclidean circles) does not change under hyperbolic isometries. This follows since the property can  also be reformulated in purely hyperbolic terms (see  part (c) below).

(c) Hyperbolic interpretation of convergence of horocycles: the horocycles $h_i$ at $p_i$ {\it converge} to the horocycle $h_*$ at $p_*$ if they pointwise converge in the hyperbolic plane. More precisely, for every point $q\in \partial \mathcal H$, let $\xi_i(q)=qp_i \cap h_i$ be the point of intersection of the horocycle $h_i$ and  the hyperbolic line $qp_i$ connecting $q$ to $p_i$. Then $\xi_i(q)$ converges to $\xi_*(q)=qp_*\cap h_*$ as $p_i\to p_*$.

\end{remark}

Part (c) of Remark~\ref{rem converging} allows us to introduce   the notion of converging horocycles in any infinite hyperbolic surface.

\begin{definition}[Converging horocycles in $\cals$, surface with converging horocycles]
Let $\mathcal S$ be an infinite surface with an admissible hyperbolic metric. 
Fix a horocycle at every marked point.
\begin{itemize}
\item[(a)]
Let $p_i$ be a sequence of boundary marked points converging to an accumulation point $p_*$.
We say that the horocycles $h_i$ at $p_i$  {\it converge} to the horocycle  $h_*$ at $p_*$ if there are lifts of $h_i$ to the universal cover converging to a lift of $h_*$. 

\item[(b)]
We say that the horocycles in $\mathcal S$  {\it converge} if for any converging sequence $p_i$ of  boundary marked points the horocycles $h_i$ at $p_i$ converge to the horocycle $h_*$ at the accumulation point $p_*$, i.e. $h_i\to h_*$ as $i\to \infty$.
\end{itemize}
\end{definition}

Similarly to the case of finite surfaces, we will use horocycles to define lambda lengths of arcs and boundary segments.

\begin{definition}(Lambda length of an arc,~\cite{P})
Given an admissible hyperbolic structure in $\mathcal S$ together with a choice of converging horocycles, consider an arc or boundary arc $\gamma$  in $\cals$. The horocycles at the both endpoints of $\gamma$ define the {\it signed length} $l(\gamma)$ of $\gamma$,
where $l(\gamma)$ is negative if the horocycles overlap, zero if they are tangent and positive if they  do not intersect. The {\it lambda length} $\lambda(\gamma)$ is defined by
$$\lambda(\gamma)=e^{\frac{l(\gamma)}{2}}.$$
\end{definition}

\begin{remark}
To be able to work with punctured surfaces, the authors of~\cite{FST} introduced the notion of tagged arcs and tagged triangulations. \cite{FT} gives the notion of the lambda length of a tagged arc using the idea of conjugate horocycles. We will follow exactly the same definitions without presenting them here.   
\end{remark}

\begin{remark}
(Ptolemy relation in a quadrilateral)
Let $\mathcal S$ be a finite hyperbolic surface with a choice of horocycles at all punctures and boundary marked points. Let $pqrs$ be a quadrilateral in $\mathcal S$ (where each of $p,q,r,s$ is a  puncture or a boundary marked point), see Fig.~\ref{fig: ptolemy}.  It is shown in~\cite{P} that the lambda lengths of the arcs connecting these marked points satisfy the {\it Ptolemy relation}
$$\lambda_{pr}\lambda_{qs}=\lambda_{pq}\lambda_{rs}+\lambda_{qr}\lambda_{ps}, 
$$
where $\lambda_{ij}$ is the lambda length of the arc connecting $i$ to $j$ for $i,j\in \{p,q,r,s\}$. 

In the case of an infinite surface, the same relation obviously holds, as all elements of the relation are just a part of one quadrilateral.
\end{remark}

\begin{figure}[!h]
\begin{center}
\psfrag{p}{$p$}
\psfrag{q}{$q$}
\psfrag{r}{$r$}
\psfrag{s}{$s$}
\epsfig{file=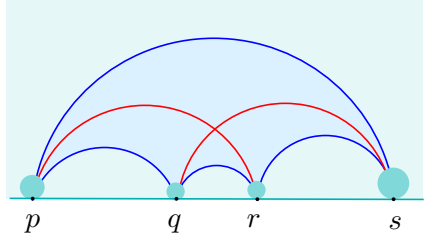,width=0.35\linewidth}
\caption{Ptolemy relation in a quadrilateral: $\lambda_{pr}\lambda_{qs}=\lambda_{pq}\lambda_{rs}+\lambda_{qr}\lambda_{ps}.$}
\label{fig: ptolemy}
\end{center}
\end{figure}

\subsubsection{Decorated Teichm\"uller space}
\begin{definition}[Decorated Teichm\"uller space]
Let $\cals$ be an infinite surface. 
A point in the {\it decorated Teichm\"uller space  $\widetilde{ \mathcal T}(\mathcal S)$ of $\mathcal S$}
is an admissible hyperbolic structure on $\mathcal S$ together with a choice of converging horocycles.

\end{definition}

It is shown in~\cite{FT} that given a triangulation $T$  of a finite hyperbolic surface $\mathcal S$, lambda lengths of  the arcs of $T$ uniquely define a point of the decorated Teichm\"uller space $\widetilde{\mathcal T}(\mathcal S)$.
We will obtain a similar result for  triangulations containing no infinite zig-zag domains, see Theorem~\ref{Thm:CompatibilityWhole}.

In the case of a finite surface, the construction (implying both existence and uniqueness) was based on the  
following simple observation (which will also play the key role for infinite surfaces).

\begin{observation}[\cite{P}, Corollary 4.8]
\label{Lem:L}
For every  $x_1,x_2,x_3\in \bR_+$, there exists a unique hyperbolic triangle $\Delta$ with all vertices at $\partial \mathcal H$ (modulo isometry of hyperbolic plane) and a unique choice of horocycles at its vertices such that $x_1, x_2, x_3$ are the lambda lengths of the sides of $\Delta$. 
\end{observation}

\begin{proof}
Up to an isometry of hyperbolic plane (i.e. up to a linear-fractional map of $\mathcal H$) all triangles with vertices on $\partial \mathcal H$ are equivalent to the triangle $01\infty$. Choose any horocycle $h_\infty$ at $\infty$ (it is represented in $\mathcal H$ by a horizontal line $\textup{Im } z=C$ for some $C\in \mathbb R$). Then there exists a unique horocycle $h_0$ at $0$ such that the lambda length of the arc $0\infty$ is $x_1$. Similarly, there is a unique horocycle $h_1$ at 1 such that the lambda length of the arc $1\infty$ is $x_2$. The lambda length of the arc $01$ is determined by  the horocycles $h_0$ and $h_1$, but in most cases is not equal to $x_3$. However, it changes monotonically depending on the initial choice of $h_\infty$ and it tends to 0 or $\infty$ when the constant $C$ in the definition of the horocycle $h_\infty$ is  very large or very small, respectively. This implies that for every value of $x_3$ there exists a unique choice of suitable horocycle $h_\infty$, which in turn gives a unique choice of the horocycles $h_0$ and $h_1$.   
\end{proof}

\subsection{Conditions induced by the geometry of the surface}
\label{compatibility}
We will now analyse   the behaviour of infinite sequences of arcs in an infinite surface $\cals.$ 
As a first easy result in this direction we will show that the lambda length of a limit arc equals to the limit of the lambda lengths of the arcs.

\begin{proposition}
\label{Prop:LimitCond} 
Let $\cals$ be an infinite surface with an admissible hyperbolic structure and with converging horocycles. Let $\{\gamma_i\}$ be a sequence of arcs in $\cals$ such that $\gamma_i \rightarrow \gamma_{\ast}$ as $i\to \infty$. Let  $x_{i}$ be the lambda length of $\gamma_i$ for $i\in \bN\cup \{*\}$.
Then $x_{i} \rightarrow x_{\ast}$ as $i\to \infty$.
\end{proposition} 

\begin{proof}
Since the horocycles on $\cals$ are converging by construction, 
 the intersection points of $\gamma_i$ with corresponding horocycles at both endpoints of $\gamma_i$  converge to the intersection points of $\gamma_{\ast}$ with horocycles at both endpoints of $\gamma_{\ast}$.
\end{proof} 

\begin{proposition}\label{Prop:HyperbolicEuclidean} Choose marked points $\{p_i\ | \  i\in\bN\cup \{* \}\}$ on the boundary of the upper half-plane $\calh$ so that $p_i\rightarrow p_{\ast}$ with $p_*\neq\infty,$ and assign horocycles $h_i$ at each marked point $p_i$ so that $h_i\rightarrow h_{\ast}$ as $i\rightarrow\infty.$ Let $x_{i,i+1}$ be the lambda length of the geodesic arc from $p_i$ to $p_{i+1}$ and let $s_{i,i+1}$ be the Euclidean distance between $p_i$ and $p_{i+1},$ for each $i$. Then
\begin{itemize}
\item[(1)] $\frac{x_{i,i+1}}{s_{i,i+1}} \rightarrow \frac{1}{2},$
\item[(2)] $\sum\limits_{i=1}^{\infty}x_{i,i+1} < \infty, $ and
\item[(3)] if $x_{i}$ is the lambda length of the arc starting at $p_*$ and ending at $p_i$, for $i\in\bN$,  then $x_i\rightarrow 0$ as $i \rightarrow \infty.$
\end{itemize}
\end{proposition}

\begin{proof} Part (1) is a straightforward computation in the upper half-plane (we write the equation of the hyperbolic line explicitly, find its intersections with horocycles, compute the lambda length, and take the limit). 

To see part (2), notice that since $p_*\neq \infty$ and $p_i \to p_*$ as $i \to \infty$, the condition $s_{i,i+1}=|p_{i+1}-p_i|>0$ implies  that the positive number series $\sum\limits_{i=1}^{\infty}s_{i,i+1}$ converges. By (1), $\frac{x_{i,i+1}}{s_{i,i+1}} \to \frac{1}{2}$, so there exists $N$ such that for $i > N$ we have $0 \leq x_{i,i+1}<s_{i,i+1}$.  So, $\sum\limits_{i=N}^{\infty}x_{i,i+1} < \sum\limits_{i=N}^{\infty}s_{i,i+1}< \infty$, and hence $\sum\limits_{i=1}^{\infty}x_{i,i+1}<\infty$.

Finally, to see part (3), notice that as the horocycles $h_i$ converge to some horocycle $h_*$ and the points $p_i$ converge to $p_*$ then, for large $i$, the horocycles $h_i$ and $h_*$ at the ends of the arc $p_ip_*$ produce the larger and larger intersections $\hat h_i\cap \hat h_*\cap p_ip_*$,
so that the lambda lengths $x_i$ tend to $0$ as $i \to \infty$ (here by  $\hat h_i$ and $\hat h_*$ we mean the horoballs bounded by the horocycles $h_i$ and $h_*$).

\end{proof}

\begin{figure}[!h]
\begin{center}
\psfrag{*}{ $p_*$}
\psfrag{x*}{\scriptsize $x_*$}
\psfrag{x_i,i+1}{\scriptsize $x_{i,i+1}$}
\psfrag{x1}{\scriptsize $x_{1}$}
\psfrag{x2}{\scriptsize $x_{2}$}
\psfrag{x3}{\scriptsize $x_{3}$}
\psfrag{xn}{\scriptsize $x_{n}$}
\psfrag{xn1}{\scriptsize $x_{n+1}$}
\psfrag{x12}{\scriptsize $x_{1,2}$}
\psfrag{x23}{\scriptsize $x_{2,3}$}
\psfrag{xnn1}{\scriptsize $x_{n,n+1}$}
\psfrag{1}{\tiny $p_1$}
\psfrag{2}{\tiny $p_2$}
\psfrag{3}{\tiny $p_3$}
\psfrag{n}{\tiny $p_n$}
\psfrag{n1}{\tiny $p_{n+1}$}
\psfrag{0}{\tiny $p_0$}
\psfrag{a}{\small (a)}
\psfrag{b}{\small (b)}
\psfrag{c}{\small (c)}
\psfrag{d}{\small (d)}
\epsfig{file=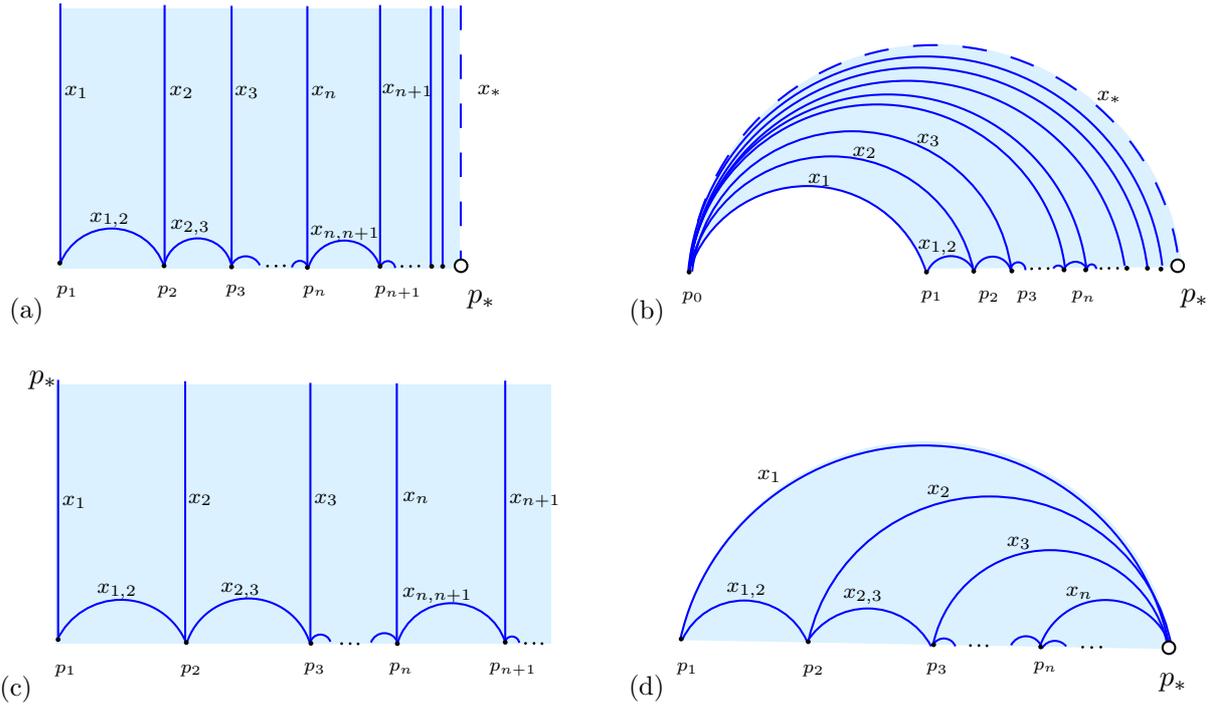,width=0.99\linewidth}
\caption{Notation for an incoming fan (on the top) and for an outgoing fan (at the bottom), respectively.
On the left, each fan is drawn  with the source at infinity, whereas the same fan is viewed with the source at a finite point on the right.}
\label{fig: notation for fans}
\end{center}
\end{figure}

Combining Propositions~\ref{Prop:LimitCond}  and~\ref{Prop:HyperbolicEuclidean}(3), we obtain the following corollary.

\begin{corollary}
\label{limit labmda length}
If $\{\gamma_i\}$ is a converging sequence of arcs in $\cals$, then the corresponding lambda lengths   $\{x_{\gamma_i}\}$ converge. Moreover,
\begin{itemize}
\item[-] if  $\{\gamma_i\}$ converges to a limit arc $\gamma_*$ on $\cals$ then  $x_{\gamma_i}\to x_{\gamma_*}$, and 
\item[-]
if $\{\gamma_i\}$ converges to an accumulation point, then  $x_{\gamma_i}\to 0$.
\end{itemize}
\end{corollary}

\begin{definition}
By a \emph{Laurent series} in an infinite variable set $\sX = \{x_i : i \in \bN \}$ we mean an infinite sum $\sum\limits_{i=1}^{\infty}m_i$ where each $m_i$ is a Laurent monomial in finitely many variables $\{x_{j_1}, x_{j_2}, \dots, x_{j_{k_i}}\} \subset \sX$ for some $k_i \in \bN$.
\end{definition}

\begin{proposition}\label{Rem:IncomingFan} 
Let  $\overline{\calf}=\{\gamma_i,\gamma_{i,i+1} \mid i\in\bN\}\cup\{\gamma_{\ast}\}$ 
be an elementary incoming fan  with converging horocycles. Let $x_{i}$   be the  lambda lengths of the arcs $\gamma_i$, $i\in \bN,$ and $x_{i,j}$ be the lambda length of the geodesic arc between marked points $p_i$ and $p_j$, for $i\neq j$,  where $i,j\in \mathbb N\cup \{ *\}$, see Fig.~\ref{fig: notation for fans}(a), (b).  Then 
\begin{itemize}
\item[(1)] $x_{i}\rightarrow x_{\ast}$ as $i\rightarrow\infty$; 
\item[(2)] $\sum\limits_{i=1}^{\infty}x_{i,i+1}<\infty$;
\item[(3)]  $x_{s,n}=x_sx_n\sum\limits_{i=s}^{n-1} \displaystyle \frac{x_{i,i+1}}{x_ix_{i+1}}$ for $s\in\bN$, $n>s$;
\item[(4)]  $x_{s,\ast}=x_sx_{\ast}\sum\limits_{i=s}^{\infty}  \displaystyle \frac{x_{i,i+1}}{x_ix_{i+1}}$ for $s\in\bN$;
in particular,   $x_{s,\ast}$ is an absolutely converging Laurent series in $\{x_i,x_{i,i+1} \mid i\in \bN\}$.
\end{itemize}
\end{proposition}

\begin{figure}[!h]
\begin{center}
\psfrag{1}{ \tiny $p_1$}
\psfrag{0}{ \tiny $p_s$}
\psfrag{i}{ \tiny $p_i$}
\psfrag{i+1}{  \tiny$p_{i+1}$}
\psfrag{*}{ \tiny $p_*$}
\psfrag{x*}{\scriptsize $x_*$}
\psfrag{x_1,*}{\color{red} \scriptsize $x_{1,*}$}
\psfrag{xi,i+1}{\scriptsize $x_{i,i+1}$}
\psfrag{x_i}{\scriptsize $x_{i}$}
\psfrag{x_i+1}{\scriptsize $x_{i+1}$}
\includegraphics[width = 0.55\textwidth]{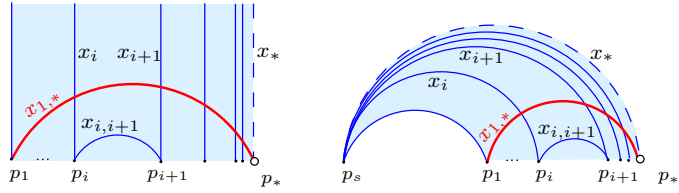}
\caption{An arc crossing an infinite incoming fan (left: viewed with the source at $\infty$; right: viewed with the source at  a finite point).}
\label{fig: notation for fans-arc}
\end{center}
\end{figure}

\begin{proof}
\begin{itemize}
\item[(1)] Follows from Proposition~\ref{Prop:LimitCond}.
\item[(2)] Follows from Proposition~\ref{Prop:HyperbolicEuclidean} (2).
\item[(3)] Without loss of generality, we will set $s=1$. We will show by induction on $n$ that 
\begin{align*}
x_{1,n}=x_1x_n\sum\limits_{i=1}^{n-1}\frac{x_{i,i+1}}{x_ix_{i+1}}.
\end{align*} 
 The base step is evident since
\begin{align*}
x_{1,2}&=x_{1}x_{2}\sum\limits_{i=1}^1\frac{x_{i,i+1}}{x_{i}x_{i+1}}.
\end{align*}
Assume now the identity holds for all $n<k$ for some $k\in\bN.$ By the Ptolemy relation in the quadrilateral $p_1p_{k-1}p_kp_{*}$,
\begin{align*}
x_{1,k}&= \frac{x_{k}x_{1,k-1}+x_{1}x_{k-1,k}}{x_{k-1}}\\
&=\frac{x_{k}}{x_{k-1}}x_{1}x_{k-1}\sum\limits_{i=1}^{k-2}\frac{x_{i,i+1}}{x_{i}x_{i+1}}+\frac{x_{1}x_{k-1,k}}{x_{k-1}}\\
&=x_{1}x_{k}\sum\limits_{i=1}^{k-1}\frac{x_{i,i+1}}{x_{i}x_{i+1}}.
\end{align*} 

\item[(4)] It remains to show $x_{1,n} \to x_1x_{\ast}\sum\limits_{i=1}^{\infty}  \displaystyle \frac{x_{i,i+1}}{x_ix_{i+1}}$ as $n \to \infty$, or in other words that   
$$\delta_k:= |x_1x_{\ast}\sum\limits_{i=1}^{\infty}  \displaystyle \frac{x_{i,i+1}}{x_ix_{i+1}} - x_{1,k}|\to 0 \quad  \text{as} \quad  k\to \infty.$$

By part (1) we have  $x_i \to x_{\ast}$, and  by part (2) we know that  $\sum\limits_{i=1}^{\infty}x_{i,i+1}$ converges. 
Therefore, for any  $\varepsilon>0,$ we can choose $k\in\bN$ such that  $\lvert x_i-x_{\ast} \rvert < \varepsilon$, and  $\sum\limits_{i=k}^{\infty}x_{i,i+1}<\varepsilon$ for all $i>k$. Also, as  $x_i>0$ for all $i$ and $x_i \to x_{\ast}$ with $x_i>0$ and $x_*>0$,  there exists $K\in \bR$ such that $\frac{1}{x_i}<\frac{1}{x_*}+\varepsilon$ for $i>K$ and $\big\vert \frac{1}{x_j} \big\vert <K$  for all $j$. Thus,  
\begin{align*}
\delta_k=\lvert x_1x_{\ast}\sum\limits_{i=1}^{\infty}  \displaystyle \frac{x_{i,i+1}}{x_ix_{i+1}} - x_{1,k}\rvert 
&= \bigg\vert x_1(x_{\ast}-x_k) \sum\limits_{i=1}^{k-1}  \displaystyle \frac{x_{i,i+1}}{x_ix_{i+1}}
+ x_1x_{\ast} \sum\limits_{i=k}^{\infty} \displaystyle \frac{x_{i,i+1}}{x_ix_{i+1}} \bigg\vert \\
&< x_1 \varepsilon K^2\sum\limits_{i=1}^{k-1}{x_{i,i+1}} + x_1x_{\ast}(\frac{1}{x_{\ast}}+\varepsilon)^2\sum\limits_{i=k}^{\infty} {x_{i,i+1}}.
\end{align*}
By Proposition~\ref{Prop:HyperbolicEuclidean} (2), $\sum\limits_{i=1}^{\infty}x_{i,i+1}$ converges, say to $C$. Therefore,
\begin{align*}
\delta_k
&<x_1 \varepsilon K^2C+x_1x_{\ast}(\frac{1}{x_{\ast}}+\varepsilon)^2\varepsilon.
\end{align*}
Hence, $\delta_k$ is bounded by a constant times $\varepsilon,$ which completes the proof.
\end{itemize}
\end{proof}

\begin{proposition}
\label{Prop:OutgoingFanConditions} 
Let $\calf=\{\gamma_i,\gamma_{i,i+1}\mid i\in\bN\}$ be an elementary outgoing fan with converging horocycles  and let $\{x_i,x_{i,i+1}\mid i\in\bN\}$ be the corresponding lambda lengths, $i\in \bN$  (see Fig.~\ref{fig: notation for fans} (c)-(d)). Then
\begin{itemize}
\item[(1)] $x_{i}\rightarrow 0$ as $i\rightarrow\infty$;
\item[(2)] $x_{n} \sum\limits_{i=1}^{n-1}\displaystyle \frac{x_{i,i+1}}{x_{i} x_{i+1}} \to 1$ as $n\rightarrow\infty.$
\end{itemize}
\end{proposition} 

\begin{proof}
Part (1) follows from Proposition~\ref{Prop:HyperbolicEuclidean} (3).
To show part (2), notice that  $x_{1,n}\to x_1$ as $n\to \infty$ in view of Proposition~\ref{Prop:LimitCond}. 
On the other hand, by Proposition~\ref{Rem:IncomingFan} (3), we  have
$x_{1,n}=x_{1}x_{n}\sum\limits_{i=1}^{n-1}\frac{x_{i,i+1}}{x_{i}x_{i+1}}$ for $n\in\bN.$ Hence, in order to have $x_{1,n} \to x_{1}$ as $n \to \infty,$ we need  $x_{n} \sum\limits_{i=1}^{n-1}\displaystyle \frac{x_{i,i+1}}{x_{i} x_{i-1}} \to 1.$ 
\end{proof} 

The proof of part (2) of Proposition~\ref{Prop:OutgoingFanConditions} implies the following corollary.

\begin{corollary}
\label{Cor:Lim} Let $\calf=\{\gamma_i,\gamma_{i,i+1}\mid i\in\bN\}$ be an elementary outgoing fan on the hyperbolic plane. Let $\{p_i\}$, $i\in \bN,$ be the marked points in $\calf$ such that $p_i\to p_*$ for some $p_*$, and let $h_i$ be a horocycle at $p_i$ and $h_*$ be a horocycle at $p_*$.   Let $\{x_i,x_{i,i+1}\mid i\in\bN\}$ be the corresponding lambda lengths. Then $h_i\to h_*$ if and only if $x_{n} \sum\limits_{i=1}^{n-1}\displaystyle \frac{x_{i,i+1}}{x_{i} x_{i+1}} \rightarrow 1$ as $n\rightarrow\infty$.
\end{corollary}

\begin{proof}
If  $h_i\to h_*$ then  $x_{n} \sum\limits_{i=1}^{n-1}\displaystyle \frac{x_{i,i+1}}{x_{i} x_{i+1}} \rightarrow 1$ by 
 Proposition~\ref{Prop:OutgoingFanConditions}(2).

Conversely, if   $x_{n} \sum\limits_{i=1}^{n-1}\displaystyle \frac{x_{i,i+1}}{x_{i} x_{i+1}} \rightarrow 1$,  then 
$x_{1,n}\to x_1$, and hence the horocycles converge.
\end{proof}

Now, we will introduce a metric characteristic of a fan which will allow us to distinguish incoming fans from outgoing ones.

\begin{definition}[Width of fan]
\label{def:width}
Let $\calf$ (resp. $\overline \calf$) be an elementary outgoing (resp. incoming) fan with the source $p_*$ (resp. $p_{0}$) and the arcs $\gamma_{n}=p_*p_n$, for $n\in \bN$ (resp. the arcs  $\gamma_{n}=p_0p_n$, for $n\in \bN$,  and a limit arc $\gamma_{*}=p_0p_*$). We will define the \textit{width} of $\calf$ (resp. $\overline \calf$) as follows:
\begin{itemize}
\item[-] drop a perpendicular from $p_1$ to every arc  $\gamma_{n}$   for $n\in \bN$ and denote by $q_n$ the foot of the perpendicular;
\item[-] let $l_n$ be the (signed) hyperbolic distance from $q_n$ to the horocycle $h_1$ centred at $p_1$ (in other words, $l_n$ is the hyperbolic length of the portion of the perpendicular lying outside of the horocycle $h_1$);
\item[-] $w_n=exp(l_n)$  will be called the {\it width} of a finite fan with $n-1$ triangles;
\item[-] the \emph{width $w_{\calf}$} of an infinite fan $\calf$ (resp. $\overline \calf$) is defined by the (finite or infinite) limit of $w_n$ as $n\rightarrow\infty$ (see Lemma~\ref{width} for the proof of the existence of the limit).
\end{itemize}
See Fig.~\ref{fig: width of incoming fan} for an illustration. We will simply use the notation $w$ for the width of a fan $\calf$ when the fan is clear from the context.

\end{definition}

\begin{lemma} 
\label{width}
The width of an infinite elementary fan is well-defined, i.e. the (finite or infinite) limit  in Definition~\ref{def:width} does exist. Moreover, the width of an incoming fan is finite and the width of an outgoing fan is infinite.

\end{lemma}

\begin{proof}
In the case of an incoming fan, the length $l_n$ tends to the distance between the horocycle $h_1$ centred at $p_1$ and the limit arc, and this distance is finite (see Fig.~\ref{fig: width of incoming fan}~{(a)-(b)}). In the case of an outgoing fan, the length $l_n$ is the distance between the horocycle $h_1$ and the line $\gamma_{n}$, which grows as $n$ gets larger and tends to the infinite distance between the horocycle $h_1$ and the accumulation point $p_*$, see Fig.~\ref{fig: width of incoming fan}~{(c)-(d)}. 

\end{proof}

\begin{figure}[!h]
\begin{center}
\psfrag{*}{\tiny $p_*$}
\psfrag{0}{\tiny $p_0$}
\psfrag{1}{\tiny $p_1$}
\psfrag{2}{\tiny $p_2$}
\psfrag{3}{\tiny $p_3$}
\psfrag{n}{\tiny $p_n$}
\psfrag{w12}{\tiny \color{RedOrange} $w_2$}
\psfrag{w13}{\tiny  \color{RedOrange} $w_3$}
\psfrag{w1n}{\tiny  \color{RedOrange} $w_n$}
\psfrag{w1*}{{\color{red} \tiny $w$}}
\psfrag{a}{\small (a)}
\psfrag{b}{\small (b)}
\psfrag{c}{\small (c)}
\psfrag{d}{\small (d)}
\epsfig{file=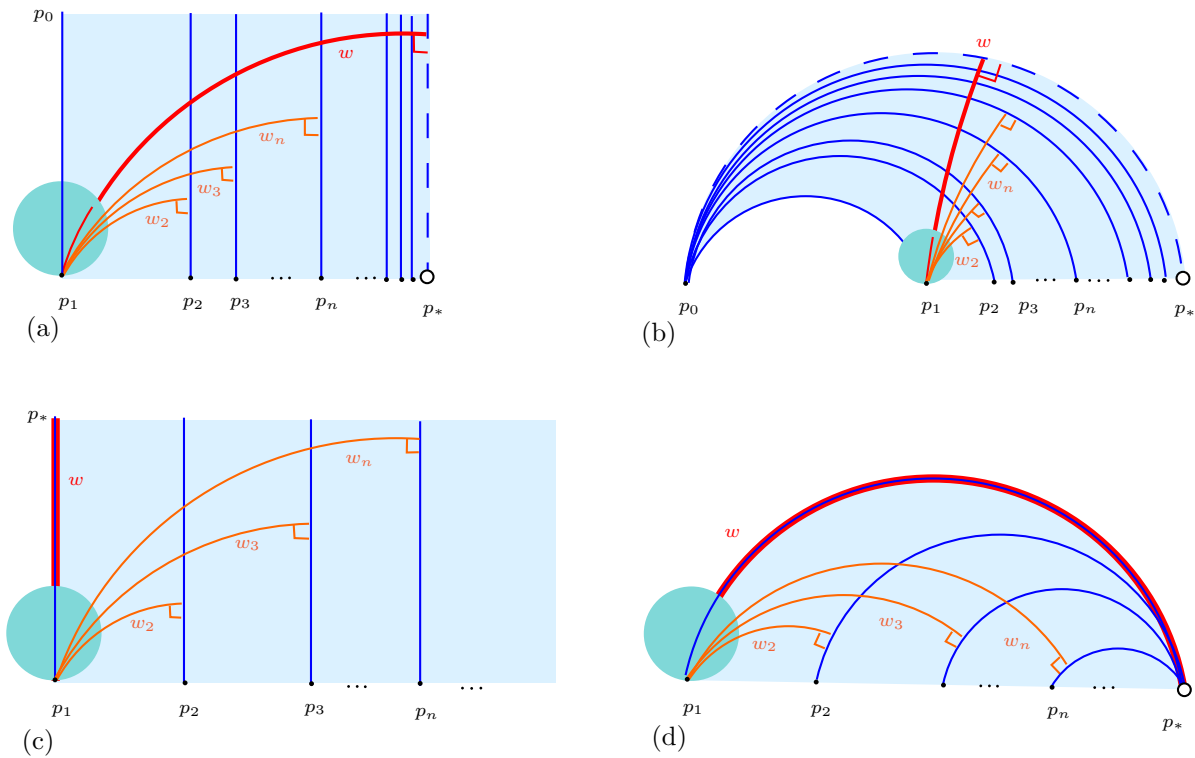,width=0.99\linewidth}
\caption{Width of an incoming fan ((a) and (b)) and of an outgoing fan ((c) and (d)). For each type of fans, we show the fan with the source at $\infty$ on the left and the same fan viewed with the source not at $\infty$ on the right. }
\label{fig: width of incoming fan}
\end{center}
\end{figure}

\begin{lemma}
\label{L width}
Let $\calf_n$ be a finite fan $\{\gamma_i,\gamma_{i,i+1} \ | \ 1\le i \le n-1\}\cup \{\gamma_n \}$, let $x_i$ and $x_{i,i+1}$  be the corresponding lambda lengths of arcs and boundary arcs of $\calf$. Then the width $w_n$ of $\calf_n$ satisfies
 $w_{n} =2x^2_{1} \sum\limits_{i=1}^{n-1}\frac{x_{i,i+1}}{x_{i}x_{i+1}}.$
\end{lemma}

\begin{proof}
By definition, the width $w_n$ is the lambda length of the double of the arc from $p_1$ to $q_n$,  namely, in order to compute $w_{n}$, we double the fan by taking a symmetric copy of it (see Fig.~\ref{fig: width}). If we set $p_{n+i}=p_n+(p_n-p_{n-i})$ for $1\leq i\leq n-1,$ then $w_n=x_{1,2n-1}$ and by Proposition~\ref{Rem:IncomingFan}~(3) we have
$$
w_{n}=x_{1,2n-1}=x_{1}x_{2n-1}\sum\limits_{i=1}^{2n-2}\frac{x_{i,i+1}}{x_{i}x_{i+1}}=2x_{1}^2 \sum\limits_{i=1}^{n-1}\frac{x_{i,i+1}}{x_{i}x_{i+1}}, 
$$
where the last equality holds since $x_{n+i}=x_{n-i}$  and $x_{i,i+1}=x_{2n-i-1,2n-i}$ for $i=1,\dots, n-1$. 
\end{proof}

\begin{figure}[!h]
\begin{center}
\psfrag{*}{ \tiny $p_*$}
\psfrag{x12}{\tiny $x_{1,2}$}
\psfrag{x13}{\tiny $x_{2,3}$}
\psfrag{x_3}{\tiny $x_{3}$}
\psfrag{x_n}{\tiny $x_{n}$}
\psfrag{xn1n}{\tiny $x_{n-1,n}$}
\psfrag{1}{\tiny $p_1$}
\psfrag{n-1}{\tiny $p_{n-1}$}
\psfrag{2}{\tiny $p_2$}
\psfrag{3}{\tiny $p_3$}
\psfrag{n}{\tiny $p_n$}
\psfrag{2n}{\tiny $p_{2n-1}$}
\psfrag{n+1}{\tiny $p_{n+1}$}
\psfrag{2n-1}{\tiny $p_{2n-2}$}
\psfrag{2n-2}{\tiny $p_{2n-3}$}
\psfrag{x1}{\tiny $x_1$}
\psfrag{xn1}{\tiny $x_{n-1}$}
\psfrag{x2}{\tiny $x_2$}
\psfrag{x3}{\tiny $x_3$}
\psfrag{xn}{\tiny $x_n$}
\psfrag{wn}{\scriptsize $ \color{RedOrange} w_n$}
\epsfig{file=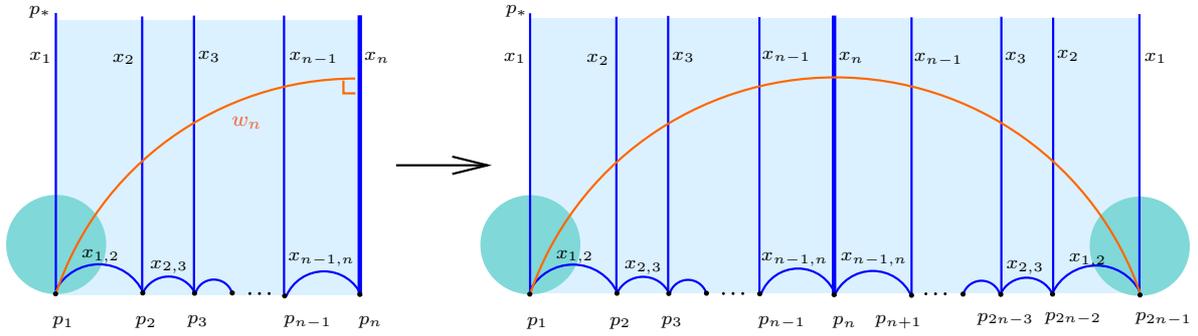,width=0.99\linewidth}
\caption{Computing the width of a finite fan by using two symmetric copies of it.}
\label{fig: width}
\end{center}
\end{figure}

\subsection{Geometry induced by compatible data}
\label{data->hyperbolic structure}

In this section, we will start with combinatorial data (triangulation $T$ of $\cals$) 
and numerical data (numbers associated with arcs of $T$) 
and show that under certain conditions the data  is realisable by a decorated hyperbolic surface with converging horocycles. The data sets $(\calf,\cald)$ and $(\overline{\calf},\overline{\cald})$ will be considered  for an outgoing and incoming elementary fans  and will have the following ingredients:
\begin{itemize}
\item[-] a collection of arcs 
\begin{center}
\begin{tabular}{ccc}
$\calf=\{\gamma_i,\gamma_{i,i+1} \mid i\in\bN\}$ & or & $\overline{\calf}=\{\gamma_i,\gamma_{i,i+1} \mid i\in\bN\}\cup\{\gamma_{\ast}\}$ 
\end{tabular}
\end{center}
forming an elementary outgoing or incoming fan (together with its base), respectively;
\item[-] a collection of positive real numbers
\begin{center}
\begin{tabular}{ccc}
$\cald=\{x_{i},x_{i,i+1}\in\mathbb{R}_{>0} \mid i\in\bN\}$ & or & $\overline{\cald}=\{x_{i},x_{i,i+1}\in\mathbb{R}_{>0}\mid i\in\bN\}\cup\{x_{\ast}\in\mathbb{R}_{>0}\}$
\end{tabular}
\end{center}
associated with each arc in the outgoing or  incoming fan, respectively, see Fig.~\ref{fig: notation for fans}.
\end{itemize}

\begin{definition}[Compatibility of combinatorial and numerical data] We say that $\cald$  
is \emph{compatible} with $\calf$ (resp.  $\overline{\cald}$ is \emph{compatible} with $\overline{\calf}$) if there is an admissible hyperbolic surface $\cals$ with converging horocycles triangulated according to $\calf$ (resp. $\overline{\calf}$) 
 and such that the lambda lengths of the geodesic arcs in $\cals$ coincide with the numerical data given in $\cald$ (resp. $\overline{\cald}$).
\end{definition}

\begin{proposition}\label{Prop:CompatibilityIncomingFan} Let $\overline \calf=\{\gamma_i, \gamma_{i,i+1} \mid i\in\bN\}\cup \{\gamma_* \}$ be an elementary incoming fan  and let $\overline{\cald}=\{x_{i},x_{i,i+1}\in\mathbb{R}_{>0}\mid i\in\bN\}\cup\{x_{\ast}\in\mathbb{R}_{>0}\}$, see Fig.~\ref{fig: notation for fans}~(a)-(b). Then $\overline{\calf}$ is compatible with $\overline{\cald}$ if and only if the following two conditions hold:
\begin{itemize}
\item[(1)] $x_n \to x_{{\ast}}$ as $n\to\infty,$ and 
\item[(2)] $\sum\limits_{i=1}^{\infty} x_{i,i+1} < \infty.$  
\end{itemize}
\end{proposition}

\begin{proof} By Proposition~\ref{Rem:IncomingFan}, the conditions are necessary for the existence of the corresponding admissible hyperbolic structure with  converging horocycles. To show they are also sufficient, we will assume that conditions (1) and (2) hold and construct an embedding of the incoming fan $\overline \calf$ to the upper half-plane.
We will map the source of the fan to $\infty$.

We start with  embedding the first triangle of the fan by mapping  it to $(0,1,\infty)=(p_1,p_2,p_0)$. Then we can glue each successive triangles $\Delta_i$ in $\overline \calf$ one by one so that each triangle satisfy $\Delta_i=p_ip_{i+1}p_0$. Indeed, Observation~\ref{Lem:L} implies that  given $p_i\in \partial \mathcal H$ (together with the choice of horocycles at $p_i$ and at $p_0$)  there exists a unique position for $p_{i+1}$ and a unique choice of horocycle $h_{i+1}$ at $p_{i+1}$ allowing to match the numerical data assigned to $\gamma_{i,i+1}$ and $\gamma_i$. It is left to show that the constructed points $p_i$ converge to some $p_*\ne \infty$ (i.e. we obtain an incoming fan as required and not an outgoing one) and that the constructed horocycles $h_i$ at $p_i$ converge to a horocycle at $p_*$.

By condition (1), $x_i\to x_*$ as $i\to \infty$, which implies that the Euclidean sizes of horocycles $h_i$ converge as $i\to \infty$. Condition (2) implies that $x_{i,i+1}\to 0$ as $i\to \infty$, which (as Euclidean sizes of horocycles converge) is only possible  if   
 $(p_{i+1}-p_i)\to 0$.

Now, shift each of the triangles $\Delta_i$ to $p_1$ by the map $f_i(x)=x-p_i+p_1$, $i\in\bN$ (the corresponding horocycles are shifted together with the triangles). Since $p_{i+1}-p_i\rightarrow 0$, we have $f_i(p_{i+1})\rightarrow p_1$ as $i\rightarrow\infty.$ Note also that the shifted horocycles $f_i(h_{i+1})$ at $f_i(p_{i+1})$ converge. Hence, we can use Proposition~\ref{Prop:HyperbolicEuclidean} (1) to see that $\frac{x_{i,i+1}}{s_{i,i+1}}\rightarrow \frac{1}{2}$ where $s_{i,i+1}=p_{i+1}-p_i.$ Since $\sum x_{i,i+1} <\infty$ by assumption, this implies that $\sum s_{i,i+1} <\infty$. Thus, $p_i\to p_*$ for some point $p_*\ne \infty$. As Euclidean sizes of the horocycles $h_i$ converge, we conclude that the horocycles $h_i$ converge.
\end{proof}

\begin{proposition} \label{Prop:CompatibilityOutgoingFan}
Let $\calf=\{\gamma_i,\gamma_{i,i+1}\mid i\in\bN\}$ be an elementary outgoing fan and let $\cald=\{x_{i},x_{i,i+1}\in\mathbb{R}_{>0} \mid i\in\bN\}$, see Figure~\ref{fig: notation for fans}~(c)-(d). Then $\calf$ is compatible with $\cald$ if and only if the following two conditions hold:
\begin{itemize}
\item[(1)] $x_n\to 0$ as $n\to \infty$, and
\item[(2)] $x_{n} \sum\limits_{i=1}^{n-1}\displaystyle \frac{x_{i,i+1}}{x_{i} x_{i-1}} \to 1$ as $n\to\infty$.  
\end{itemize}
\end{proposition}

\begin{proof} The conditions are necessary in view of  Proposition~\ref{Prop:OutgoingFanConditions}. To prove they are also sufficient, we suppose that the assumptions (1) and (2) hold and construct an embedding of the outgoing fan to the upper half-plane (with the source of the fan mapped to $\infty$). 

We start by embedding the first triangle of the fan and then we glue each successive triangle while assigning the appropriate horocycles at its vertices at each step. We construct some infinite fan in $\mathcal H$. Combining assumptions (1) and (2), we have   
$\sum\limits_{i=1}^{n-1}\displaystyle \frac{x_{i,i+1}}{x_{i} x_{i-1}} \to \infty$ as $n\to \infty$. In view of Lemma~\ref{L width} this implies that
 the width of the constructed fan is infinite, so the fan is outgoing. This implies $p_i\rightarrow\infty$ as $i\rightarrow\infty.$ Assumption (2) together with Corollary~\ref{Cor:Lim} implies that horocycles at marked points $p_i$ converge, as required. 
\end{proof}

In the presence of almost elementary fans (i.e. in the case of an elementary fan with infinitely many finite polygons attached to it), some additional conditions will be required. To introduce them we will need  the following notation.

\begin{notation}\label{notation:bits} Let $\calf$ be an almost elementary fan and $p_1, p_2,\dots$ be the marked points at the base of the elementary fan contained in $\calf$, see Fig.~\ref{Fig: bits}.  We may assume without loss of generality that the marked points $p_1,p_2, \dots$ lie at the same boundary component of $\cals$ (up to shifting the base of the fan by finitely many marked points, i.e. by relabelling $p_n$ by $p_{n-k}$ for some $k\in\bN$). Let $s_n$ be any boundary marked point lying between $p_n$ and $p_{n+1}$. Denote by $x_{n,s_n}$ and $x_{s_n,n+1}$ the lambda lengths of the arcs $p_ns_n$ and $s_n,p_{n+1}$. Then the following condition holds.
\end{notation}

\begin{proposition}  In the Notation~\ref{notation:bits}, the following additional condition holds for almost elementary domains:
\[
\frac{x_{n,s_n}+x_{s_n,n+1}}{x_{n,n+1}} \to 1.
\]
\end{proposition}

\begin{proof}
Denote by $x_{s_n}$ the lambda length of the arc $p_*s_n$. By Ptolemy relation, we have
\[
x_{n,n+1}x_{s_n} = x_{s_n,n+1}x_n+x_{n,s_n}x_{n+1}.
\]
As the horocycles at the boundary marked points accumulating to $p_*$ converge, each of $x_{s_n},x_n, x_{n+1}$ tends to $x_*$. So, we have $x_{n,s_n}+x_{s_n,n+1} \to x_{n,n+1}$ as $n \to \infty$.
\end{proof}

We collect compatibility conditions for elementary and almost elementary fan domains in Table~\ref{Table:Compatibility}.

\begin{remark}
For each of the two types of fans it is easy to check that none of the two conditions  is sufficient alone without the other condition.
\end{remark}

\begin{table}[!h]
\caption{Compatibility conditions elementary fans}
\label{Table:Compatibility}
\begin{center}
\begin{tabular}{|c|c|c||c|}
\hline
& \raisebox{-9pt}{Domain} &  \raisebox{-9pt}{Conditions} &  \raisebox{-2pt}{Additional conditions}\\
&&& \raisebox{2pt}{for almost elementary fans}\\
  \hline \hline

\rotatebox{90}{ incoming fan } & \raisebox{0.pt}{
\psfrag{1}{\tiny $p_0$}
\psfrag{*}{\tiny $p_*$}
\psfrag{x_*}{\tiny $x_*$}
\psfrag{i}{\tiny $p_i$}
\psfrag{i+1}{\tiny $p_{i+1}$}
\psfrag{x_i}{\tiny $x_i$}
\psfrag{x_i+1}{\tiny $x_{i\!+\!1}$}
\psfrag{x_i,i+1}{\tiny $x_{i,i\!+\!1}$}
\includegraphics[scale=1.2]{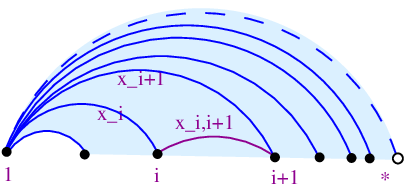}} 
& 
\raisebox{30pt}{
\begin{tabular}{l}
\raisebox{10pt}{$x_n\to x_{\ast}$}\\ \raisebox{0pt}{$\sum\limits_{i=1}^\infty x_{i,i+1} < \infty$} 
\end{tabular}
}
& {\Large $\frac{x_{n,s_n}+x_{s_n,{n+1}}}{x_{n,n+1}}$  } $\to 1 $ 
\\
\cline{1-3} 
\cline{1-3}
\rotatebox{90}{ outgoing fan }&  \raisebox{0.pt}{
\psfrag{*}{\tiny $p_*$}
\psfrag{1}{\tiny $p_1$}
\psfrag{i}{\tiny $p_i$}
\psfrag{i+1}{\tiny $p_{i+1}$}
\psfrag{x_i,*}{\tiny $x_i$}
\psfrag{x_i+1,*}{\tiny $x_{i+1}$}
\psfrag{x_i,i+1}{\tiny $x_{i,i+1}$}
\includegraphics[scale=1.2]{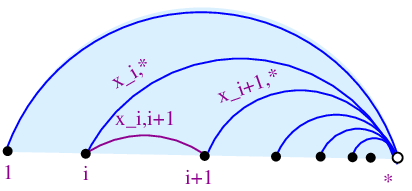}}  & 
\raisebox{30pt}{
\begin{tabular}{l}
\raisebox{10pt}{$x_n\to 0$}\\ \raisebox{0pt}{$x_{n}\sum\limits_{i=1}^\infty\frac{x_{i,i+1}}{x_ix_{i-1}} \to 1$} 
\end{tabular}
}
& 
\\
\hline
\end{tabular}
\end{center}
\end{table}

\begin{proposition}\label{Prop:CompatibilityAlmostElemFan} The compatibility conditions for almost elementary incoming and outgoing fans are exactly the same as the ones  given in  Table~\ref{Table:Compatibility} (this comprises two conditions for incoming / outgoing fans together with one additional condition for almost elementary fans). 
\end{proposition}

\begin{proof} We use Propositions~\ref{Prop:CompatibilityIncomingFan} and~\ref{Prop:CompatibilityOutgoingFan} 
to construct elementary incoming and outgoing fans embedded into almost elementary fans. We also construct separately all polygons attached to almost elementary fans, see Fig.~\ref{Fig: bits}. Then we glue them to elementary fans to construct a surface for almost elementary incoming and outgoing fans.  We have constructed an admissible surface with a choice of horocycles at marked points so that the lambda lengths of the arcs are equal to the given numerical data. It remains to show that the horocycles given by this construction converge.

Notice that by Propositions~\ref{Prop:CompatibilityIncomingFan} and \ref{Prop:CompatibilityOutgoingFan} we know that the horocycles at $p_i$ converge to the one at $p_*$, so we only need to show that the horocycles at the marked points $s_i$  converge to the one at $p_*$. Assuming the condition $\frac{x_{n,s_n}+x_{s_n,n+1}}{x_{n,n+1}} \to 1$ is satisfied and using Ptolemy relation $x_{n,n+1}x_{s_n} = x_{s_n,n+1}x_n+x_{n,s_n}x_{n+1}$ as in Proposition~\ref{Prop:CompatibilityAlmostElemFan} together with the condition $x_n\to x_*$ which follows from Propositions~\ref{Prop:CompatibilityIncomingFan} and \ref{Prop:CompatibilityOutgoingFan}, we compute
\[
x_{s_n}=\frac{x_{s_n,n+1}x_{n}+x_{n,s_n}x_{n+1}}{x_{n,n+1}}\to x_*.
\]
This implies $x_{s_n}\to x_*$, hence the horocycles at $x_{s_n}$ converge to the horocycle at $p_*$.
\end{proof}

\begin{figure}[!h]
\begin{center}
\psfrag{xi}{\scriptsize $x_i$}
\psfrag{xi1}{\scriptsize $x_{i+1}$}
\psfrag{xi,i1}{\scriptsize $x_{i,i+1}$}
\psfrag{pst}{\scriptsize $p_\ast$}
\psfrag{pi}{\scriptsize $p_i$}
\psfrag{p0}{\scriptsize $p_0$}
\psfrag{si}{\scriptsize $s_i$}
\psfrag{pi1}{\scriptsize $p_{i\!+\!1}$}
\psfrag{pi-1}{\scriptsize $p_{i\!-\!1}$}
\epsfig{file=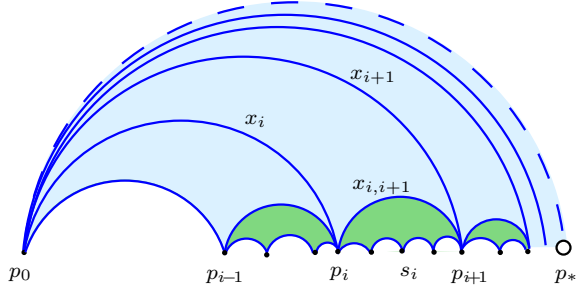,width=0.5\linewidth}
\caption{An almost elementary fan as a union of an elementary fan and an infinite number of finite polygons, attached to its base.}
\label{Fig: bits}
\end{center}
\end{figure} 

\begin{remark} We do not know how to derive compatibility conditions for elementary zig-zag pieces. This leads us to the following definition.
\end{remark}

\begin{definition}[Fan triangulations, incoming/outgoing fan triangulations] 
\label{Def:FanTri}
We say a  triangulation $T$ of $\cals$ is
\begin{itemize}
\item  a \emph{fan triangulation} if 
$T$ is a finite union of almost elementary fan domains;
\item an {\it incoming} (resp. {\it outgoing}) fan triangulation if all almost elementary domains are incoming (resp. outgoing) fans.
\end{itemize}
\end{definition}

\begin{example}
The triangulation used in the proof of Theorem~\ref{distinguished triangulation} is an outgoing fan triangulation (see Fig.~\ref{fig: best triang}).
\end{example}

\begin{remark}
\label{curve in gan T}
Let $T$ be a fan triangulation of $\cals$ and $\gamma\in \cals$ be  an arc. Then the domain $\cald_\gamma^T$ of $\gamma$ in $T$ is a finite union of elementary fans (compare to Remark~\ref{fan is elementary}).
\end{remark}

\begin{theorem}\label{Thm:CompatibilityWhole} Given a fan triangulation $T$ of an infinite surface $\cals$ together with a collection of positive real numbers associated with arcs and boundary arcs of $T$, if the compatibility conditions listed in Table~\ref{Table:Compatibility} are satisfied for all almost elementary domains of $T$ then the numerical data is compatible with the  combinatorial data in the whole surface $\cals.$
\end{theorem}

\begin{proof}
We cut the surface $\cals$ into finitely many almost elementary domains and realise each domain using Proposition~\ref{Prop:CompatibilityAlmostElemFan}. We then glue these domains in finitely many steps.
\end{proof}

\subsection{Teichm\"uller space}
\begin{definition}\label{Def:Comp} We say numerical values $\{x_i\in \bR_{>0}  \mid i\in\bN\}$ are \emph{compatible} with a fan triangulation $T=\{\gamma_i\mid i\in\bN\}$ of an infinite surface $\cals$ if compatibility conditions hold for every almost elementary domain in $T$.
\end{definition}

In view of Definition~\ref{Def:Comp}, we reformulate Theorem~\ref{Thm:CompatibilityWhole} as follows.

\begin{corollary}
Let $T$ be a fan triangulation of an infinite surface $\cals.$ Collections of values compatible with $T$ are in bijection with points in the decorated Teichm\"uller space $\widetilde{\calt}(\cals)$. 

\end{corollary}

Since the compatibility conditions for incoming fan triangulations are linear, the Teichm\"uller space has the following nice properties.

\begin{theorem} The set of points  in the decorated Teichm\"uller space $ \widetilde{\calt}$ of an infinite surface $\cals$ forms a convex cone.
\end{theorem}

\begin{proof}
Let $T=\{\gamma_i\mid i\in\bN\}$ be a triangulation formed as a union of incoming fan triangulations and $\underline{x}_T=\{x_i \mid  i\in \bN\}$ be the associated collection of values.
Since the compatibility conditions (see top row of Table~\ref{Table:Compatibility}) are linear in the case of an incoming fan, it is immediate that, given $\underline{x}\in\widetilde\calt$ and $\lambda>0,$  we have $\lambda \underline{x}\in\widetilde{\calt}$ and given $\underline{x},\underline{y}\in\widetilde{\calt}$ we have $\underline{x}+\underline{y}\in\widetilde{\calt}.$ Also $\lambda\underline{x}+(1-\lambda)\underline{y}\in\widetilde{\calt}$ for any scalar $\lambda,$ where $0<\lambda<1.$ Hence, $\widetilde \calt$ is a convex cone.
\end{proof}

\subsection{Lambda lengths as Laurent series}
Given a triangulated infinite surface $\cals$ with an admissible hyperbolic structure and with converging horocycles,  
consider the lambda lengths $\{x_i\}$ of the arcs $\gamma_i$ of $T$ as variables (depending on the point in $\widetilde \calt$).

\begin{proposition}
\label{prop Laurent}
Let $\cals$ be an unpunctured infinite surface. Let  $T$ be a fan triangulation of $\cals$ and $\underline x=\{x_i\in \bR_{>0} \mid i\in \bN\}$ be variables (including boundary variables) compatible with $T.$ If $\gamma$ is an arc in $\cals$, then the lambda length of $\gamma$ is a  Laurent series in $\{x_1,x_2,\dots\}$ converging absolutely for any $\underline x$ compatible with $T$.
\end{proposition}

\begin{proof}
Consider the lift $D$ of the domain $\cald_\gamma^T$ of $\gamma$ to the universal cover. By Proposition~\ref{Lem:FinArcs-->Discs}  together with Remark~\ref{fan is elementary},  $D$ is a finite union of elementary fans (as $T$ contains no infinite zig-zags).  
The proof of the statement is by induction on the number of elementary fans in $D$. The base case is  when $\gamma$ lies in  a single fan and it follows from Proposition~\ref{Rem:IncomingFan}.

To show that the statement holds for an arc crossing $k$ elementary fans, we assume that it holds for all arcs crossing less than $k$ fans and apply the Ptolemy relation (see Fig.~\ref{Fig: Ptolemy-induction} for notation): 
$$ x_\gamma x_{\gamma'}=x_{\beta_1}x_{\beta_3}+x_{\beta_2}x_{\beta_4},$$
where the $k$-th elementary fan of $D$ is attached to the $(k-1)$-th one along the arc $\gamma'\in T$. 
Notice that the arcs $\beta_1$, $\beta_2$, $\beta_3$, $\beta_4$ satisfy the induction assumption, which implies that the variables $x_{\beta_j}$ are converging Laurent series in $\{x_i\}$. As $\gamma'\in T$, dividing both sides by $x_{\gamma'}$ we obtain the required expression for $x_\gamma$. 
\end{proof}

\begin{figure}[!h]
\begin{center}
\psfrag{g}{\color{red}  $\gamma$}
\psfrag{*}{\color{blue}  $\gamma'$}
\psfrag{1}{\color{dgreen}  $\beta_1$}
\psfrag{2}{\color{dgreen}  $\beta_2$}
\psfrag{3}{\color{dgreen}  $\beta_3$}
\psfrag{4}{\color{dgreen}  $\beta_4$}
\epsfig{file=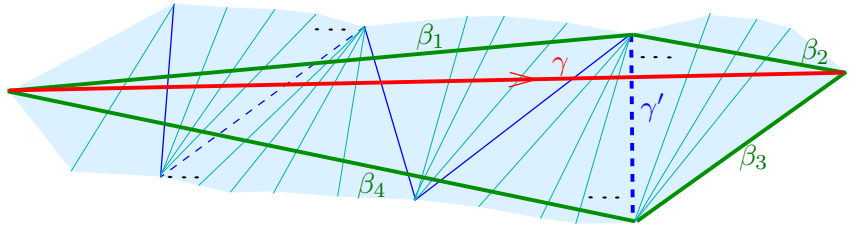,width=0.7\linewidth}
\caption{Induction on the number of fans crossed by $\gamma$: Ptolemy relation.}
\label{Fig: Ptolemy-induction}
\end{center}
\end{figure} 

\begin{remark} In the proof of Proposition~\ref{prop Laurent}, the fans adjacent to the arc $\gamma'$ may be finite or infinite, they may  either share the same source or have the sources at different endpoints of $\gamma'$; however, the proof does not depend on any of these factors.
\end{remark}

We will extend the result of  Proposition~\ref{prop Laurent}  to the case of punctured surfaces in Theorem~\ref{Laurent for punctures}.
For that matter we will need the following definitions and a technical lemma.

\begin{definition}[Ordinary and conjugate pair punctures with respect to $T$]
Let $T$ be a triangulation of an infinite surface $\cals$. We say that a puncture $p$ of $\cals$ 
\begin{itemize}
\item is an {\it ordinary puncture} with respect to $T$ if all arcs of $T$ incident to $p$ are tagged the same way at $p$;
\item otherwise, $p$ is the endpoint of two copies of the same arc with different taggings at $p$ (these arcs are called a {\it conjugate pair}). In this case, we  will  call $p$  a {\it conjugate pair puncture} for $T$  (see Fig.~\ref{fig: conjugate puncture}(a) for an example).
\end{itemize}
\end{definition}

\begin{definition}[Right and wrong tagging]
Let $T$ be a triangulation of a punctured surface and let $p$ be an ordinary puncture with respect to $T$. An arc $\gamma$ with an end point at $p$ has  a {\it right tagging}  (resp. wrong tagging) at $p$ if it is tagged the same (resp. tagged oppositely) at $p$ as the arcs of $T$.    
\end{definition}

\begin{figure}[!h]
\begin{center}
\psfrag{ps}{\color{blue} \scriptsize $\psi$}
\psfrag{ps'}{\color{blue} \scriptsize $\psi'$}
\psfrag{m}{\color{DarkOrchid} \scriptsize $\mu$}
\psfrag{p}{\scriptsize  $p$}
\psfrag{a}{\small (a)}
\psfrag{b}{\small (b)}
\epsfig{file=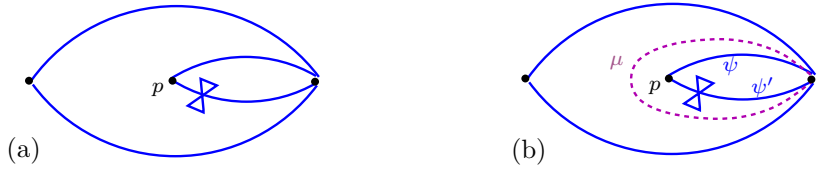,width=0.7\linewidth}
\caption{A conjugate pair puncture: $x_\mu=x_\psi x_{\psi'}$.}
\label{fig: conjugate puncture}
\end{center}
\end{figure} 

\begin{lemma}
\label{l:arc at puncture}
Let $T$ be a triangulation of an infinite surface $\cals$ and let $p$ be a puncture on $\cals.$ Then there exists an arc $\psi\in T$ having exactly one endpoint at $p$.
\end{lemma}

\begin{proof}
Suppose that all arcs of $T$ incident to $p$ have {\it both} endpoints at $p$. Then all triangles incident to $p$ have all three vertices at $p$. The neighbours of those triangles also have the same property, namely all their vertices coincide with $p$. This implies that all vertices of all triangles in $T$ coincide with $p$. However, an infinite surface $\cals$ should contain infinitely many boundary marked points, which implies that  there should be some vertices different from $p$.
\end{proof}

\begin{theorem}
\label{Laurent for punctures}
The statement of Proposition~\ref{prop Laurent}  holds  for punctured infinite surfaces, i.e. if $T$ is a fan triangulation of an infinite surface $\cals$ and $\underline x=\{x_i\mid i\in \bN\}$ are variables (including boundary variables) compatible with $T,$ and $\gamma$ is an arc in $\cals$, then the lambda length of $\gamma$ is a  Laurent series in $\{x_1,x_2,\dots\}$ converging absolutely for any $\underline x$ compatible with $T$.
\end{theorem}

\begin{proof} 
The proof remains the same for arcs with no endpoints at punctures and for arcs having one or two endpoints at ordinary punctures with the right tagging at these punctures.
This implies that we are left to prove the statement in two cases:

\begin{itemize}
\item when  $\gamma$  has a wrongly tagged endpoint at an ordinary  puncture $p$;
\item when  $\gamma$ has an endpoint at a conjugate pair puncture $p$.
\end{itemize}

\medskip
\noindent
{\bf Case 1: wrong tagging at an ordinary puncture.}
Let $\gamma$ be an arc with a wrongly tagged endpoint at an ordinary puncture $p$. We will assume that the other endpoint of   $\gamma$ is not a conjugate puncture. We will consider the three cases given below: 

\begin{itemize}
\item[1.a:] {\bf only one endpoint of   $\gamma$ has wrong tagging.}
Let $\gamma'$ be the same arc as $\gamma$ but with the opposite tagging at $p$. As it is shown above,   $x_{\gamma'}$   is a Laurent series in $\{x_1,x_2,\dots\}$. We will use $x_{\gamma'}$ to find  $x_{\gamma}$.

Notice that given two curves $\varphi$ and $\psi$ both with exactly one endpoint at $p$ and tagged the same at $p$, and given the arcs $\varphi'$ and $\psi'$ which coincide with $\psi$ and $\varphi$ except that they are tagged oppositely at $p$, one has 
$$
x_{\varphi}/x_{\varphi'}=x_{\psi}/x_{\psi'}
$$
since the ratio above is $e^{\pm d/2}$, where $d$ is the distance between the two conjugate horocycles  at $p$, see~\cite[Section 7]{FT}.
Hence, it is sufficient to show that the ratio $x_{\psi}/x_{\psi'}$ is a Laurent series for some specifically  chosen arc $\psi'\in T$ and get 
$$x_{\gamma}=x_{\gamma'}x_{\psi}/x_{\psi'},$$
which is a Laurent series as  $x_{\gamma'}$ is.

If the triangulation $T$ is locally finite at $p$ (i.e. there are finitely many triangles incident to $p$), then we choose 
$\psi'\in T$ to be an arc having exactly one end at $p$ (there is such an arc in view of Lemma~\ref{l:arc at puncture}). Notice that $\psi'$ is tagged oppositely to $\gamma$ at $p$ by the assumption. Then
$x_{\psi}$ is a Laurent polynomial in view of classical theory for finite surfaces, and hence $x_{\psi}/x_{\psi'}$ is also a Laurent polynomial in $\{x_1,x_2,\dots \}$. 

If $T$ is not locally finite around $p,$ then there are infinitely many arcs of $T$  incident to $p$. Hence, there is a limit arc $\psi'\in T$ incident to $p$, i.e. an arc with one endpoint at $p$ and another endpoint $q\in\partial \cals$ (since a limit arc should have at least one endpoint at the boundary). Again by the assumption, $\psi'$ is tagged oppositely to $\gamma$ at $p$. Let $\psi$ be the arc parallel to $\psi'$ but  tagged the same as $\gamma$ at $p$, and let $\mu$ be the loop going around $\psi'$, see Fig.~\ref{fig: conjugate puncture}(b).
Then
 $x_\mu$ is a Laurent series in $\{x_1,x_2,\dots\}$ as it is shown above (since this arc has no endpoints at punctures). 
By~\cite[Lemma~7.10]{FT},  $x_\mu=x_\psi x_{\psi'}$. 
This implies that $x_{\psi}/x_{\psi'}=x_\mu/(x_{\psi'})^2$ is also a converging Laurent series since $\psi'\in T$. 

\item[1.b:] {\bf two endpoints of  $\gamma$ are distinct ordinary punctures, both with wrong taggings.}
In this case we apply the same reasoning as in Case~1.a, changing the tagging of ${\gamma}$ first in one endpoint and then at the other.

\item[1.c:] {\bf both endpoints of ${\gamma}$  are at the same ordinary puncture $p$, tagged wrongly at $p$.}
In this case, $x_{\gamma'}/x_{\gamma}=(e^{\pm d/2})^2$ since we add or remove a segment of length $d$ from both ends of $\gamma$
which implies that 
$$x_{\gamma}=x_{\gamma'} (x_{\psi}/x_{\psi'})^2,$$
where $\psi'\in T$ is chosen in the same way as in Case 1.a.
\end{itemize}

\medskip
\noindent
{\bf Case 2: conjugate pair puncture.} Let $p$ be a puncture and let $\rho,\rho'\in T$ be two curves forming a conjugate pair incident to $p$. Let $\alpha,\beta\in T$ be the sides of the digon containing the conjugate pair, and let $\gamma$ be an arc coming to $p$ through the arc $\beta$ (see Fig.~\ref{fig: conjugate}(a)). Without loss of generality, we may assume that $\gamma$ is tagged at $p$ in the same way as $\rho$. 

We will consider the following three cases:

\begin{itemize}
\item[2.a:] {\bf only one endpoint of $\gamma$ is a conjugate pair puncture}. 
Let $\sigma$  be the arc inside the digon tagged the same as $\rho$ at $p$. By the Ptolemy relation, we have
$$
x_\gamma x_\beta= x_\sigma x_\psi+x_\rho x_\varphi,
$$
see Fig.~\ref{fig: conjugate}(a).
Notice that none of the endpoints of the arcs $\alpha$ and $\beta$ can be a conjugate pair puncture in $T$. This implies that  $x_\psi$ and $x_\varphi$ are Laurent series in $\{x_1,x_2,\dots \}$ by Case 1. As $\rho\in T$, $x_\rho$ is one of $\{x_1,x_2, \dots \}$.
The lambda length $x_{\sigma}$ may be computed by the Ptolemy relation in a (self-folded) quadrilateral lying inside the digon:
$$
x_\sigma x_{\mu}=x_{\rho}( x_\alpha+ x_\beta),
$$
where $\mu$ is a loop around $\rho$. By~\cite[Lemma~7.10]{FT}, $x_\mu=x_\rho x_{\rho'},$ 
which implies
$$
x_\sigma=\frac{x_\alpha+x_\beta}{x_{\rho'}}.
$$
Finally, as $\beta\in T$ we see that  
$$
x_\gamma = \frac{x_\sigma x_\psi+x_\rho x_\varphi}{x_\beta}
$$
is a converging Laurent series.

\begin{figure}[!h]
\begin{center}
\psfrag{g}{\color{red}  \scriptsize $\gamma$}
\psfrag{r}{\color{blue} \scriptsize  $\rho$}
\psfrag{r'}{\color{blue} \scriptsize  $\rho'$}
\psfrag{a}{\color{blue} \scriptsize  $\alpha$}
\psfrag{b}{\color{blue} \scriptsize  $\beta$}
\psfrag{m}{\color{DarkOrchid} \scriptsize  $\mu$}
\psfrag{s}{\color{dgreen} \scriptsize  $\sigma$}
\psfrag{p}{\color{dgreen} \scriptsize  $\psi$}
\psfrag{f}{\color{dgreen} \scriptsize  $\varphi$}
\psfrag{1}{\color{dgreen} \scriptsize  $\delta_1$}
\psfrag{2}{\color{dgreen} \scriptsize  $\delta_2$}
\psfrag{3}{\color{dgreen} \scriptsize  $\delta_3$}
\psfrag{4}{\color{dgreen} \scriptsize  $\delta_4$}
\psfrag{a_}{\small (a)}
\psfrag{b_}{\small (b)}
\psfrag{c_}{\small (c)}
\epsfig{file=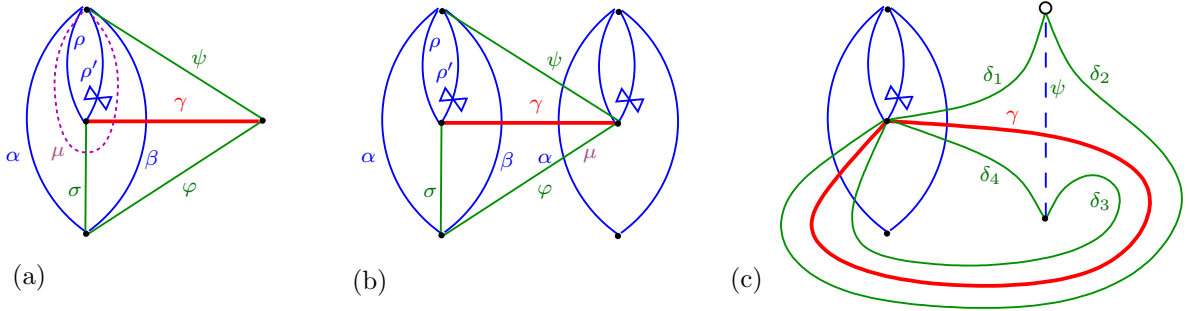,width=0.98\linewidth}
\caption{Arc $\gamma$ with an endpoint in a conjugate puncture. }
\label{fig: conjugate}
\end{center}
\end{figure} 

\item[2.b:] {\bf two endpoints of $\gamma$ are at two distinct conjugate pair punctures $p$ and $q$.}
The proof is the same as in  Case~2.a. We apply the Ptolemy relation as in Fig.~\ref{fig: conjugate}(b),
and  since the arcs $\psi$ and $\varphi$ have only one of their endpoints at a conjugate pair puncture, 
$x_\psi$ and $x_\varphi$ are converging Laurent series by  Case~2.a.

\item[2.c:] {\bf both endpoints of $\gamma$ are at the same conjugate pair puncture $p$.} First suppose that $\gamma$ intersects finitely many arcs of $T$. Then the domain $\cald_\gamma^T$ of $\gamma$ for $T$ consists of finitely many triangles, which implies that $x_\gamma$ is a Laurent polynomial in terms of the lambda lengths of the finitely many arcs contained in the finite surface $\cald_\gamma^T$.

Now, suppose that $\gamma$ intersects infinitely many arcs. As both endpoints of $\gamma$ are at the conjugate puncture $p$ (i.e. there are finitely many arcs intersecting $\gamma$ at the endpoints), there exists a limit arc $\psi\in T$ crossing $\gamma$.
Notice that both endpoints of $\psi$ are distinct from $p$ (as $p$ is a conjugate pair puncture), i.e. for any arc $\varphi$ starting from $p$ and terminating at an endpoint of $\psi$, the lambda length $x_\varphi$ is a Laurent series in $\{x_1,x_2,\dots \}$ in view of Cases~2.a and~2.b. Consider a lift  of $\gamma$ to the universal cover (together with a lift of $\psi$ intersecting the lift of $\gamma$). Applying the Ptolemy relation as in  Fig.~\ref{fig: conjugate}(c), we obtain
$$x_\gamma x_\psi=x_{\delta_1} x_{\delta_3} +x_{\delta_2}  x_{\delta_4}, $$
which implies that $x_\gamma$ is a Laurent series as each of $x_{\delta_i}$ is for $i=1,2,3,4$; and thus, $\psi\in T$
(here $\delta_i$, $i=1,\dots,4$, is obtained by resolving the crossing of $\psi$ and $\gamma$). 
\end{itemize}

\end{proof}

\begin{corollary}
\label{positivity}
The coefficients of the Laurent series obtained in Theorem~\ref{Laurent for punctures} are positive integers.

\end{corollary}

\begin{proof}
The Laurent series were constructed by applying Ptolemy relations repeatedly, where the only division performed was  by the variable corresponding to an arc lying in the triangulation. Also, we have not used subtraction in the construction. So, the coefficients of the terms are positive integers by induction.
\end{proof}

\begin{theorem}
\label{variables are distinct}
Let $T$ be a fan triangulation of an infinite surface $\cals$. Let $\beta,\gamma \in \cals$ be two distinct arcs
and $x_\beta$ and $x_\gamma$  be the Laurent series associated with $\beta$ and $\gamma$  with respect to the triangulation $T$. Then $x_\gamma\ne x_\beta$. 
\end{theorem}

\begin{proof}
First, notice that given a fan triangulation $T$, the Laurent series (in terms of the initial lambda lengths $\{x_1,x_2,\dots \}$ of the internal and boundary arcs of $T$) does not depend on the choice of initial numerical data compatible with the combinatorial structure of the triangulation. Hence, it is sufficient to show that for any two arcs $\beta$ and $\gamma$  in the combinatorial infinite surface $\cals$, there is an admissible hyperbolic structure with converging horocycles such that the lambda lengths of $\beta$ and $\gamma$  are different. In what follows, we construct such a hyperbolic structure on the (combinatorial) surface $\cals$.

Fix two arcs $\beta,\gamma\in \cals$, and let $T$ be an outgoing fan triangulation (as in Definition~\ref{Def:FanTri}). 
Then each of the arcs $\beta$ and $\gamma$ crosses only  finitely many arcs of $T$. 
Hence, it is possible to choose finitely many arcs $\alpha_1,\dots,\alpha_n$ (where $n= Acc(\cals)$ is the total number of accumulation points, see Definition~\ref{Defn:Acc(S)}) so that 
\begin{itemize}
\item[-] none of the arcs $\alpha_i$ crosses $\beta$ or $\gamma$, and 
\item[-] $\{\alpha_i\}$ cuts $\cals$ into finitely many pieces one of which is a finite subsurface $\cals'$ containing only finitely many triangles of $T$ and all the others 
 being  elementary outgoing fans. 
\end{itemize}

Without loss of generality (i.e. by changing the choice of $\alpha_i$ if needed), we may assume that $\beta,\gamma\in \cals'$. Then the lambda lengths of $\gamma$ and $\beta$ are expressed through finitely many of $\{x_i\}$ (more precisely, both of them are Laurent polynomials in  $\{x_i \mid i\in I(\cals')\}$, where $I(\cals')$ is a finite index set of the arcs lying inside or on the boundary of $\cals'$).
In particular, $x_\gamma\ne x_\beta$ in view of the corresponding result for finite surfaces (see~\cite{FST}, proof of Proposition~9.21). This implies that for some choice of  $\{x_i \mid i\in I(\cals')\}$ the corresponding lambda lengths of $\beta$ and $\gamma$ are different.

Now, fix a hyperbolic structure and a choice of horocycles on $\cals'$ such that the lambda lengths of $\beta$ and $\gamma$ are different. To get the required structure on $\cals$, we are left to attach finitely many outgoing fans with a unique condition that the lambda lengths of arcs $\alpha_i$ cutting the fans are given from the boundary of $\cals'$. This can be easily done by constructing the corresponding fans in the upper half-plane.
\end{proof}

\subsection{Cluster algebras from infinite surfaces}
The definition of cluster algebras from infinite surfaces is an immediate generalisation of the one associated with finite surfaces.

\begin{definition}[Cluster algebra from infinite triangulated surface] Let $\cals$ be an infinite surface together with an admissible hyperbolic structure and a choice of  converging horocycles, and let $T$ be a \emph{fan} triangulation of $\cals$. Denote by $\{\gamma_i\}$, $i\in \bN,$ the set of arcs of $T$ including infinitely many  boundary arcs $\{ \beta_j \}$ as a subset. Let $\{x_i \}$ be the lambda lengths of $\{ \gamma_i\}$. Then
\begin{itemize}
\item[-] for any (internal) arc $\gamma$ in $\cals$, the Laurent series giving the lambda length of $\gamma$ in terms of   $\{ x_i \}$
will be called the {\it cluster variable} associated with $\gamma$;
\item[-] denote by $\mathfrak X$ the {\it set of all cluster variables};
\item[-] the lambda lengths (considered as independent variables) of the boundary arcs will be called {\it boundary coefficients}; 
\item[-] by a {\it cluster} we mean a collection of cluster variables associated with arcs lying in one triangulation together with the boundary coefficients;
\item[-] by an {\it initial cluster} we mean the set  $\{x_i \}$ corresponding to the arcs of an {\it initial triangulation} $T$ (notice that we require an initial triangulation to be a fan triangulation);
\item[-] by a {\it seed} we mean a triangulation of $\cals$ together with the corresponding cluster;
\item[-] {\it mutations} of triangulations are defined as in Section~\ref{Sec:Combinatorics}, and \emph{mutations of clusters} are corresponding transformation of the set of functions.
\end{itemize}
Finally,
\begin{itemize}
 
\item[-] the {\it cluster algebra} $\mathcal A(\cals,T)$  associated with $(\cals,T)$ is 
the $\bZ$-algebra   generated by all cluster variables, i.e. 
$$
\cala(\cals, T)=\left\{ \displaystyle{\sum\limits_{k=1}^t} \prod\limits_{i=1}^{n_k} x_{j_i}^{m_i} \mid n_k, j_i, t\in \bN, \ m_i\in\bZ, x_{j_i}\in\mathfrak X \right\}.
$$
\end{itemize}
\end{definition}

\begin{remark}
A cluster variable may be defined as an infinite Laurent series, however when generating a cluster algebra we only consider {\it finite  sums of finite products} of cluster variables. 
\end{remark}

\begin{remark} 
As usual, in the case of punctured surfaces by ``triangulations'' we mean ``tagged triangulations''.
\end{remark}

\begin{remark}[Independence of triangulation]
\label{independence}
The definition of cluster algebra  $\mathcal A(\cals,T)$ given above depends on the choice of a fan triangulation,
however one can check that given two fan triangulations $T_1$ and $T_2$ of the  surface $\cals$, there exists a homomorphism $f$  of algebras  $\cala_1 =\mathcal A(\cals,T_1)$  and $\cala_2=\mathcal A(\cals,T_2)$. 

Indeed, by Theorem~\ref{variables are distinct} we see that a cluster variable is uniquely determined by the corresponding arc $\gamma\in \cals$. Construct the bijection between the sets of cluster variables of $\cala_1$ and $\cala_2$ by mapping the cluster variable of $\cala_1$  associated  to an arc $\gamma$  to the cluster variable of $\cala_2$ associated with $\gamma$. 

To see that the bijection $f$ defines an algebra homomorphism, notice that  cluster variables of both algebras satisfy the same relations coming from the geometry of $\cals$, i.e.

\begin{itemize}
\item[-] Ptolemy relations, and
\item[-] limit relations (see Corollary~\ref{limit labmda length} stating that the lambda lengths of the converging sequence of arcs converge).
\end{itemize}

Moreover, for any given initial fan triangulation $T$, every cluster variable may be obtained from the initial cluster variables in finitely many steps (where at every step we apply either a Ptolemy relation or a limit relation to the variables obtained in the previous steps). 
\end{remark}

Remark~\ref{independence} leads to the following definition.

\begin{definition}[Cluster algebra from an infinite surface]
By a {\it cluster algebra $\cala(\cals)$ from an infinite surface $\cals$} we mean the cluster algebra obtained from any fan triangulation of $\cals$.
\end{definition}
\section{Infinite staircase snake graphs} \label{Sec:SG}
In this section, we will generalise the construction of snake graphs, introduced by Musiker, Schiffler and Williams in \cite{MSW}, for arcs in an infinite surface and use them to give an expansion formula for cluster variables. Since we consider cluster algebras with initial variables associated with fan triangulations,  we will introduce the construction of snake graphs in this setting for simplicity, though infinite snake graphs associated with non-fan triangulations can be given in a similar way.
\subsection{Construction of snake graphs associated with fan triangulations}
Let $\cals$ be an infinite surface with a fan triangulation $T$ and $\gamma$ be an arc in $\cals.$ Consider the domain $\cald_{\gamma}^T$ of $\gamma$ lifted to the universal cover. The domain $\cald_{\gamma}^T$ is a finite union of elementary fans $\calf_i$ where $1\leq i \leq k$ for some $k\in\bN.$ Fix an orientation on $\gamma.$ 

We will construct the snake graph of $\gamma$ in several steps depending on the complexity of the collection of crossings of $\gamma$ with $T$. In the case of finitely many crossings, our construction coincides with that of \cite{MSW}. Suppose $\{\tau_i\mid i\in I\}$ is the collection of arcs in $T$ that $\gamma$ crosses in the universal cover (remark that $\tau_i$ may be a lift of the same arc as $\tau_j$,  for some $i$ and $j$). 

\bigskip
\noindent\textit{Step I. Finitely many crossings  of $\gamma$ with $T$:}

\begin{itemize}
\item[-] to each crossing of $\gamma$ in $T$, associate a weighted tile which is a square of a fixed side length:
\begin{itemize}
\item[--] if $\gamma$ crosses an arc $\tau_i,$ consider the quadrilateral $Q_i$ in $T$ which contains $\tau_i$ as a diagonal; the corresponding tile $G_i$ will have $\tau_i$ in the diagonal and weights on the edges induced by  the labelling on $Q_i$, see Figure~\ref{Fig:Cross_SG};
\end{itemize}
\item[-] glue successive tiles as follows: the tiles $G_i$ and $G_{i+1}$ are glued along the edge with weight $\tau_b$, where $\tau_b$ is the third side of the triangle bounded by $\tau_i$ and $\tau_{i+1}$; arrange the gluing so  that the diagonals $\tau_i$ and $\tau_{i+1}$  of the tiles $G_i$ and $G_{i+1}$ connect the top-left to the bottom-right corners of the tiles, see Figure~\ref{Fig:Cross_SG}.
\end{itemize}

\begin{figure}[!h]
\begin{center}
\psfrag{ti-1}{\scriptsize $\tau_{i-1}$}
\psfrag{ti}{\scriptsize $\tau_{i}$}
\psfrag{ti+1}{\scriptsize $\tau_{i+1}$}
\psfrag{ti+2}{\scriptsize $\tau_{i+2}$}
\psfrag{ti+3}{\scriptsize $\tau_{i+3}$}
\psfrag{ti+4}{\scriptsize $\tau_{i+4}$}
\psfrag{ta}{\scriptsize $\tau_{a}$}
\psfrag{tb}{\scriptsize $\tau_{b}$}
\psfrag{tc}{\scriptsize $\tau_{c}$}
\psfrag{td}{\scriptsize $\tau_{d}$}
\psfrag{te}{\scriptsize $\tau_{e}$}
\psfrag{i-1}{\scriptsize ${i-1}$}
\psfrag{i}{\scriptsize ${i}$}
\psfrag{i+1}{\scriptsize ${i+1}$}
\psfrag{i+2}{\scriptsize ${i+2}$}
\psfrag{i+3}{\scriptsize ${i+3}$}
\psfrag{i+4}{\scriptsize ${i+4}$}
\psfrag{a}{\scriptsize ${a}$}
\psfrag{b}{\scriptsize ${b}$}
\psfrag{c}{\scriptsize ${c}$}
\psfrag{d}{\scriptsize ${d}$}
\psfrag{e}{\scriptsize ${e}$}
\epsfig{file=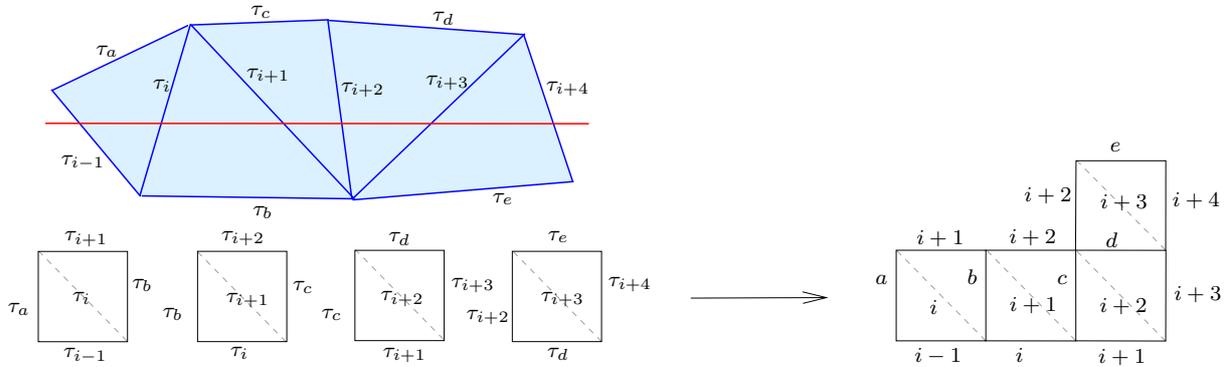,width=0.99\linewidth}
\caption{Construction of snake graphs. Notice that two successive tiles have opposite orientations relative to the orientation of the surface.}
\label{Fig:Cross_SG}
\end{center}
\end{figure} 

\noindent\textit{Step II. Crossing   an infinite elementary fan:}
\begin{itemize}
\item[-] associate a tile to the first crossing of an almost elementary fan;
\item[-] glue each tile associated with successive crossings as in the finite case; 
\item[-] obtain an infinite graph $\overline{\calg}$;
\item[-] associate a  \emph{limit tile} $G_{lim}$ at the end of $\overline{\calg}$ to indicate that we approach to the limit arc $\gamma_{\ast}$ of the fan:
\begin{itemize}
\item $G_{lim}$  is a dotted square of the same side length; 
\item  $\overline{\calg}$ approaches $G_{lim}$  from the bottom;
\item 
associate the weight $*$ to the top edge of  $G_{lim}$, see left of Fig.~\ref{Fig:Inf_zigzag_SG}.
\end{itemize}
\end{itemize}

\begin{figure}[!h]
\centering
\psfrag{*}{\color{blue} \scriptsize $*$}
\psfrag{g}{\color{red} \scriptsize $\gamma$}
\epsfig{file=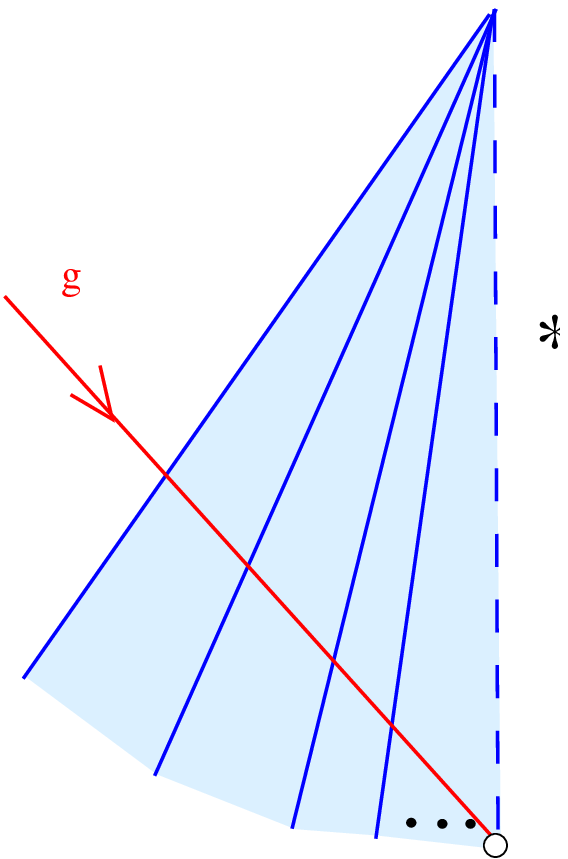,width=0.085\linewidth}
\hspace{.2in}
\begin{tikzpicture}[scale=.4]
\draw (0,0)--(2,0)--(2,1)--(0,1)--(0,0) (1,0)--(1,1)
(1,1)--(1,2)--(2,2)--(2,1) (2,1)--(3,1)--(3,2)--(2,2)--(2,3)--(3,3)--(3,2);
\draw[thick,loosely dotted] (3.2,2.5)--(4.5,3.8);
\draw[thick,densely dotted] (4,4)--(5,4)--(5,5)--(4,5)--(4,4);
\node[scale=.7,color=blue] at (4.5,5.3){$\ast$};
\end{tikzpicture}
\hspace{1.1in}
\epsfig{file=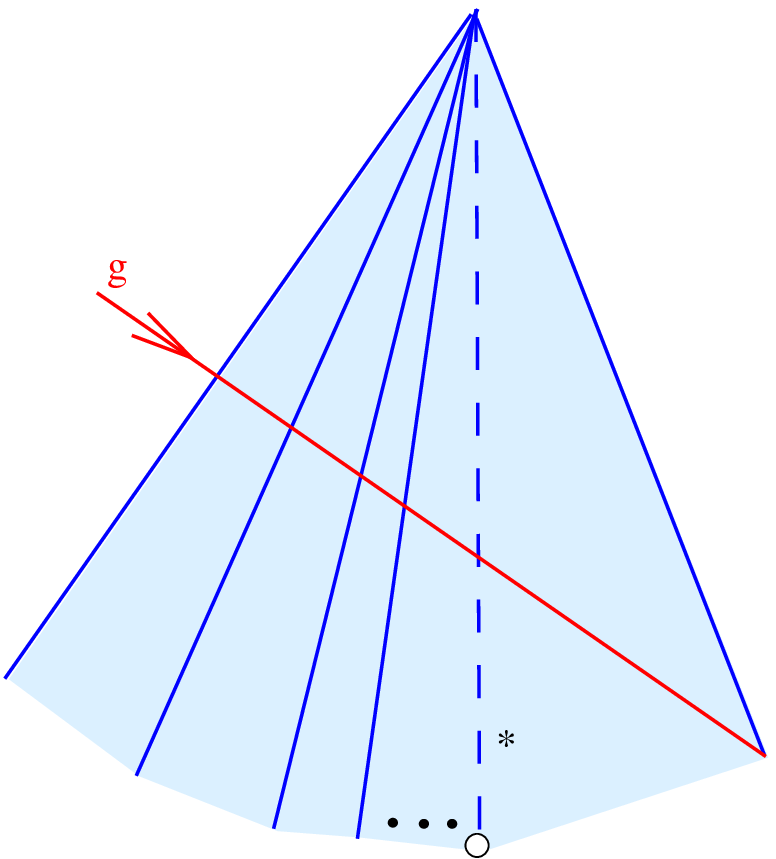,width=0.117\linewidth}
\hspace{.2in}
\begin{tikzpicture}[scale=.4]
\draw (7,0)--(9,0)--(9,1)--(7,1)--(7,0) (8,0)--(8,1)
(8,1)--(8,2)--(9,2)--(9,1) (9,1)--(10,1)--(10,2)--(9,2)--(9,3)--(10,3)--(10,2);
\draw[thick,loosely dotted] (10.2,2.5)--(11.5,3.8);\draw[thick,densely dotted] (11,4)--(12,4)--(12,5)--(11,5)--(11,4);
\node[scale=.7,color=blue] at (11.5,5.3){$\ast$};
\node[scale=.7] at (12.5,4.5){$\ast$};
\draw (12,4)--(13,4)--(13,5)--(12,5);
\end{tikzpicture}
\caption{Infinite one-sided zig-zag snake graph associated with an infinite elementary fan on the left and infinite snake graph associated with an arc crossing an elementary fan and the limit arc of the fan on the right.} 
\label{Fig:Inf_zigzag_SG}
\end{figure}

\noindent\textit{Step III. Crossing a limit arc:}
\begin{itemize}
\item[-] if an arc $\gamma$ crosses a limit arc $\gamma_{\ast}$, we associate  a regular tile to this crossing and glue it to  $\overline{\calg}$ according to the local configuration of $\gamma$ and $\gamma_{\ast}$ in $(S,T)$, see Fig.~\ref{Fig:Inf_zigzag_SG} (right).
\end{itemize}

\noindent\textit{Step IV. Crossing   a finite union of  infinite elementary fans and polygons:}

\begin{itemize}
\item[-] construct finitely many (finite or one-sided infinite zig-zag) snake graphs $\{\calg_i\}_{i=1}^{k}$   for elementary fans $\calf_i$ in $\cald_{\gamma}$ by Step I and II; 
\item[-] glue each $\calg_i$ and $\calg_{i+1}$ at the common ends according to the local configuration of the change of the fans from $\calf_{i}$ to $\calf_{i+1}$, see Fig.~\ref{snake-example} (i.e. attaching the last tile of $\calg_i$ and the first tile of $\calg_{i+1}$ to the tile $G_{0}$ associated with $\gamma_0=\calf_i\cap\calf_{i+1}$ in a zig-zag way when $\calf_i$ has the same source as $\calf_{i+1}$ and in a straight way otherwise, compare to Fig.~\ref{Fig:Fan_vs_Zigzag});
\end{itemize}

\noindent\textit{Step V. Completing the construction of snake graph $\ga$ associated with $\gamma$:}
\begin{itemize}
\item[-] remove the diagonals in all tiles.
\end{itemize}

\noindent
After this process, we obtain an infinite snake graph $\ga$ associated with $\gamma$, see Fig.~\ref{snake-example} for an  example.  

\begin{figure}[!h]
\begin{center}
\psfrag{1}{\color{Orange} \scriptsize $1$}
\psfrag{2}{\color{Orange} \scriptsize $2$}
\psfrag{3}{\color{Orange} \scriptsize $3$}
\psfrag{*}{  $*$}
\psfrag{**}{  $**$}
\psfrag{cast}{\scriptsize $\circledast$}
\psfrag{1_}{\color{brown(traditional)} \scriptsize $\tilde 1$}
\psfrag{2_}{\color{brown(traditional)} \scriptsize $\tilde 2$}
\psfrag{3_}{\color{brown(traditional)} \scriptsize $\tilde 3$}
\psfrag{1'}{\color{blue} \scriptsize $1'$}
\psfrag{2'}{\color{blue} \scriptsize $2'$}
\psfrag{3'}{\color{blue} \scriptsize $3'$}
\psfrag{2''}{\color{dgreen} \scriptsize $2''$}
\psfrag{3''}{\color{dgreen} \scriptsize $3''$}
\psfrag{4''}{\color{dgreen} \scriptsize $4''$}
\psfrag{1'''}{\color{DarkOrchid} \scriptsize $1'''$}
\psfrag{2'''}{\color{DarkOrchid} \scriptsize $2'''$}
\psfrag{3'''}{\color{DarkOrchid} \scriptsize $3'''$}
\epsfig{file=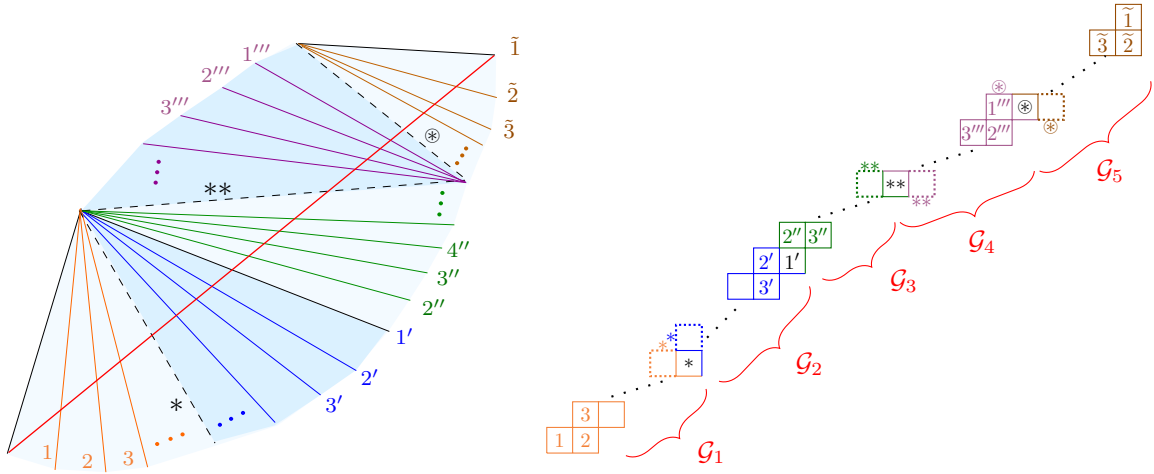,width=0.43\linewidth}
\raisebox{0pt}{
\begin{tikzpicture}[scale=.34]
{
\color{Orange}

\draw (0,0)--(2,0)--(2,1)--(0,1)--(0,0) (1,0)--(1,1)
(1,1)--(1,2)--(3,2)--(3,1)--(2,1)--(2,2) 
(6,3)--(5,3)--(5,4);

\draw[thick,densely dotted] (5,3)--(4,3)--(4,4)--(5,4);

\node[scale=.7] at (0.5,0.5){$1$}; 
\node[scale=.7] at (1.5,0.5){$2$};
\node[scale=.7] at (1.5,1.5){$3$};
 
\node[scale=.7] at (4.5,4.2){$*$};

}
\draw[thick,loosely dotted] (2.5,2.2)--(4.5,2.8);
\node[scale=.7] at (5.5,3.5){$*$};

{
\color{blue}
\node[scale=.7] at (4.8,4.5){$*$}; 

\draw (6,3)--(6,4)--(5,4)  
(10,7)--(8,7)--(8,8)--(9,8)--(9,6)--(7,6)--(7,7)--(8,7)--(8,6);
\draw[thick,densely dotted] (5,4)--(5,5)--(6,5)--(6,4);
\node[scale=.7] at (8.5,7.5){$2'$}; 
\node[scale=.7] at (8.5,6.5){$3'$}; 

}

\draw[thick,loosely dotted] (6.2,4.5)--(7.5,5.8);
\node[scale=.7] at (9.5,7.5){$1'$};

{
\color{dgreen}
\node[scale=.7] at (9.5,8.5){$2''$}; 
\node[scale=.7] at (10.5,8.5){$3''$}; 
\draw (10,7)--(10,9)--(9,9)--(9,8)--(11,8)--(11,9)--(10,9);
\draw[thick,densely dotted] (13,10)--(12,10)--(12,11)--(13,11);
\draw (13,11)--(13,10)--(14,10);
\node[scale=.7] at (12.5,11.2){$**$}; 

}

\draw[thick,loosely dotted] (10.5,9.2)--(12.5,9.8);
\node[scale=.7] at (13.5,10.5){$**$};

{
\color{DarkOrchid}
\node[scale=.7] at (14.5,9.8){$**$}; 
\draw (13,11)--(14,11)--(14,10);
\draw[thick,densely dotted] (14,11)--(15,11)--(15,10)--(14,10);
\draw (18,13)--(16,13)--(16,12)--(18,12)--(18,14)--(17,14)--(17,12) (18,13)--(19,13);
\node[scale=.7] at (16.5,12.5){$3'''$}; 
\node[scale=.7] at (17.5,12.5){$2'''$}; 
\node[scale=.7] at (17.5,13.5){$1'''$}; 
\node[scale=.7] at  (17.5,14.28){$\circledast$};
}

{
\color{brown(traditional)}
\draw (18,14)--(19,14)--(19,13);
\draw[thick,densely dotted] (19,13)--(20,13)--(20,14)--(19,14);
\node[scale=.7] at  (19.5,12.72){$\circledast$};
\draw (21,15.5)--(23,15.5)--(23,17.5)--(22,17.5)--(22,15.5) (21,15.5)--(21,16.5)--(23,16.5);
\node[scale=.7] at (21.5,16){$\widetilde 3$}; 
\node[scale=.7] at (22.5,16){$\widetilde 2$}; 
\node[scale=.7] at (22.5,17){$\widetilde 1$}; 
}

\node[scale=.7] at  (18.5,13.5){$\circledast$};

\draw[thick,loosely dotted] (14.5,11.2)--(16.5,11.8)
(19.5,14.1)--(21.62,15.49);

\draw[color=red,decorate,decoration={brace,amplitude=8pt},xshift=-4pt,yshift=-0pt]
 (6.3,2.5)--(3.2,-0.1) node [color=red,midway,xshift=.6cm,yshift=-.4cm] 
{\small $\calg_1$};

\draw[color=red,decorate,decoration={brace,amplitude=8pt},xshift=-4pt,yshift=-0pt]
 (10.3,6.5)--(6.8,2.8) node [color=red,midway,xshift=.6cm,yshift=-.4cm] 
{\small $\calg_2$};

\draw[color=red,decorate,decoration={brace,amplitude=8pt},xshift=-4pt,yshift=-0pt]
 (13.6,9.0)--(10.8,6.8) node [color=red,midway,xshift=.6cm,yshift=-.4cm] 
{\small $\calg_3$};

\draw[color=red,decorate,decoration={brace,amplitude=8pt},xshift=-4pt,yshift=-0pt]
 (19.0,11)--(13.8,9.0) node [color=red,midway,xshift=.2cm,yshift=-.6cm] 
{\small $\calg_4$};

\draw[color=red,decorate,decoration={brace,amplitude=8pt},xshift=-4pt,yshift=-0pt]
 (23.5,14.5)--(19.2,11) node [color=red,midway,xshift=.2cm,yshift=-.6cm] 
{\small $\calg_5$};

\end{tikzpicture}
}
\caption{The snake graph for an arc crossing five elementary fans $\calf_1,\dots, \calf_5$ is composed out of five one-sided zig-zag snake graphs   $\calg_1,\dots, \calg_5$ (we drop the boundary weights on the snake graph except the ones coming from limit arcs). }
\label{snake-example}
\end{center}
\end{figure}

\begin{remark} 
\label{rem:orient}
\begin{itemize}
\item[(a)] The limit tile may be approached by an infinite zig-zag from the left or from the bottom and there is no difference how we draw this; however, its effect is detected on the snake graph by sign function introduced in Definition~\ref{defn:signfunction}.
\item[(b)] In Step IV, by a one-sided infinite zig-zag we mean an infinite zig-zag snake graph together with a limit tile associated with a fan in the triangulation, see left of Figure~\ref{Fig:Inf_zigzag_SG}. 
\item[(c)] Changing the orientation of an arc $\gamma$ gives rise to isomorphic snake graphs, one obtained from the other by a reflection. On the other hand, the choice of orientation of $\gamma$ induces a total order on the tiles of $\ga.$ 
\item[(d)] A limit tile of a snake graph has no face weight and  only one of its edges has a weight. This edge is labelled by the limit arc.
\end{itemize}
\end{remark}

\begin{observation} 
\label{Obs:LimitTile}
\begin{itemize}
\item[(a)] If an arc $\gamma\in\cals$ crosses arcs forming a zig-zag in a triangulation, then the corresponding piece in the snake graph is straight and if it forms a fan, then the corresponding piece in the snake graph is zig-zag, see Fig.~\ref{Fig:Fan_vs_Zigzag} (we acknowledge that this terminology is a bit confusing but it will be clear from the context which meaning of {\it zig-zag} we refer to). 
\item[(b)] By construction, there is no tile in a snake graph which is a limit tile for two consecutive infinite subgraphs of an infinite staircase  snake graph (approaching the tile from two distinct sides). Indeed, if an arc $\gamma$ crosses an elementary fan leaving it through the limit arc, then $\gamma$ crosses at least a limit arc $\gamma_*$ for which we associate a regular tile $G_*$ in $\ga$. Hence, a tile associated with a limit arc separates the limit tile from the rest of the snake graph. Therefore, $\ga$ cannot have a two-sided limit tile.
\end{itemize}
\end{observation}

\begin{figure}[!h]
\begin{center}
\epsfig{file=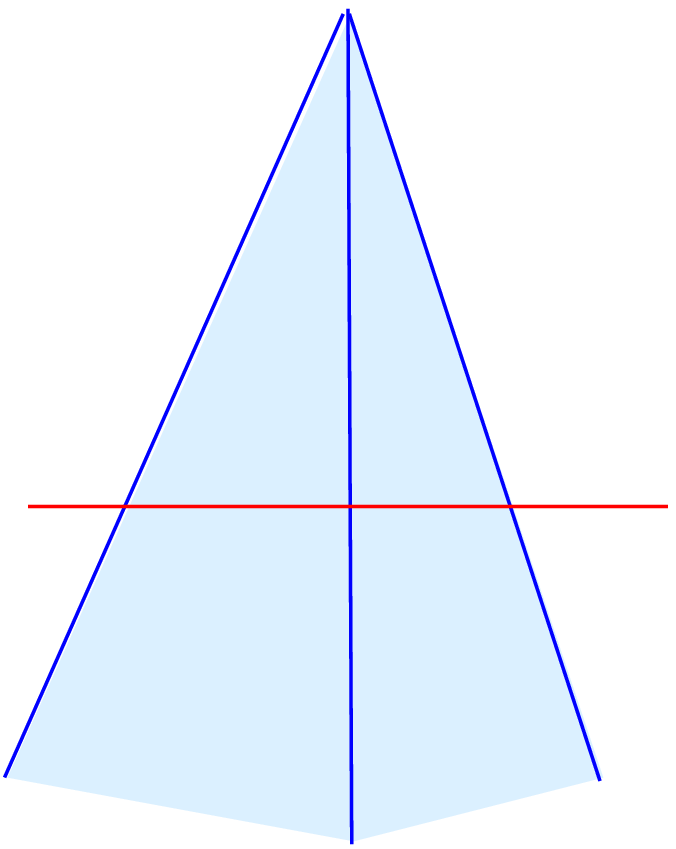,width=0.06\linewidth}
\quad \raisebox{9pt}{$\to$} \quad 
\raisebox{0pt}{
\begin{tikzpicture}[scale=.4]
\draw (0,0)--(2,0)--(2,1)--(0,1)--(0,0) (1,0)--(1,1)
(1,1)--(1,2)--(2,2)--(2,1);
\end{tikzpicture}
}
\qquad \qquad \qquad
\epsfig{file=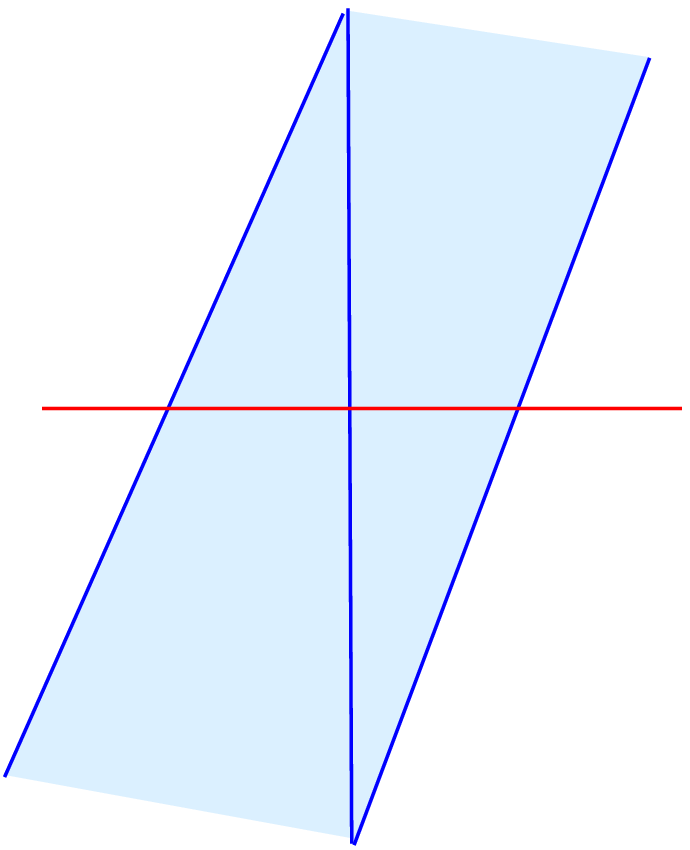,width=0.06\linewidth}
\quad  \raisebox{9pt}{$\to$} \quad
\raisebox{5pt}{
\begin{tikzpicture}[scale=.4]
\draw (4,.5)--(7,.5)--(7,1.5)--(4,1.5)--(4,.5) (5,.5)--(5,1.5) (6,.5)--(6,1.5);
\end{tikzpicture}
}
\caption{Crossing a fan in a surface gives rise to a zig-zag in the  snake graph and crossing a zig-zag in a surface gives rise to a straight piece in the snake graph.}
\label{Fig:Fan_vs_Zigzag}
\end{center}
\end{figure}

The next proposition follows directly from the construction of snake graphs.

\begin{proposition} 
\label{Prop:InfSG}
Let $T$ be a fan triangulation of $\cals$ and $\gamma$ be an arc. Then $\ga$ is either
\begin{itemize}
\item[(a)] a finite snake graph, or
\item[(b)] a finite union of finite and infinite one-sided zig-zag snake graphs.
\end{itemize}
\end{proposition}

\begin{definition}[Infinite staircase snake graph] Let $T$ be a fan triangulation and $\gamma$ be an arc such that $| \gamma\cap T | =\infty$. We will call the infinite snake graph $\ga$ associated with the arc $\gamma$,  the \emph{infinite staircase snake graph $\calg_{\gamma}$ associated with $\gamma$} (see Fig.~\ref{snake-example} and~\ref{Fig:MatchingNoMatching}~(d) for examples).
\end{definition}

We now extend the notion of sign function of \cite{CS} to infinite staircase snake graphs.

\begin{definition}[Sign function]\label{defn:signfunction} A \emph{sign function} $f$ of an infinite staircase snake graph $\calg$ is a map from the edges of $\calg$ to $\{+,-\}$ such that 
\begin{itemize}
\item[-] the bottom and the right edges of a tile $G_i$ have the same sign which is opposite to that of the top and the left edges;
\item[-] fixing a sign function on a single edge of a snake graph induces a sign function on the whole snake graph by extending with the rule $\raisebox{-2ex}{\begin{tikzpicture}[scale=.5]\draw (0,0)--(1,0)--(1,1)--(0,1)--(0,0);\node[scale=.8] at (.5,-.3){$\varepsilon$}; \node[scale=.8] at (.5,1.3){$-\varepsilon$}; \node[scale=.8] at (-.45,.5){$-\varepsilon$}; \node[scale=.8] at (1.3,.5){$\varepsilon$}; 
\end{tikzpicture}
}$ where $\varepsilon\in\{+,-\}$ and applying the limit rules when needed; in particular, the sign function on a limit tile is induced by the sign of the interior edges of the zig-zag approaching the limit tile;
\item[-] the sign function of a  limit tile induces a sign function in the infinite zig-zag following this tile.
\end{itemize}
\end{definition}

\begin{remark} 
\begin{itemize}
\item[(a)] In accordance with Remark~\ref{rem:orient}~(a), there is no way to determine whether we approach to a limit tile from the left or bottom. In terms of the sign function, this does not create any ambiguity since we associate the same sign on the left or the bottom of the limit tile as the sign in all interior edges in the infinite zig-zag approaching to this limit tile.
\item[(b)] A sign function of an infinite zig-zag snake graph $\calg$ has either all $+$ or all $-$ assigned to the interior edges of $\calg.$ 

\end{itemize}
\end{remark}

By construction,  two successive tiles of $\calg$ have different orientations relative to the orientation of the surface $\cals$. We will use this to define {\it edge positions} in the tiles relative to the orientation of the surface and to the sign function.

\begin{definition}[Orientation of a limit tile]
The {\it orientation of a limit tile} is induced by the orientation of the tiles in the zig-zag approaching it, i.e. the orientation of the limit tile coincides with the orientation of the tiles having the previous tile attached at the bottom edge, 
see Fig.~\ref{or}.
\end{definition}

\begin{figure}[!h]
\begin{center}
\epsfig{file=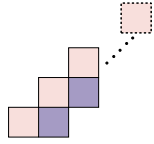,width=0.12\linewidth}
\caption{The induced orientation of the limit tile: two different shadings represent two different orientations of the tiles with respect to the surface.}
\label{or}
\end{center}
\end{figure} 

\begin{definition}[Edge positions in the tiles]\label{Def:EdgePos}
Let $G_i$ be a tile of an infinite staircase snake graph $\calg$ together with its  orientation induced  from $\calg$ and $f$ be a sign function of $\calg$. The  edges $a,b,c,d$ of $G_i$ are {\it in positions}
I, II, III, IV, respectively if 
\begin{itemize}
\item[-] $a,b,c,d$ follow each other in the direction of the orientation of $G_i$, and
\item[-] $a$ is the first edge of $G_i$ with sign ``$-$'' in $f(\calg)$ (when walked in the direction of the orientation).
\end{itemize}
See Fig.~\ref{orient}.
\end{definition}

\begin{figure}[!h]
\begin{center}
\psfrag{+}{\scriptsize\color{red}  $\mathbf +$}
\psfrag{-}{\scriptsize\color{red}  $\mathbf -$}
\psfrag{1}{\scriptsize I}
\psfrag{2}{\scriptsize II}
\psfrag{3}{\scriptsize III}
\psfrag{4}{\scriptsize IV}
\psfrag{a}{\small (a)}
\psfrag{b}{\small (b)}
\epsfig{file=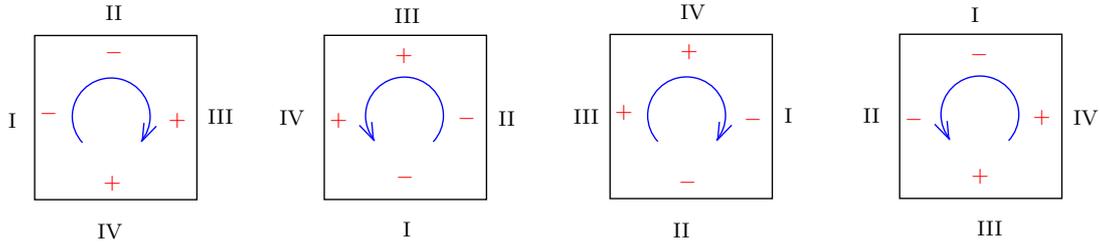,width=0.9\linewidth}
\caption{Edge positions defined depending on the orientation and the sign function. }
\label{orient}
\end{center}
\end{figure} 

\begin{convention}[On drawing snake graphs] 
\label{limit tile}
We always assume that the first tile is oriented clockwise and that the second tile of a snake graph  is attached to the first from the right.
\end{convention}

Next, we will extend the notion of perfect matchings  to infinite staircase snake graphs.

\begin{definition} 
\label{def:perfect matching}
A \emph{perfect matching} $P$ of an infinite staircase snake graph $\calg$ is a set of edges of $\calg$ such that each vertex in $\calg$ is incident to precisely one edge in $P$ and an edge of a limit tile is in the perfect matching if one of the following applies:
\begin{itemize}
\item[-] a limit tile $G_{lim}$ is followed by a regular tile $G_{\ast}$ and the edge between $G_{lim}$ and $G_{\ast}$ is in the matching of $G_{\ast}$;
\item[-] \emph{limit edge condition:} if $G_{lim}$ is a limit tile of the tiles $G_i$, for $i\in\bN$,  and there exists $k\in \bN$ such that an edge
of $G_i$ at the position $Y\in \{I,II,III,IV\}$ is in $P$ for all $i>k$, then the edge of $G_{lim}$ at the position $Y$ is also in $P$, see Example~\ref{Ex:NoMatching}.
\end{itemize}
\end{definition}

\begin{remark}
\label{Rem:LimTile} 
\begin{itemize}
\item[-] By definition, the vertices of a limit tile are either matched by the limit condition or by an edge of a successive tile (if such tile exists).  We cannot add any other edge of a limit tile to complete  a set of edges to a perfect matching of an infinite snake graph.
\item[-] By Observation~\ref{Obs:LimitTile} (b), a limit tile cannot be a two-sided limit which implies that two opposite \emph{boundary} edges of a limit tile are never in a matching of an infinite staircase snake graph, compare with Example~\ref{Ex:NoMatching}~(b).
\end{itemize} 

\end{remark}

\begin{example}
\label{Ex:NoMatching} 
\begin{itemize}
\item[(a)] The set of edges $P$ in Fig.~\ref{Fig:MatchingNoMatching}~(a) is a perfect matching of the snake graph $\calg$ since the top edges of every odd tile and the right edge of every even tile is in the matching which induces the top edge of the limit tile in the matching (since it has the same orientation as the odd tiles). Indeed, comparing to Fig.~\ref{orient}, we see that the edge at  position II is in the matching of $G_i$ for all $i$. Therefore, by the  limit edge condition the edge at  position II of the limit tile is also in the matching, i.e.  
the top edge of the limit tile is in the matching.  Hence,  every vertex of $\calg$ is incident to precisely one edge in $P$  and thus $P$ is a perfect matching. 
\item[(b)] (A non-perfect matching example) We consider the set of edges $P'$ consisting of the edge at position I of $G_i$ for each $i$ together with the edge at position I of the limit tile (obtained by the limit rule) as illustrated in  Fig.~\ref{Fig:MatchingNoMatching}~(b). Then the top-right vertex of the limit tile $G_{lim}$ is not incident to any edge in $P'$, and hence $P'$ is {\bf not} a perfect matching of $\calg.$
\item[(c)] The set of edges of $P''$ in Fig.~\ref{Fig:MatchingNoMatching}(c) is a perfect matching of the  bi-infinite zig-zag snake graph $\calg'$ since two vertices of each  limit tiles are matched by the limit condition and  other vertices of the limit tiles are matched by an edge in the adjacent tile $G_*$.
\item[(d)]  The set of edges of $P'''$ in Fig.~\ref{Fig:MatchingNoMatching}(d) forms a perfect matching of an infinite staircase snake graph.
\end{itemize} 
\end{example}

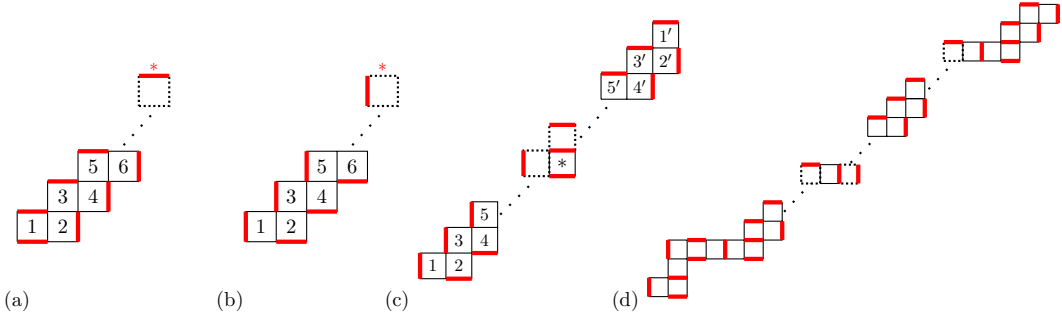
\begin{figure}[!h]
\centering
\begin{tikzpicture}
\node at (0,0){$\begin{tikzpicture}[scale=.4]
\draw (0,0)--(2,0)--(2,1)--(0,1)--(0,0) (1,0)--(1,1)
(1,1)--(1,2)--(2,2)--(2,1) (2,2)--(3,2)--(3,1)--(2,1) (2,2)--(2,3)--(4,3)--(4,2)--(3,2)--(3,3);
\draw[thick,loosely dotted] (3.5,3.2)--(4.5,4.3);
\draw[thick,densely dotted] (4,4.5)--(5,4.5)--(5,5.5)--(4,5.5)--(4,4.5);
\node[scale=.7] at (.5,.5){$1$};
\node[scale=.7] at (1.5,.5){$2$};
\node[scale=.7] at (1.5,1.5){$3$};
\node[scale=.7] at (2.5,1.5){$4$};
\node[scale=.7] at (2.5,2.5){$5$};
\node[scale=.7] at (3.5,2.5){$6$};
\draw[color=red,line width=1.6] (0,0)--(1,0) (0,1)--(1,1) (2,0)--(2,1) (1,2)--(2,2) (3,1)--(3,2) (2,3)--(3,3) (4,2)--(4,3) (4,5.5)--(5,5.5);
\node[scale=.6,color=red] at (4.5,5.8){$\ast$};
\end{tikzpicture}$};
\node at (3,0){$\begin{tikzpicture}[scale=.4]
\draw (0,0)--(2,0)--(2,1)--(0,1)--(0,0) (1,0)--(1,1)
(1,1)--(1,2)--(2,2)--(2,1) (2,2)--(3,2)--(3,1)--(2,1) (2,2)--(2,3)--(4,3)--(4,2)--(3,2)--(3,3);
\draw[thick,loosely dotted] (3.5,3.2)--(4.5,4.3);
\draw[thick,densely dotted] (4,4.5)--(5,4.5)--(5,5.5)--(4,5.5)--(4,4.5);
\node[scale=.7] at (.5,.5){$1$};
\node[scale=.7] at (1.5,.5){$2$};
\node[scale=.7] at (1.5,1.5){$3$};
\node[scale=.7] at (2.5,1.5){$4$};
\node[scale=.7] at (2.5,2.5){$5$};
\node[scale=.7] at (3.5,2.5){$6$};
\draw[color=red,line width=1.6] (0,0)--(0,1) (1,1)--(1,2) (2,2)--(2,3)  
(1,0)--(2,0) (2,1)--(3,1) (3,2)--(4,2) 
(4,4.5)--(4,5.5);
\node[scale=.6,color=red] at (4.5,5.8){$\ast$};
\end{tikzpicture}
$};
\node at (6,0){$\begin{tikzpicture}[baseline,scale=.34] 
\draw (1,0)--(1,-1)--(-1,-1)--(-1,0)--(0,0)--(0,-1) (0,0)--(2,0)--(2,1)--(0,1)--(0,0) (1,0)--(1,1)
(1,1)--(1,2)--(2,2)--(2,1); 
\draw[thick,loosely dotted] (2.1,1.5)--(3.5,2.9);
\draw[thick,densely dotted] (3,3)--(4,3)--(4,4)--(3,4)--(3,3);
\draw[thick,densely dotted] (5,4)--(5,5)--(4,5)--(4,4)--(5,4);
\draw (4,3)--(5,3)--(5,4)--(4,4)--(4,3);
\draw[thick,loosely dotted] (5.2,4.5)--(6.5,6);
\draw (6,6)--(8,6)--(8,7)--(6,7)--(6,6) (7,6)--(7,8)--(8,8)--(8,7)--(9,7)--(9,9)--(8,9)--(8,8)--(9,8); 
\draw[color=red,line width=1.6]  (-1,-1)--(-1,0) (0,-1)--(1,-1) (0,0)--(0,1) (1,0)--(2,0) (1,1)--(1,2)  (3,4)--(3,3) (4,3)--(5,3) (4,4)--(5,4) (4,5)--(5,5) (6,7)--(7,7) (8,6)--(8,7) (7,8)--(8,8) (9,7)--(9,8) (8,9)--(9,9);
\node[scale=.6] at (-.5,-.5){$1$};
\node[scale=.6] at (.5,-.5){$2$};
\node[scale=.6] at (.5,.5){$3$};
\node[scale=.6] at (1.5,.5){$4$};
\node[scale=.6] at (1.5,1.5){$5$};
\node[scale=.7] at (4.5,3.5){${\ast}$};
\node[scale=.6] at (6.5,6.5){$5'$};
\node[scale=.6] at (7.5,6.5){$4'$};
\node[scale=.6] at (7.5,7.5){$3'$};
\node[scale=.6] at (8.5,7.5){$2'$};
\node[scale=.6] at (8.5,8.5){$1'$};
\end{tikzpicture}
$};
\node at (10,0){$\begin{tikzpicture}[baseline,scale=.25] 

\draw  (-4,-1)--(-4,-3)--(-3,-3)--(-3,-1) (-3,-2)--(-5,-2)--(-5,-3)--(-4,-3) (-1,-1)--(-4,-1)--(-4,0)--(-1,0) (-2,-1)--(-2,0) (-3,-1)--(-3,0) (1,0)--(1,-1)--(-1,-1)--(-1,0)--(0,0)--(0,-1);
\draw (0,0)--(2,0)--(2,1)--(0,1)--(0,0) (1,0)--(1,1)
(1,1)--(1,2)--(2,2)--(2,1); 
\draw[thick,loosely dotted] (2.1,1.5)--(3.5,3);
\draw[thick,densely dotted] (3,3)--(4,3)--(4,4)--(3,4)--(3,3);
\draw[thick,densely dotted] (5,4)--(6,4)--(6,3)--(5,3);
\draw (4,3)--(5,3)--(5,4)--(4,4)--(4,3);
\draw[thick,loosely dotted] (5.6,4.1)--(7,5.5);
\draw (6.5,5.5)--(8.5,5.5)--(8.5,6.5)--(6.5,6.5)--(6.5,5.5) (7.5,5.5)--(7.5,7.5)--(8.5,7.5)--(8.5,6.5)
(8.5,7.5)--(9.5,7.5)--(9.5,6.5)--(8.5,6.5)
(8.5,7.5)--(8.5,8.5)--(9.5,8.5)--(9.5,7.5)--(8.5,7.5);
\draw[thick,loosely dotted] (9.7,8)--(11,9.4);
\draw[thick,densely dotted] (10.5,9.5)--(11.5,9.5)--(11.5,10.5)--(10.5,10.5)--(10.5,9.5);
\draw (11.5,10.5)--(14.5,10.5)--(14.5,9.5)--(11.5,9.5)--(11.5,10.5) (12.5,9.5)--(12.5,10.5) (13.5,9.5)--(13.5,10.5) (13.5,10.5)--(13.5,11.5)--(15.5,11.5)--(15.5,10.5)--(14.5,10.5)--(14.5,12.5)--(16.5,12.5)--(16.5,11.5)--(15.5,11.5)--(15.5,12.5);
\draw[color=red,line width=1.6] (-5,-3)--(-5,-2) (-4,-3)--(-3,-3) (-4,-2)--(-3,-2) (-4,-1)--(-4,0) (-3,-1)--(-2,-1) (-3,0)--(-2,0) (-1,0)--(-1,-1) (0,-1)--(1,-1) (0,0)--(1,0) (0,1)--(1,1) (2,0)--(2,1) (1,2)--(2,2) (3,4)--(4,4) (5,3)--(5,4) (6,3)--(6,4) (6.5,6.5)--(7.5,6.5) (8.5,5.5)--(8.5,6.5) (7.5,7.5)--(8.5,7.5) (9.5,6.5)--(9.5,7.5) (8.5,8.5)--(9.5,8.5)
(10.5,10.5)--(11.5,10.5) (12.5,10.5)--(12.5,9.5) (13.5,9.5)--(14.5,9.5) (13.5,10.5)--(14.5,10.5) 
(13.5,11.5)--(14.5,11.5) (15.5,11.5)--(15.5,10.5)
(15.5,12.5)--(14.5,12.5) (16.5,12.5)--(16.5,11.5);
\end{tikzpicture}$};
\node[scale=.7] at (-1,-2){(a)};
\node[scale=.7] at (1.8,-2){(b)};
\node[scale=.7] at (4,-2){(c)};
\node[scale=.7] at (7,-2){(d)};
\end{tikzpicture}
\caption{ (a), (c), (d)  are examples of perfect matchings; (b) is {\bf not} a perfect matching (see Example~\ref{Ex:NoMatching}).}
\label{Fig:MatchingNoMatching}
\end{figure}

\begin{notation} Given a snake graph $\calg$, 
\begin{itemize}
\item[-] we denote by $\match \calg$ the set of all perfect matchings of $\calg$;
\item[-] we  sometimes use the shorthand notation  $P \vdash \calg$ to indicate that  $P\in \match \calg$.
\end{itemize}
\end{notation}

\begin{lemma}\label{Lem:UniquePMexcept_1tile} 
Let $\calg=\cup_{i\in\bN} G_i$ be an infinite one-sided zig-zag snake graph (as in Fig.~\ref{or}) and let $P\in \match \calg$ be a perfect matching.  Then there exists a unique $i\in\bN$ such that $i$ is the first tile of $\calg$ with the property that the restriction $P|_{G_i}$ is a perfect matchings of $G_i.$  Moreover, the matching $P$ is uniquely determined by the matching $P|_{G_i}$ of $G_i.$
\end{lemma}

\begin{proof} Suppose that no tile of $\calg$ is completely matched. The left edge of the first tile $G_1$ must be in the matching because otherwise in order to cover the bottom-left vertex we would have to take the bottom edge in the matching of $G_1$ which would require to take the top edge giving rise to a complete matching on $G_1$. Similarly, this forces the bottom edge of $G_2$ to be in the matching.  Continuing in this manner successively, we construct the pattern in Fig.~\ref{Fig:MatchingNoMatching}~(b) which does not lead to a perfect matching as it is shown in  Example~\ref{Ex:NoMatching}~(b). Hence, there exists $i\in\bN$ such that $P|_{G_i}$ is a perfect matching of $G_i.$  
We can choose $i$ so that no tile $G_k$ is completely matched for $k<i$.

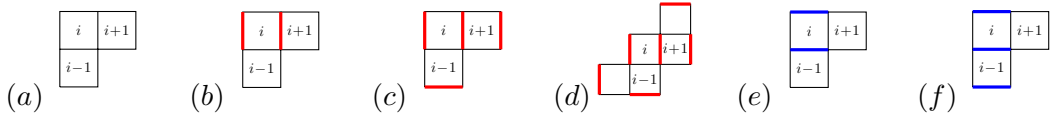
\begin{figure}[!h]
\centering
\begin{tikzpicture}[scale=.8]
\node at (0,0){$\begin{tikzpicture}[scale=.5]
\draw (0,0)--(1,0)--(1,2)--(0,2)--(0,0) (0,1)--(2,1)--(2,2)--(1,2);
\node[scale=.5] at (.5,.5){$i\!-\!1$};
\node[scale=.5] at (.5,1.5){$i$};
\node[scale=.5] at (1.5,1.5){$i\!+\!1$};
\end{tikzpicture}$};
\node at (3,0){$\begin{tikzpicture}[scale=.5]
\draw (0,0)--(1,0)--(1,2)--(0,2)--(0,0) (0,1)--(2,1)--(2,2)--(1,2);
\node[scale=.5] at (.5,.5){$i\!-\!1$};
\node[scale=.5] at (.5,1.5){$i$};
\node[scale=.5] at (1.5,1.5){$i\!+\!1$};
\draw[color=red, line width=1.2] (0,1)--(0,2) (1,1)--(1,2);
\end{tikzpicture}$};

\node at (6,0){$\begin{tikzpicture}[scale=.5]
\draw (0,0)--(1,0)--(1,2)--(0,2)--(0,0) (0,1)--(2,1)--(2,2)--(1,2);
\node[scale=.5] at (.5,.5){$i\!-\!1$};
\node[scale=.5] at (.5,1.5){$i$};
\node[scale=.5] at (1.5,1.5){$i\!+\!1$};
\draw[color=red, line width=1.2] (0,1)--(0,2) (1,1)--(1,2)
(0,0)--(1,0) (2,1)--(2,2);
\end{tikzpicture}$};

\node at (9,0){$\begin{tikzpicture}[scale=.4]
\draw  (0,0)--(-1,0)--(-1,1)--(0,1) (0,0)--(1,0)--(1,2)--(0,2)--(0,0) (0,1)--(2,1)--(2,2)--(1,2) (2,2)--(2,3)--(1,3)--(1,2);
\node[scale=.5] at (.5,.5){$i\!-\!1$};
\node[scale=.5] at (.5,1.5){$i$};
\node[scale=.5] at (1.5,1.5){$i\!+\!1$};
\draw[color=red, line width=1.2] (0,1)--(0,2) (1,1)--(1,2)
(0,0)--(1,0) (2,1)--(2,2) (-1,0)--(-1,1) (1,3)--(2,3);
\end{tikzpicture}$};

\node at (12,0){$\begin{tikzpicture}[scale=.5]
\draw (0,0)--(1,0)--(1,2)--(0,2)--(0,0) (0,1)--(2,1)--(2,2)--(1,2);
\node[scale=.5] at (.5,.5){$i\!-\!1$};
\node[scale=.5] at (.5,1.5){$i$};
\node[scale=.5] at (1.5,1.5){$i\!+\!1$};
\draw[color=blue, line width=1.2] (0,1)--(1,1) (0,2)--(1,2);
\end{tikzpicture}$};

\node at (15,0){$\begin{tikzpicture}[scale=.5]
\draw (0,0)--(1,0)--(1,2)--(0,2)--(0,0) (0,1)--(2,1)--(2,2)--(1,2);
\node[scale=.5] at (.5,.5){$i\!-\!1$};
\node[scale=.5] at (.5,1.5){$i$};
\node[scale=.5] at (1.5,1.5){$i\!+\!1$};
\draw[color=blue, line width=1.2] (0,0)--(1,0) (0,1)--(1,1) (0,2)--(1,2);
\end{tikzpicture}$};

\node at (-1.2,-0.8){$(a)$};
\node at (1.8,-0.8){$(b)$};
\node at (4.8,-0.8){$(c)$};
\node at (7.8,-0.8){$(d)$};
\node at (10.8,-0.8){$(e)$};
\node at (13.8,-0.8){$(f)$};
\end{tikzpicture}
\caption{Corner snake graph.}
\label{Fig:corner_SG}
\end{figure}

Suppose that $i>1$.
Without loss of generality, suppose $G_i$ is placed on the top-left of the corner piece $(G_{i-1}, G_i,G_{i+1}),$ see Fig.~\ref{Fig:corner_SG}~(a) (bottom-right case may by considered similarly). Then the matching $P$ contains either the two vertical or the two horizontal edges of $G_i$ as in  Fig.~\ref{Fig:corner_SG}~(b) and (e). First, assume that the two vertical edges of $G_i$ are in the matching. This enforces to have the bottom edge of the tile $G_{i-1}$ and the right edge of $G_{i+1}$ in the matching, see Fig.~\ref{Fig:corner_SG}~(c). The matching on $G_{i-1}$ enforces the left edge of $G_{i-2}$ and the matching on $G_{i+1}$ enforces the top edge of $G_{i+2},$ see Fig.~\ref{Fig:corner_SG}~(d). Continuing in this manner, all the bottom edges of the tiles $G_{i-2k+1}$  and all the left edges of the tiles $G_{i-2k}$ for $1\leq k\leq \left \lfloor{\frac{i}{2}}\right \rfloor $ are in the matching of $\calg.$ Similarly, all the right edges of the tiles $G_{i+2j-1}$  and all the top edges of the tiles $G_{i+2j}$ for $ j\in\bN$ are in the matching of $\calg.$ Moreover, the limit tile $G_{lim}$ have the top edge in the matching by the limit condition, hence this is a perfect matching of $\calg$ which is uniquely determined by a matching on the tile $G_i.$

If $G_i$ has two horizontal edges in the matching (see Fig.~\ref{Fig:corner_SG}~(f)), then this forces the bottom of the tile $G_{i-1}$ in the matching, see Fig.~\ref{Fig:corner_SG}~(f). But then $P|_{G_{i-1}}$  is a matching of $G_{i-1}$ and  $i$ is not the first position with a complete matching, hence we argue as above for $i-1$.  

In the case of $i=1$, each of two possible perfect matchings of $G_i=G_1$ (see Fig.~\ref{Fig:Height_01}) extends uniquely to a perfect matching of $\calg$.
\end{proof} 

\begin{figure}[!h]
\centering
\begin{tikzpicture}[scale=.7]
\node at (0,0){$\begin{tikzpicture}[scale=.5]
\draw (0,0)--(2,0)--(2,2)--(1,2)--(1,0) (0,0)--(0,1)--(2,1);
\node[scale=.6] at (.5,.5){$1$};
\node[scale=.6] at (1.5,.5){$2$};
\node[scale=.6] at (1.5,1.5){$3$};
\draw[thick,loosely dotted] (2.1,1.5)--(3.5,2.9);
\draw[color=red, line width=1.4] (0,0)--(1,0) (0,1)--(1,1) (2,0)--(2,1)  (1,2)--(2,2);
\end{tikzpicture}$};
\node at (5,0){$\begin{tikzpicture}[scale=.5]
\draw (0,0)--(2,0)--(2,2)--(1,2)--(1,0) (0,0)--(0,1)--(2,1);
\node[scale=.6] at (.5,.5){$1$};
\node[scale=.6] at (1.5,.5){$2$};
\node[scale=.6] at (1.5,1.5){$3$};
\draw[thick,loosely dotted] (2.1,1.5)--(3.5,2.9);
\draw[color=red, line width=1.4] (0,0)--(0,1) (1,0)--(1,1) (2,0)--(2,1) (2,2)--(1,2);
\end{tikzpicture}$};
\node at (-1.8,-1.2){$(a)$};
\node at (3.2,-1.2){$(b)$};
\end{tikzpicture}
\caption{Initial segments of two perfect matchings of infinite one-sided zig-zag snake graph with two complete matchings of the first tile $G_1$. }
\label{Fig:Height_01}
\end{figure}
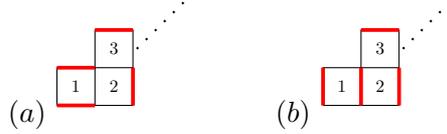

\begin{definition}[Height of a perfect matching] \label{Def:height}
 Let $\calg$ be a one-sided zig-zag snake graph, let $P$ be a matching of $\calg$, and let $i$ be the first position of a tile of $\calg$ such that $P|_{G_i}$ is a perfect matching of $G_i$ (this tile exists and is unique by Lemma~\ref{Lem:UniquePMexcept_1tile}). 
If $i>1$, we say that $P$ is a perfect matching of {\it height} $i$.
If $i=1$,  
we will associate {\it height $0$}  and  {\it height $1$} to the perfect matchings  of the form shown in Fig.~\ref{Fig:Height_01}~(a) and~(b), respectively.
\end{definition}

\subsection{Laurent series associated with snake graphs} 

\cite{MSW} introduced an explicit formula for cluster variables of surface cluster algebras which is parametrised by perfect matchings of snake graphs associated with arcs in a finite surface. In this section, we aim to generalise this formula for infinite surfaces with fan triangulations.

\begin{definition}[Laurent series associated with $\calg_\gamma$]
Let $\cals$ be an infinite surface and $T$ be a fan triangulation of $\cals$.
Let $\gamma$ be an arc in  $\cals$ and $\calg_{\gamma}$ its snake graph. Let $P$ be a perfect matching of $\calg_\gamma$. Then
\begin{itemize}
\item[-] assign a formal variable $x_{i}$ to every arc $\tau_i$ (including the limit arcs and the boundary arcs) of the triangulation $T$;
\item[-]  the (possibly infinite) {\it weight monomial} $x(P)$ of 
$P$ is given by
 $$x(P)=\prod\limits_{e_i\in P}x_{e_i},$$
where $x_{e_i}=x_j$ for an edge $e_i\in P$ with weight $\tau_j$; 
\item[-]
the {\it crossing monomial } of $\gamma$ with $T$ is given by $$\cross(\gamma,T)=\prod\limits_{j\in J}x_{j},$$  where $J$ is the index set of the tiles in $\calg_\gamma$, i.e. 
$\ga=\displaystyle\cup_{j\in J} G_j$ and 
$$
x_j=\begin{cases}
x_k & \text{if $G_j$ has weight $\tau_k$,}\\
1&  \text{if $G_j$ is a limit tile;}
\end{cases}
$$

\item[-] the  \emph{Laurent series associated with $\ga$} is defined by
\begin{equation*} 
x_{\calg_{\gamma}}=\sum\limits_{P \in \match \calg_{\gamma}}  \displaystyle \frac{x(P)}{\cross (\gamma, T)},
\end{equation*}
where $x(P)$ is the weight monomial of $P$ and $\cross(\gamma,T)$ is the crossing monomial of $\gamma$ with $T$.
\end{itemize}
\end{definition}

As in the finite case, $x_{\calg}$ gives an explicit formula for cluster variables. We will first show this for arcs lying inside a single fan.

\begin{proposition}\label{Thm:MSWforInfZigzag}
Let $\calf=\{\gamma_i,\gamma_{i,i+1}\mid i\in\bZ_{\geq 0}\}\cup\{\gamma_{\ast}\}$ be an infinite incoming elementary fan and $\gamma$ be the arc which crosses $\gamma_i$ for $i\in\bN,$ see Fig.~\ref{Fig:PM_pi}. 
 Let $x_\gamma$ be the cluster variable associated with $\gamma$ and $\calg_{\gamma}$ be the snake graph of $\gamma$. Then $x_{\gamma}=x_{\calg_{\gamma}}.$ 
\end{proposition} 

\begin{proof} The snake graph $\calg_\gamma$ associated with $\gamma$ is an infinite one-sided zig-zag.
By Definition~\ref{Def:height}, perfect matchings of $\ga$ are indexed by $\bZ_{\geq 0}$ corresponding to the height  $i$ of a perfect matching $P_i$. We will show 
\begin{align*}
\frac{x(P_i)}{\cross{(\gamma,T)}}=\frac{x_0x_{\ast}x_{i,i+1}}{x_{i}x_{i+1}},
\end{align*}
for $i\in \bZ_{\ge 0}$, which will imply
\begin{align*}
x_{\ga}&=x_0x_{\ast}\sum\limits_{i\in\bZ_{\geq 0}}\frac{x_{i,i+1}}{x_{i}x_{i+1}}=x_{\gamma},
\end{align*}
where the last equality follows from Proposition~\ref{Rem:IncomingFan}~(4). To show the claim, consider the perfect matching $P_i$ given in  Fig.~\ref{Fig:PM_pi}. By definition, we have 

\begin{align*}
\frac{x(P_i)}{\cross{(\gamma,T)}}&=\frac{x_0\crossout{x_1}\crossout{x_2}\ldots \crossout{x_{i\!-\!1}}x_{i,i\!+\!1}\crossout{x_{i\!+\!2}}\ldots x_{\ast}}{\crossout{x_1}\crossout{x_2}\ldots\crossout{x_{i\!-\!1}} x_{i}x_{i\!+\!1}\crossout{x_{i\!+\!2}}\ldots} \\
&=\frac{x_0x_{\ast}x_{i,i+1}}{x_{i}x_{i+1}},
\end{align*}
as required. 
\end{proof}

\begin{figure}[!h]
\centering
\psfrag{g}{\scriptsize \color{red} $x_\gamma$}
\psfrag{0}{\scriptsize \color{blue} $x_0$}
\psfrag{1}{\scriptsize \color{blue} $x_1$}
\psfrag{2}{\scriptsize \color{blue} $x_2$}
\psfrag{3}{\scriptsize \color{blue} $x_3$}
\psfrag{*}{\scriptsize \color{blue} $x_*$}
\psfrag{12}{\scriptsize \color{blue} $x_{0,1}$}
\psfrag{23}{\scriptsize \color{blue} $x_{1,2}$}
\psfrag{34}{\scriptsize \color{blue} $x_{2,3}$}
\raisebox{5pt}{
\epsfig{file=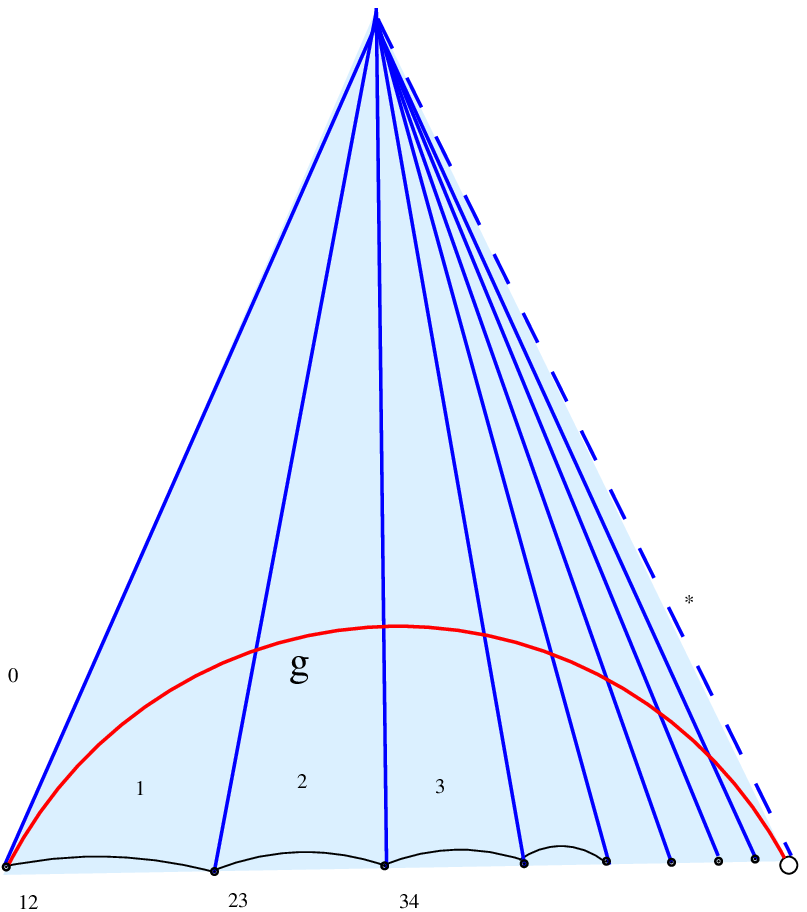,width=0.3\linewidth}
}
\qquad \qquad 
\begin{tikzpicture}[scale=.6]
\draw (0,0)--(2,0)--(2,1)--(0,1)--(0,0) (1,0)--(1,1)
(1,1)--(1,2)--(2,2)--(2,1);
\draw[thick,loosely dotted] (2.2,1.5)--(3.5,2.9);
\draw[thick,loosely dotted] (6.2,4.5)--(7.5,5.9);
\draw (3,3)--(5,3)--(5,5)--(4,5)--(4,3) (3,3)--(3,4)--(6,4) (5,5)--(6,5);
\draw[thick,densely dotted] (7,6)--(8,6)--(8,7)--(7,7)--(7,6);
\node[scale=.7] at (.5,.5){$1$};
\node[scale=.7] at (1.5,.5){$2$};
\node[scale=.7] at (1.5,1.5){$3$};
\node[scale=.6] at (3.5,3.5){$i\!-\!1$};
\node[scale=.7] at (4.5,3.5){$i$};
\node[scale=.6] at (4.5,4.5){$i\!+\!1$};
\node[scale=.6] at (5.5,4.5){$i\!+\!2$};
\draw[color=red,line width=1.6] (0,0)--(0,1) (1,0)--(2,0) (1,1)--(1,2) (3,3)--(3,4) (4,3)--(5,3) (4,4)--(5,4) (4,5)--(5,5) (7,7)--(8,7) (6,4)--(6,5);
\node[scale=.6,color=red] at (-.3,.5){$0$};
\node[scale=.6,color=red] at (1.5,-.3){$1$};
\node[scale=.6,color=red] at (.7,1.5){$2$};
\node[scale=.55,color=red] at (2.5,3.5){$i\!-\!2$};
\node[scale=.55,color=red] at (4.5,2.7){$i\!-\!1$};
\node[scale=.55,color=red] at (4.5,5.2){$i\!+\!2$};
\node[scale=.55,color=red] at (5.9,3.6){$(i,i\!+\!1)$};
\node[scale=.55,color=red] at (6.6,4.3){$i\!+\!3$};
\draw[color=red] (4.65,3.9)--(5.2,3.6); 
\node[scale=.6,color=red] at (7.5,7.2){$\ast$};
\end{tikzpicture}
\caption{Arc $\gamma$ inside an incoming fan (left) and the corresponding one-sided infinite snake graph $\calg_\gamma$ (right). $\calg_\gamma$ is given with a perfect matching  of height $i$.}
\label{Fig:PM_pi}
\end{figure}
We are now ready to present the main theorem of this section.

\begin{theorem}
\label{Thm:MSW-main}
Let $T$ be a fan triangulation of an infinite unpunctured surface $\cals$ and let $\gamma$ be an arc. Let $x_{\gamma}$ be the cluster variable associated with $\gamma$ and $\ga$ be the snake graph of $\gamma.$ Then $x_{\gamma}=x_{\ga}.$
\end{theorem}

\begin{proof} 
Abusing the notation, we will denote by $\gamma$ the lift of an arc $\gamma$ to the universal cover. Let $\cald_\gamma^T$ be the domain of $\gamma$ in the universal cover.
By Remark~\ref{fan is elementary}, the domain $\cald_{\gamma}^T$  is a finite union of elementary fans. Similarly to the proof of Theorem~\ref{Laurent for punctures}, we will argue by induction on the number of  elementary fans in $\cald_\gamma^T$.

For the base  of the induction, we will consider the case when  $\cald_\gamma^T$ is a fan. Then it is either a finite fan or an infinite elementary fan. In the first case, the snake graph is finite, hence the statement follows from~\cite{MSW}. In the latter, the snake graph is a one-sided infinite zig-zag and  $x_{\gamma}=x_{\ga}$ by Proposition~\ref{Thm:MSWforInfZigzag}.

Assume that $x_{\calg_{\delta}}=x_{\delta}$ for every arc $\delta$ crossing strictly less than $k$ fans. Now suppose $\gamma$ is an arc crossing $k$ fans in the universal cover. We will denote these fans by $\calf^{(i)}$ for $1\leq i \leq k.$ Fix an orientation on $\gamma$ and let $\gamma_{\ast}$ be the arc separating the fans $\calf^{(k-1)}$ and $\calf^{(k)}$. Note that $\gamma_{\ast}$ might be a limit arc. We have one of the following main configurations in the neighbourhood of $\gamma_{\ast}$ in $\cald_{\gamma}$  (see  also Fig.~\ref{Fig:MSW-main}):

\begin{itemize}
\item[(I)] the fans $\calf^{(k-1)}$ and $\calf^{(k)}$ have sources at the same end of $\gamma_{\ast}$;

\item[(II)] the fans $\calf^{(k-1)}$ and $\calf^{(k)}$ have  sources at opposite ends of $\gamma_{\ast}.$
\end{itemize}

\begin{figure}[!h]
\begin{center}
\psfrag{g}{{\color{red} $\gamma$}}
\psfrag{a}{\small (I)}
\psfrag{b}{\small (II)}
\epsfig{file=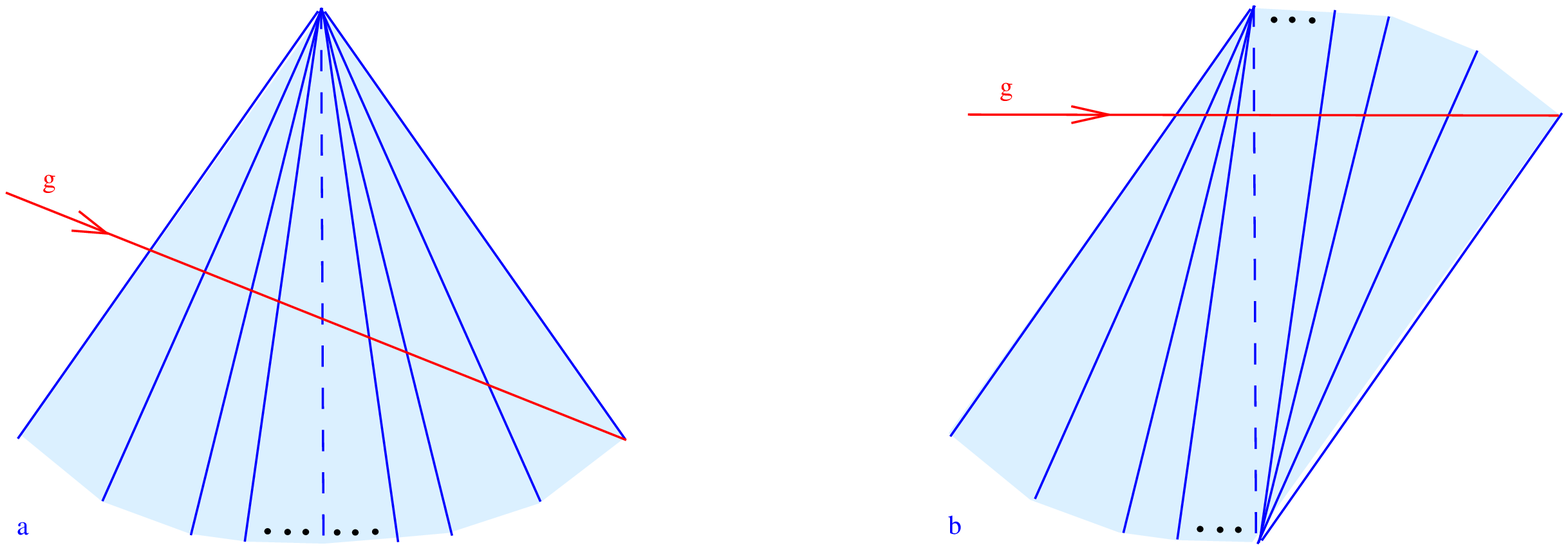,width=0.8\linewidth}
\caption{Two principle ways to change an elementary fan (each of the
two fans in the figures may be finite or infinite).}
\label{Fig:MSW-main}\end{center}
\end{figure}

\bigskip

\noindent
{\textbf{Case (I):}} Suppose that the last change of fan in  $\cald_\gamma$ is of type (I). We will also assume that both of the fans $\calf_{k-1}$ and $\calf_{k}$ are infinite. Then the snake graph of $\gamma$ is of the form given in Fig.~\ref{Fig:MSW-type(I)}. Denote the arcs in   $\cald_\gamma$  as shown in the figure,
denote also the snake subgraph corresponding to the arc $\overline{\gamma}$ by $\overline{\calg}$ and the snake subgraph corresponding to the crossing of $\gamma$ with the fan $\calf_{i}$ by $\calg_{i}$. 

\begin{figure}[!h]
\psfrag{g}{{\color{red} $\gamma$}}
\psfrag{bg}{{\color{dgreen}  $\bar \gamma$}}
\psfrag{b1}{{\color{dgreen}  $\beta_1$}}
\psfrag{b2}{{\color{dgreen}  $\beta_2$}}
\psfrag{0'}{{\color{dgreen}  $0'$}}
\psfrag{1'}{{\color{blue} \scriptsize $1'$}}
\psfrag{2'}{{\color{blue}\scriptsize  $2'$}}
\psfrag{1}{{\color{blue}\scriptsize  $1$}}
\psfrag{2}{{\color{blue}\scriptsize  $2$}}
\psfrag{3}{{\color{blue}\scriptsize  $3$}}
\psfrag{*}{{\color{magenta} \large $*$}}
\epsfig{file=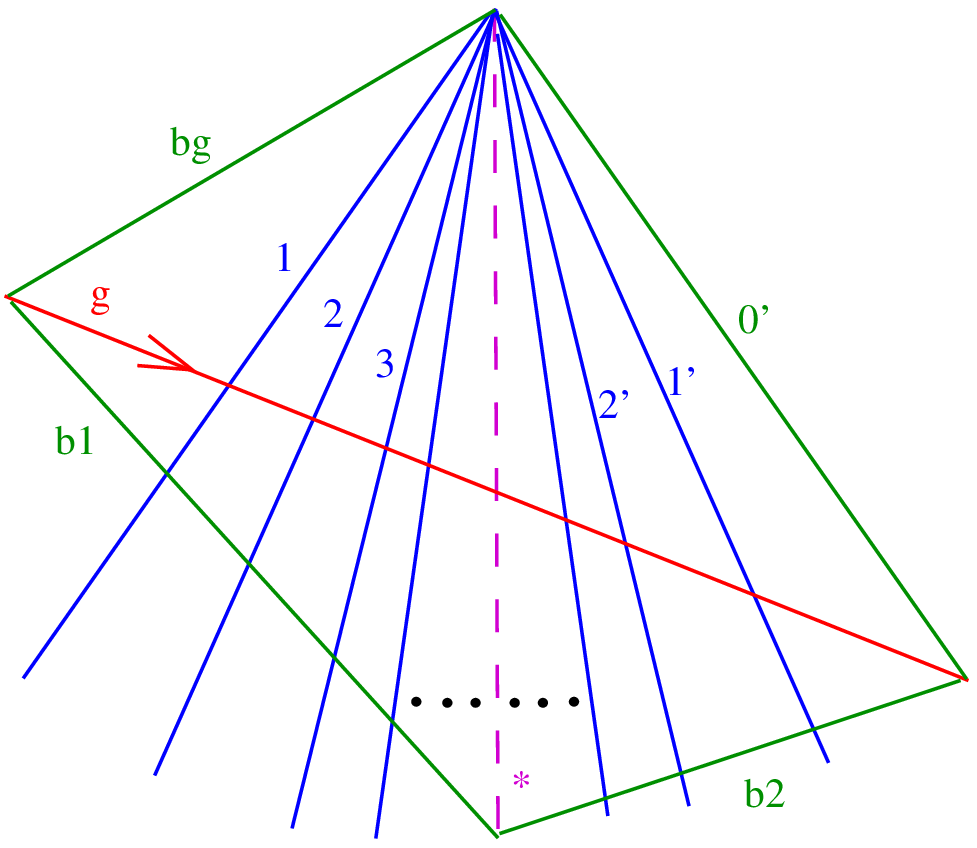,width=0.5\linewidth}
\hspace{.5in}
\raisebox{+6.5ex}{
\begin{tikzpicture}[baseline,scale=.4] 
\draw (0,0)--(2,0)--(2,1)--(0,1)--(0,0) (1,0)--(1,1)
(1,1)--(1,2)--(2,2)--(2,1) (0,0)--(-1,0) (0,1)--(-1,1); 
\draw  plot[smooth, tension=.7] coordinates {(-1,1) (-1.5,.8) (-2,.6) (-2.5,.3) (-3,-.5) (-2.7,-2) (-1.5,-1) (-1,0)};
\node[scale=.8] at (-1.8,-.4){$\overline{\calg}$};
\draw[thick,loosely dotted] (2,1.5)--(3.5,3);
\draw[thick,densely dotted] (3,3)--(4,3)--(4,4)--(3,4)--(3,3);
\draw[thick,densely dotted] (5,4)--(5,5)--(4,5)--(4,4)--(5,4);
\draw (4,3)--(5,3)--(5,4)--(4,4)--(4,3);
\draw[thick,loosely dotted] (5.2,4.5)--(6.5,6);
\draw (6,6)--(8,6)--(8,7)--(6,7)--(6,6) (7,6)--(7,8)--(8,8)--(8,7); 
\node[scale=.7] at (.5,.5){$1$};
\node[scale=.7] at (1.5,.5){$2$};
\node[scale=.7] at (1.5,1.5){$3$};
\node[scale=.8] at (4.5,3.5){${\ast}$};
\node[scale=.7] at (6.5,6.5){$3'$};
\node[scale=.7] at (7.5,6.5){$2'$};
\node[scale=.7] at (7.5,7.5){$1'$};
\end{tikzpicture}
}
\caption{Snake graph associated with $\gamma$ in Case (I).}
\label{Fig:MSW-type(I)}
\end{figure}

The snake graphs $\calg_{k-1}$ and $\calg_{k}$ are one-sided infinite zig-zag snake graphs since they are associated with infinite zig-zag fans $\calf_{k-1}$ and $\calf_{k}$.   Note that since the fans $\calf_{k-1}$ and $\calf_{k}$ have the same source, the corresponding glueing in $\ga$ between $\calf_{k-1},\gamma_{\ast}$ and $\calf_{k}$ is a zig-zag and $\calg=\overline{\calg} \cup \calg_{k-1} \cup G_{\ast} \cup \calg_{k}$.

By Ptolemy relations, we have the identity
\begin{equation} 
\label{eq1}
x_{\gamma}x_{\ast}=x_{\beta_1}x_{0'}+x_{\beta_2}x_{\overline{\gamma}} 
\end{equation}
\noindent 
in terms of lambda lengths of the arcs $\gamma, \gamma_{\ast}, \beta_1,\gamma_{0'},\beta_2,$ and $\overline{\gamma}$, see Fig.~\ref{Fig:MSW-type(I)}. We will show that the corresponding identity

\begin{align} 
\label{eq5}
x_{\ga}x_{\calg_{\ast}}=x_{\calg_{\beta_1}}x_{\calg_{0'}}+x_{\calg_{\beta_2}}x_{\calg_{\overline{\gamma}}} 
\end{align}
\noindent 
in terms of  expansion formula for the associated snake graphs holds. Observe that we have $x_{\beta_i}=x_{\calg_{\beta_i}},$ for $i=1,2$, $x_{\overline{\gamma}}=x_{\calg_{\overline{\gamma}}},$  $x_*=x_{\calg_*}$ and $x_{0'}=x_{\calg_{0'}}$  by the induction hypothesis. 

In order to show the identity~(\ref{eq5}), we will reorganise the set of perfect matchings of $\ga$ by considering those that have the vertical boundary edge  $a$
 or the horizontal boundary edge  $b$ of the tile $G_{\ast}$ in a matching of $\ga$: 
\begin{align} \match \ga  
&= \match \left(
\raisebox{-5.5ex}{\scalebox{.8}{
\begin{tikzpicture}[baseline,scale=.4] 
\draw (0,0)--(2,0)--(2,1)--(0,1)--(0,0) (1,0)--(1,1)
(1,1)--(1,2)--(2,2)--(2,1) (0,0)--(-1,0) (0,1)--(-1,1); 
\draw  plot[smooth, tension=.7] coordinates {(-1,1) (-1.5,.8) (-2,.6) (-2.5,.3) (-3,-.5) (-2.7,-2) (-1.5,-1) (-1,0)};
\node[scale=.8] at (-1.8,-.4){$\overline{\calg}$};
\draw[thick,loosely dotted] (2,1.5)--(3.5,3);
\draw[thick,densely dotted] (3,3)--(4,3)--(4,4)--(3,4)--(3,3);
\draw[thick,densely dotted] (5,4)--(5,5)--(4,5)--(4,4)--(5,4);
\filldraw[color=beaublue] (4,3)--(5,3)--(5,4)--(4,4)--(4,3);
\draw[thick,loosely dotted] (5.2,4.5)--(6.5,6);
\draw (6,6)--(8,6)--(8,7)--(6,7)--(6,6) (7,6)--(7,8)--(8,8)--(8,7); 
\draw[color=red,line width=2.4] (5,3)--(5,4);
\node[scale=.7,color=red] at (5.4,3.5){$a$};
\node[scale=.7] at (.5,.5){$1$};
\node[scale=.7] at (1.5,.5){$2$};
\node[scale=.7] at (1.5,1.5){$3$};
\node[scale=.8] at (4.5,3.5){${\ast}$};
\node[scale=.7] at (6.5,6.5){$3'$};
\node[scale=.7] at (7.5,6.5){$2'$};
\node[scale=.7] at (7.5,7.5){$1'$};
\end{tikzpicture} }
}
\right)
\cup
\match \left(
\raisebox{-5.5ex}{\scalebox{.8}{
\begin{tikzpicture}[baseline,scale=.4] 
\draw (0,0)--(2,0)--(2,1)--(0,1)--(0,0) (1,0)--(1,1)
(1,1)--(1,2)--(2,2)--(2,1) (0,0)--(-1,0) (0,1)--(-1,1); 
\draw  plot[smooth, tension=.7] coordinates {(-1,1) (-1.5,.8) (-2,.6) (-2.5,.3) (-3,-.5) (-2.7,-2) (-1.5,-1) (-1,0)};
\node[scale=.8] at (-1.8,-.4){$\overline{\calg}$};
\draw[thick,loosely dotted] (2,1.5)--(3.5,3);
\draw[thick,densely dotted] (3,3)--(4,3)--(4,4)--(3,4)--(3,3);
\draw[thick,densely dotted] (5,4)--(5,5)--(4,5)--(4,4)--(5,4);
\filldraw[color=beaublue] (4,3)--(5,3)--(5,4)--(4,4)--(4,3);
\draw[thick,loosely dotted] (5.2,4.5)--(6.5,6);
\draw (6,6)--(8,6)--(8,7)--(6,7)--(6,6) (7,6)--(7,8)--(8,8)--(8,7); 
\draw[color=red,line width=2.4] (4,3)--(5,3);
\node[scale=.7,color=red] at (4.5,2.5){$b$};
\node[scale=.7] at (.5,.5){$1$};
\node[scale=.7] at (1.5,.5){$2$};
\node[scale=.7] at (1.5,1.5){$3$};
\node[scale=.8] at (4.5,3.5){${\ast}$};
\node[scale=.7] at (6.5,6.5){$3'$};
\node[scale=.7] at (7.5,6.5){$2'$};
\node[scale=.7] at (7.5,7.5){$1'$};
\end{tikzpicture}}
}
\right) \nonumber
\\  
&=\match \left(
\raisebox{-5.5ex}{\scalebox{.75}{
\begin{tikzpicture}[baseline,scale=.4] 
\draw (0,0)--(2,0)--(2,1)--(0,1)--(0,0) (1,0)--(1,1)
(1,1)--(1,2)--(2,2)--(2,1) (0,0)--(-1,0) (0,1)--(-1,1); 
\draw  plot[smooth, tension=.7] coordinates {(-1,1) (-1.5,.8) (-2,.6) (-2.5,.3) (-3,-.5) (-2.7,-2) (-1.5,-1) (-1,0)};
\node[scale=.8] at (-1.8,-.4){$\overline{\calg}$};
\draw[thick,loosely dotted] (2,1.5)--(3.5,3);
\draw[thick,densely dotted] (3,3)--(4,3)--(4,4)--(3,4)--(3,3);
\draw[thick,densely dotted] (5,4)--(5,5)--(4,5)--(4,4)--(5,4);
\filldraw[color=beaublue] (4,3)--(5,3)--(5,4)--(4,4)--(4,3);
\draw[thick,loosely dotted] (5.2,4.5)--(6.5,6);
\draw (6,6)--(8,6)--(8,7)--(6,7)--(6,6) (7,6)--(7,8)--(8,8)--(8,7); 
\draw[color=red,line width=2.4] (5,3)--(5,4) (5,5)--(4,5) (6,7)--(7,7) (8,6)--(8,7) (7,8)--(8,8);
\node[scale=.7] at (.5,.5){$1$};
\node[scale=.7] at (1.5,.5){$2$};
\node[scale=.7] at (1.5,1.5){$3$};
\node[scale=.8] at (4.5,3.5){${\ast}$};
\node[scale=.7] at (6.5,6.5){$3'$};
\node[scale=.7] at (7.5,6.5){$2'$};
\node[scale=.7] at (7.5,7.5){$1'$};
\node[scale=.6,color=red] at (7.5,8.4){$0'$};
\node[scale=.6,color=red] at (8.4,6.5){$1'$};
\node[scale=.6,color=red] at (6.5,7.4){$2'$};
\draw[color=red,decorate,decoration={brace,amplitude=18pt},xshift=-4pt,yshift=0pt]
 (4.2,2.3) --(-2,-2.6)  node [color=red,midway,xshift=.6cm,yshift=-.7cm] 
{\footnotesize $\calg_{\beta_1}$};
\draw[color=red,decorate,decoration={brace,amplitude=18pt},xshift=-4pt,yshift=0pt]
 (9.2,8)-- (4.9,2.1) node [color=red,midway,xshift=.6cm,yshift=-.6cm] 
{\footnotesize $\widetilde{P}$};
\end{tikzpicture}}
}
\right)
\cup
\match \left(
\raisebox{-5.5ex}{\scalebox{.75}{
\begin{tikzpicture}[baseline,scale=.4] 
\draw (0,0)--(2,0)--(2,1)--(0,1)--(0,0) (1,0)--(1,1)
(1,1)--(1,2)--(2,2)--(2,1) (0,0)--(-1,0) (0,1)--(-1,1); 
\draw  plot[smooth, tension=.7] coordinates {(-1,1) (-1.5,.8) (-2,.6) (-2.5,.3) (-3,-.5) (-2.7,-2) (-1.5,-1) (-1,0)};
\node[scale=.8] at (-1.8,-.4){$\overline{\calg}$};
\draw[thick,loosely dotted] (2,1.5)--(3.5,3);
\draw[thick,densely dotted] (3,3)--(4,3)--(4,4)--(3,4)--(3,3);
\draw[thick,densely dotted] (5,4)--(5,5)--(4,5)--(4,4)--(5,4);
\filldraw[color=beaublue] (4,3)--(5,3)--(5,4)--(4,4)--(4,3);
\draw[thick,loosely dotted] (5.2,4.5)--(6.5,6);
\draw (6,6)--(8,6)--(8,7)--(6,7)--(6,6) (7,6)--(7,8)--(8,8)--(8,7); 
\draw[color=red,line width=2.4] (1,0)--(2,0) (1,1)--(1,2) (3,3)--(3,4) (4,3)--(5,3);
\node[scale=.7] at (.5,.5){$1$};
\node[scale=.7] at (1.5,.5){$2$};
\node[scale=.7] at (1.5,1.5){$3$};
\node[scale=.8] at (4.5,3.5){${\ast}$};
\node[scale=.7] at (6.5,6.5){$3'$};
\node[scale=.7] at (7.5,6.5){$2'$};
\node[scale=.7] at (7.5,7.5){$1'$};
\node[scale=.6,color=red] at (1.5,-.4){$1$};
\node[scale=.6,color=red] at (.6,1.5){$2$};
\draw[color=red,decorate,decoration={brace,amplitude=8pt},xshift=-4pt,yshift=0pt]
 (0,-.4)--(-2.4,-2.4) node [color=red,midway,xshift=.4cm,yshift=-.4cm] 
{\footnotesize $\calg_{\overline{\gamma}}$};
\draw[color=red,decorate,decoration={brace,amplitude=18pt},xshift=-4pt,yshift=0pt]
 (5.2,2.8) --(0.8,-1)  node [color=red,midway,xshift=.6cm,yshift=-.6cm] 
{\footnotesize $\widetilde{\widetilde P}$};
\draw[color=red,decorate,decoration={brace,amplitude=18pt},xshift=-4pt,yshift=0pt]
 (9.2,8)-- (5.4,4) node [color=red,midway,xshift=.7cm,yshift=-.7cm] 
{\footnotesize $\calg_{\beta_2}$};
\end{tikzpicture} }
}
\right) \nonumber
\\
&= (\match \calg_{\beta_1} \sqcup \widetilde{P}) \cup (\match \calg_{\overline{\gamma}} \sqcup \widetilde{\widetilde{P}} \sqcup \match \calg_{\beta_2}).\label{eq2}
\end{align}

Note that we obtain unique matchings $\widetilde{P}$ and $\widetilde{\widetilde{P}}$ on $\calg^{(k)}\cup \{a\}$ and $(\calg^{(k-1)}\setminus G_1) \cup \{b\}$, respectively. Indeed, when the edge $a$ is in the matching, we must have the top of the  successive limit tile in the matching as well. Hence the top of every even (or odd) tile in $\calg^{(k-1)}$ must be in the matching. This also requires that the right edge of every odd (resp. even) tile is in the matching, compare to Example~\ref{Ex:NoMatching}.  Similarly, if the edge $b$ is in the matching, this  gives rise to the unique matching $\widetilde{\widetilde P}$.

We will also use the following rearrangement of $\cross (\gamma, T)$:

\begin{align}
\cross (\gamma,T)&= \cross(\beta_1,T) \cdot x_{\ast}\cdot \cross (\beta_2,T) \label{eq3}\\
&=\cross(\overline{\gamma},T) \cdot x_1x_2\dots x_{\ast} \cdot \cross (\beta_2,T) \label{eq4}
\end{align}

Applying equation (\ref{eq2}) of perfect matchings of $\ga,$ we obtain:
\begin{align*}
x_{\ga}& \stackrel{(def)}{=} \sum\limits_{P \vdash \calg_{\gamma}}  \displaystyle \frac{x(P)}{\cross (\gamma, T)} \\
&\stackrel{\left(\ref{eq2}\right)}{=} \displaystyle \frac{\left(\sum\limits_{P_1 \vdash \calg_{\beta_1}} x(P_1) \right) x(\widetilde{P}) + \left( \sum\limits_{\overline{P}\vdash \calg_{\overline{\gamma}}} x(\overline{P})\right) x(\widetilde{\widetilde{P}})  \left( \sum\limits_{P_2 \vdash \calg_{\beta_2}} x(P_2)\right) }{\cross (\gamma, T)}
\\ 
\phantom{x_{\ga}} &
\begin{aligned}
&\stackrel{\left(\ref{eq3},\ref{eq4}\right)}{=} 
\left(\sum\limits_{P_1 \vdash \calg_{\beta_1}}  \displaystyle \frac{x(P_1)}{\cross (\beta_1, T)}\right)  \frac{1}{x_{\ast}} \left(\frac{x(\widetilde{P})}{\cross(\beta_2,T)}\right) \\
& \qquad\qquad\qquad\qquad +\left(\sum\limits_{\overline{P} \vdash \calg_{\overline{\gamma}}}  \displaystyle \frac{x(\overline{P})}{\cross (\overline{\gamma}, T)}\right)  \left(\displaystyle \frac{x(\widetilde{\widetilde{P}})}{x_1x_2x_3\cdots}\right) \frac{1}{x_{\ast}}\left(\sum\limits_{P_2 \vdash \calg_{\beta_2}}\displaystyle \frac{x(P_2)}{\cross (\beta_2, T)} \right)
\end{aligned}
\\
&\stackrel{(def)}{=}\left(x_{{\beta_1}}\right)\left(\frac{1}{x_{\ast}}\right) \frac{{x_{0'}}\crossout{x_{1'}}\crossout{x_{2'}}\dots}{\crossout{x_{1'}}\crossout{x_{2'}}\dots} 
+\left( x_{\overline{\gamma}}\right)\frac{\crossout{x_1}\crossout{x_2}\dots}{ \crossout{x_{1}}\crossout{x_{2}}\dots} \left(\frac{1}{x_{\ast}}\right) x_{{\beta_2}}\\
&=\frac{x_{{\beta_1}}x_{0'}+x_{\overline{\gamma}}x_{{\beta_2}}}{x_{\ast}}\\
&\stackrel{\left(\ref{eq1}\right)}{=} x_{\gamma},
\end{align*}
where the last equality follows from (\ref{eq1}) since $x_{\ast}$ is an initial variable.

The proof of the remaining cases where only one of the fans is infinite is  given in the same way. The key consideration is that the snake graph $\ga$ will be exactly as in Fig.~\ref{Fig:MSW-type(I)} except that one of the zig-zag subgraphs $\calg^{(k-1)}$ or $\calg^{(k)}$  will be finite.
\bigskip

\noindent\textbf{Case (II):} Consideration of this case follows similar steps to Case (I) except that the local configuration at $G_{\ast}$ will be different. Now suppose that last change of fan in the domain of $\gamma$ is of Type (II) and  that both fans around $\gamma_{\ast}$ are infinite. Then the snake graph of $\gamma$ is of the form given in Fig.~\ref{Fig:MSW-type(II)}.
Denote the arcs of $\cald_\gamma$ as in the figure,
denote also the snake subgraph of $\calg_\gamma$ corresponding to the arc $\overline{\gamma}$ by $\overline{\calg}$ and the snake subgraph corresponding to the crossing of $\gamma$ with the fan $\calf_{i}$ by $\calg_{i}$.

\begin{figure}[!h]
\begin{center}
\psfrag{g}{{\color{red} $\gamma$}}
\psfrag{bg}{{\color{dgreen} $\bar \gamma$}}
\psfrag{b1}{{\color{dgreen}  $\beta_1$}}
\psfrag{b2}{{\color{dgreen} $\beta_2$}}
\psfrag{0'}{{\color{dgreen}  $0'$}}
\psfrag{1'}{{\color{blue}\scriptsize  $1'$}}
\psfrag{2'}{{\color{blue} \scriptsize $2'$}}
\psfrag{1}{{\color{blue}\scriptsize $1$}}
\psfrag{2}{{\color{blue} \scriptsize   $2$}}
\psfrag{3}{{\color{blue}\scriptsize  $3$}}
\psfrag{*}{{\color{magenta} \large $*$}}
\epsfig{file=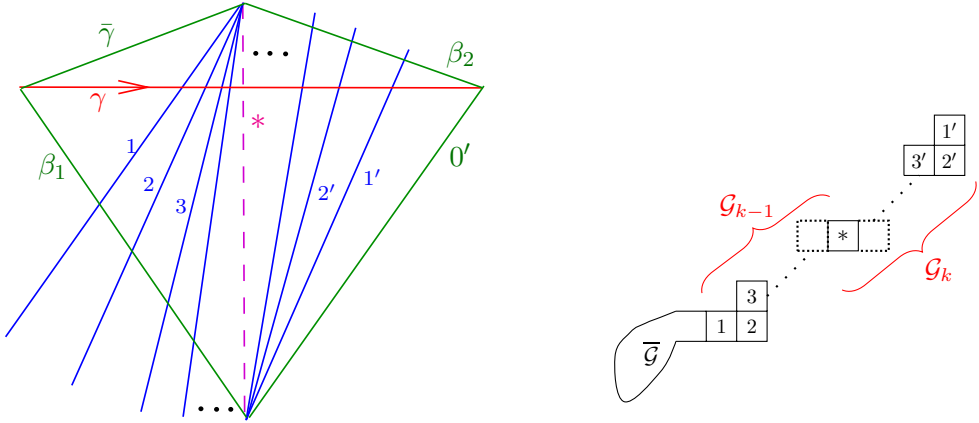,width=0.4\linewidth}
\hspace{.5in}
\raisebox{+6.5ex}{
\begin{tikzpicture}[baseline,scale=.4] 
\draw (0,0)--(2,0)--(2,1)--(0,1)--(0,0) (1,0)--(1,1)
(1,1)--(1,2)--(2,2)--(2,1) (0,0)--(-1,0) (0,1)--(-1,1); 
\draw  plot[smooth, tension=.7] coordinates {(-1,1) (-1.5,.8) (-2,.6) (-2.5,.3) (-3,-.5) (-2.7,-2) (-1.5,-1) (-1,0)};
\node[scale=.8] at (-1.8,-.4){$\overline{\calg}$};
\draw[thick,loosely dotted] (2,1.5)--(3.5,3);
\draw[thick,densely dotted] (3,3)--(4,3)--(4,4)--(3,4)--(3,3);
\draw[thick,densely dotted] (5,4)--(6,4)--(6,3)--(5,3);
\draw (4,3)--(5,3)--(5,4)--(4,4)--(4,3);
\draw[thick,loosely dotted] (5.5,4)--(7,5.5);
\draw (6.5,5.5)--(8.5,5.5)--(8.5,6.5)--(6.5,6.5)--(6.5,5.5) (7.5,5.5)--(7.5,7.5)--(8.5,7.5)--(8.5,6.5); 
\node[scale=.7] at (.5,.5){$1$};
\node[scale=.7] at (1.5,.5){$2$};
\node[scale=.7] at (1.5,1.5){$3$};
\node[scale=.8] at (4.5,3.5){${\ast}$};
\node[scale=.7] at (7,6){$3'$};
\node[scale=.7] at (8,6){$2'$};
\node[scale=.7] at (8,7){$1'$};

\draw[color=red,decorate,decoration={brace,amplitude=8pt},xshift=-4pt,yshift=-0pt]
 (0.0,1.6)--(4.2,4.8) node [color=red,midway,xshift=-0.24cm,yshift=0.5cm] 
{\small $\calg_{k-1}$};

\draw[color=red,decorate,decoration={brace,amplitude=8pt},xshift=-4pt,yshift=-0pt]
 (9.0,5.3)--(4.6,1.7) node [color=red,midway,xshift=0.4cm,yshift=-0.45cm] 
{\small $\calg_{k}$};

\end{tikzpicture}
}
\caption{Snake graph associated with $\gamma$ in Case (II).} 
\label{Fig:MSW-type(II)}
\end{center}
\end{figure}

The snake graphs $\calg_{k-1}$ and $\calg_{k}$ are one-sided infinite zig-zag snake graphs since they are associated with the infinite zig-zag fans $\calf_{k-1}$ and $\calf_{k}$.   Note that since the fans $\calf_{k-1}$ and $\calf_{k}$ have distinct sources, the corresponding glueing in $\ga$ between $\calg_{k-1}$, $G_*$ and $\calg_{k}$ 
in $\calg=\overline{\calg} \cup \calg_{k-1} \cup G_{\ast} \cup \calg_{k}$ is straight.

Using the notation as in Fig.~\ref{Fig:MSW-type(II)} and applying the Ptolemy relation,  we have the identity
\begin{align} x_{\gamma}x_{\ast}=x_{\beta_1}x_{\beta_2}+x_{0'}x_{\overline{\gamma}} \label{eq1'}
\end{align}
in terms of lambda lengths of the arcs $\gamma,\gamma_{\ast},\beta_1,\beta_2,\gamma_{0'},$ and $\overline{\gamma}$. We will show the identity
\begin{align} x_{\ga}x_{\calg_{\ast}}=x_{\calg_{\beta_1}}x_{\calg_{\beta_2}}+x_{\calg_{0'}}x_{\calg_{\overline{\gamma}}} \label{eq5'}.
\end{align}
Observe that we have $x_{\beta_i}=x_{\calg_{\beta_i}},$ for $i=1,2,$ and $x_{\overline{\gamma}}=x_{\calg_{\overline{\gamma}}},$  by the induction hypothesis. 

To show the identity (\ref{eq5'}), we will reorganise the set of perfect matchings of $\ga$ by considering those having the two boundary edges of the tile $G_{\ast}$  in the matching or not:

\begin{align}
\match \ga &=\match \left( \raisebox{-4.5ex}{\scalebox{.8}{
\begin{tikzpicture}[baseline,scale=.4] 
\draw (0,0)--(2,0)--(2,1)--(0,1)--(0,0) (1,0)--(1,1)
(1,1)--(1,2)--(2,2)--(2,1) (0,0)--(-1,0) (0,1)--(-1,1); 
\draw  plot[smooth, tension=.7] coordinates {(-1,1) (-1.5,.8) (-2,.6) (-2.5,.3) (-3,-.5) (-2.7,-2) (-1.5,-1) (-1,0)};
\node[scale=.8] at (-1.8,-.4){$\overline{\calg}$};
\draw[thick,loosely dotted] (2,1.5)--(3.5,3);
\draw[thick,densely dotted] (3,3)--(4,3)--(4,4)--(3,4)--(3,3);
\draw[thick,densely dotted] (5,4)--(6,4)--(6,3)--(5,3);
\filldraw[color=beaublue] (4,3)--(5,3)--(5,4)--(4,4)--(4,3);
\draw[thick,loosely dotted] (5.5,4)--(7,5.5);
\draw (6.5,5.5)--(8.5,5.5)--(8.5,6.5)--(6.5,6.5)--(6.5,5.5) (7.5,5.5)--(7.5,7.5)--(8.5,7.5)--(8.5,6.5); 
\draw[color=red,line width=1.2] (4.3,4.2)--(4.8,3.8) (4.3,3.8)--(4.8,4.2) (4.3,3.2)--(4.8,2.8) (4.3,2.8)--(4.8,3.2);
\node[scale=.7] at (.5,.5){$1$};
\node[scale=.7] at (1.5,.5){$2$};
\node[scale=.7] at (1.5,1.5){$3$};
\node[scale=.8] at (4.5,3.5){${\ast}$};
\node[scale=.7] at (7,6){$3'$};
\node[scale=.7] at (8,6){$2'$};
\node[scale=.7] at (8,7){$1'$};
\end{tikzpicture}}
}
\right) \cup
\match \left( \raisebox{-4.5ex}{\scalebox{.8}{
\begin{tikzpicture}[baseline,scale=.4] 
\draw (0,0)--(2,0)--(2,1)--(0,1)--(0,0) (1,0)--(1,1)
(1,1)--(1,2)--(2,2)--(2,1) (0,0)--(-1,0) (0,1)--(-1,1); 
\draw  plot[smooth, tension=.7] coordinates {(-1,1) (-1.5,.8) (-2,.6) (-2.5,.3) (-3,-.5) (-2.7,-2) (-1.5,-1) (-1,0)};
\node[scale=.8] at (-1.8,-.4){$\overline{\calg}$};
\draw[thick,loosely dotted] (2,1.5)--(3.5,3);
\draw[thick,densely dotted] (3,3)--(4,3)--(4,4)--(3,4)--(3,3);
\draw[thick,densely dotted] (5,4)--(6,4)--(6,3)--(5,3);
\filldraw[color=beaublue] (4,3)--(5,3)--(5,4)--(4,4)--(4,3);
\draw[color=red, line width=2.4] (4,4)--(5,4) (4,3)--(5,3);
\draw[thick,loosely dotted] (5.5,4)--(7,5.5);
\draw (6.5,5.5)--(8.5,5.5)--(8.5,6.5)--(6.5,6.5)--(6.5,5.5) (7.5,5.5)--(7.5,7.5)--(8.5,7.5)--(8.5,6.5); 
\node[scale=.7] at (.5,.5){$1$};
\node[scale=.7] at (1.5,.5){$2$};
\node[scale=.7] at (1.5,1.5){$3$};
\node[scale=.8] at (4.5,3.5){${\ast}$};
\node[scale=.7] at (7,6){$3'$};
\node[scale=.7] at (8,6){$2'$};
\node[scale=.7] at (8,7){$1'$};
\end{tikzpicture}}
}
\right) \nonumber
\\
&=\match \left( \raisebox{-4.5ex}{\scalebox{.78}{
\begin{tikzpicture}[baseline,scale=.4] 
\draw (0,0)--(2,0)--(2,1)--(0,1)--(0,0) (1,0)--(1,1)
(1,1)--(1,2)--(2,2)--(2,1) (0,0)--(-1,0) (0,1)--(-1,1); 
\draw  plot[smooth, tension=.7] coordinates {(-1,1) (-1.5,.8) (-2,.6) (-2.5,.3) (-3,-.5) (-2.7,-2) (-1.5,-1) (-1,0)};
\node[scale=.8] at (-1.8,-.4){$\overline{\calg}$};
\draw[thick,loosely dotted] (2,1.5)--(3.5,3);
\draw[thick,densely dotted] (3,3)--(4,3)--(4,4)--(3,4)--(3,3);
\draw[thick,densely dotted] (5,4)--(6,4)--(6,3)--(5,3);
\filldraw[color=beaublue] (4,3)--(5,3)--(5,4)--(4,4)--(4,3);
\draw[thick,loosely dotted] (5.5,4)--(7,5.5);
\draw (6.5,5.5)--(8.5,5.5)--(8.5,6.5)--(6.5,6.5)--(6.5,5.5) (7.5,5.5)--(7.5,7.5)--(8.5,7.5)--(8.5,6.5); 
\draw[color=red,line width=1.2] (4.3,4.2)--(4.8,3.8) (4.3,3.8)--(4.8,4.2) (4.3,3.2)--(4.8,2.8) (4.3,2.8)--(4.8,3.2);
\node[scale=.7] at (.5,.5){$1$};
\node[scale=.7] at (1.5,.5){$2$};
\node[scale=.7] at (1.5,1.5){$3$};
\node[scale=.8] at (4.5,3.5){${\ast}$};
\node[scale=.7] at (7,6){$3'$};
\node[scale=.7] at (8,6){$2'$};
\node[scale=.7] at (8,7){$1'$};
\draw[color=red,decorate,decoration={brace,amplitude=18pt},xshift=-4pt,yshift=0pt]
 (4,2.4)--(-2.4,-2.4) node [color=red,midway,xshift=.7cm,yshift=-.7cm] 
{\footnotesize $\calg_{\beta_2}$};
\draw[color=red,decorate,decoration={brace,amplitude=18pt},xshift=-4pt,yshift=0pt]
 (9.7,7)-- (5.2,2.1) node [color=red,midway,xshift=.7cm,yshift=-.7cm] 
{\footnotesize $\calg_{\beta_2}$};
\end{tikzpicture}}
}
\right) \cup
\match \left( \raisebox{-4.5ex}{\scalebox{.78}{
\begin{tikzpicture}[baseline,scale=.4] 
\draw (0,0)--(2,0)--(2,1)--(0,1)--(0,0) (1,0)--(1,1)
(1,1)--(1,2)--(2,2)--(2,1) (0,0)--(-1,0) (0,1)--(-1,1); 
\draw  plot[smooth, tension=.7] coordinates {(-1,1) (-1.5,.8) (-2,.6) (-2.5,.3) (-3,-.5) (-2.7,-2) (-1.5,-1) (-1,0)};
\node[scale=.8] at (-1.8,-.4){$\overline{\calg}$};
\draw[thick,loosely dotted] (2,1.5)--(3.5,3);
\draw[thick,densely dotted] (3,3)--(4,3)--(4,4)--(3,4)--(3,3);
\draw[thick,densely dotted] (5,4)--(6,4)--(6,3)--(5,3);
\filldraw[color=beaublue] (4,3)--(5,3)--(5,4)--(4,4)--(4,3);
\draw[color=red, line width=2.4] (4,4)--(5,4) (4,3)--(5,3);
\draw[color=red, line width=2.4] (3,3)--(3,4) (1,1)--(1,2) (1,0)--(2,0);
\node[color=red,scale=.6] at (.6,1.5){$2$};
\node[color=red,scale=.6] at (1.5,-.4){$1$};
\draw[color=red, line width=2.4] (6,3)--(6,4) (6.5,6.5)--(7.5,6.5) (8.5,5.5)--(8.5,6.5) (7.5,7.5)--(8.5,7.5);
\node[color=red,scale=.6] at (7,6.9){$2'$};
\node[color=red,scale=.6] at (8.9,5.9){$1'$};
\node[color=red,scale=.6] at (8,7.9){$0'$};
\draw[thick,loosely dotted] (5.5,4)--(7,5.5);
\draw (6.5,5.5)--(8.5,5.5)--(8.5,6.5)--(6.5,6.5)--(6.5,5.5) (7.5,5.5)--(7.5,7.5)--(8.5,7.5)--(8.5,6.5); 
\node[scale=.7] at (.5,.5){$1$};
\node[scale=.7] at (1.5,.5){$2$};
\node[scale=.7] at (1.5,1.5){$3$};
\node[scale=.8] at (4.5,3.5){${\ast}$};
\node[scale=.7] at (7,6){$3'$};
\node[scale=.7] at (8,6){$2'$};
\node[scale=.7] at (8,7){$1'$};
\draw[color=red,decorate,decoration={brace,amplitude=8pt},xshift=-4pt,yshift=0pt]
 (0,-.4)--(-2.4,-2.4) node [color=red,midway,xshift=.4cm,yshift=-.4cm] 
{\footnotesize $\calg_{\overline{\gamma}}$};
\draw[color=red,decorate,decoration={brace,amplitude=30pt},xshift=-4pt,yshift=0pt]
 (9.9,7)--(1,-1.1) node [color=red,midway,xshift=.9cm,yshift=-.9cm] 
{\footnotesize $\widetilde{P}$};
\end{tikzpicture}}
}
\right) \nonumber
\\
&= (\match \calg_{\beta_1} \sqcup \calg_{\beta_2}) \cup (\match \calg_{\overline{\gamma}} \sqcup \widetilde{P}).
\label{eq2'}
\end{align}

Observe that we obtain the unique matching $\widetilde{P}$ of $\ga\setminus\overline{\calg}$ as in Case (I).

Applying equation (\ref{eq2'}) of perfect matchings of $\ga,$ we get:

\begin{align*}
x_{\ga}& \stackrel{(def)}{=} \sum\limits_{P \vdash \calg_{\gamma}}  \displaystyle \frac{x(P)}{\cross (\gamma, T)} \\
&\stackrel{\left(\ref{eq2'}\right)}{=} \displaystyle \frac{\left(\sum\limits_{P_1 \vdash \calg_{\beta_1}} x(P_1) \right) 
\left( \sum\limits_{P_2 \vdash \calg_{\beta_2}} x(P_2) \right)+ \left( \sum\limits_{\overline{P}\vdash \calg_{\overline{\gamma}}} x(\overline{P})\right) x({\widetilde{P}})}{\cross (\gamma, T)}
\\ 
&\begin{aligned}
&\stackrel{\left(\ref{eq3},\ref{eq4}\right)}{=} \left(\sum\limits_{P_1 \vdash \calg_{\beta_1}}  \displaystyle \frac{x(P_1)}{\cross (\beta_1, T)}\right)  \frac{1}{x_{\ast}} \left(\sum\limits_{P_2 \vdash \calg_{\beta_2}}\frac{x(P_2)}{\cross(\beta_2,T)}\right)\\
&\qquad\qquad\qquad\qquad\qquad\qquad\qquad + \left(\sum\limits_{\overline{P} \vdash \calg_{\overline{\gamma}}}  \displaystyle \frac{x(\overline{P})}{\cross (\overline{\gamma}, T)}\right)  \left(\displaystyle \frac{x(\widetilde{P})}{x_1x_2x_3\dots x_{\ast} \dots x_{2'}x_{1'}}\right) 
\end{aligned}
\\
&\stackrel{(def)}{=}\left(x_{{\beta_1}}\right)\left(\frac{1}{x_{\ast}}\right) \left( x_{\beta_2} \right)
+\left( x_{\overline{\gamma}}\right)\frac{\crossout{x_1}\crossout{x_2}\dots \crossout{x_{2'}} \crossout{x_{1'}}x_{0'}}{ \crossout{x_{1}}\crossout{x_{2}}\dots x_{\ast}\dots \crossout{x_{2'}} \crossout{x_{1'}}}\\
&=\frac{x_{{\beta_1}}x_{\beta_2}+x_{\overline{\gamma}}x_{0'}}{x_{\ast}}\\
&\stackrel{\left(\ref{eq1'}\right)}{=} x_{\gamma},
\end{align*}
where the last equality follows from (\ref{eq1'}) as $x_{\ast}$ is an initial variable.

The remaining cases are argued in a similar way since the local configuration around  $G_{\ast}$ adjacent to the (finite or infinite) zig-zag snake subgraphs  $\calg^{(k-1)}$ and $\calg^{(k)}$  will be a straight piece inside the staircase snake graph $\ga$.
\end{proof}
\subsection{Skein relations} 
In this section we will consider fan triangulations of \emph{unpunctured} infinite surfaces. The goal is to generalise \emph{skein relations} to the infinite case under this setting. To establish these relations, we will generalise the \emph{snake graph calculus} of~\cite{CS} to the infinite case.

For a finite unpunctured surface $\cals$, given two arcs  $\gamma_1, \gamma_2\subset\cals$ and a crossing of them, one can resolve the crossing by locally replacing the crossing {\Large $\times$} with the pair of segments  $\genfrac{}{}{0pt}{5pt}{\displaystyle\smile}{\displaystyle\frown}$  or  $\supset \subset$. 
This gives rise to four new (generalised) arcs  $\gamma_3, \gamma_4$ and $\gamma_5, \gamma_6$  (see Definition~\ref{Def:GenArc})  corresponding to the two different ways of smoothing the crossing. 
As shown initially in \cite{MS}, the corresponding elements $x_{\gamma_1}, x_{\gamma_2}, \dots, x_{\gamma_6}$ in the cluster algebra satisfy the 
\emph{skein relations} given by  
$$x_{\gamma_1} x_{\gamma_2} =  x_{\gamma_3} x_{\gamma_4} + x_{\gamma_5} x_{\gamma_6}.$$

Again in the finite case, another proof of skein relations is given combinatorially in \cite{CS} in terms of snake graphs associated with a pair of crossing arcs and those obtained by smoothing a crossing. The main idea of the proof is given by
\begin{itemize}
\item detecting the region of a crossing in the surface in terms of the associated snake graphs (this notion is called \emph{crossing overlap} in \cite{CS}):
\begin{itemize}
\item crossing in a (finite or infinite) surface occurs exactly in one of the $3$ ways given in Fig.~\ref{Fig:Crossing}. In the infinite surface case, the same configuration remains fixed and the only difference is that the region where the crossing occurs might be a triangulated infinite polygon in the universal cover of the surface. However, the crossing condition only depends on the local configuration on the boundary of this region, and therefore, the whole argument in terms of snake graphs immediately generalises when infinite surfaces are considered;
\end{itemize}
\begin{figure}
\epsfig{file=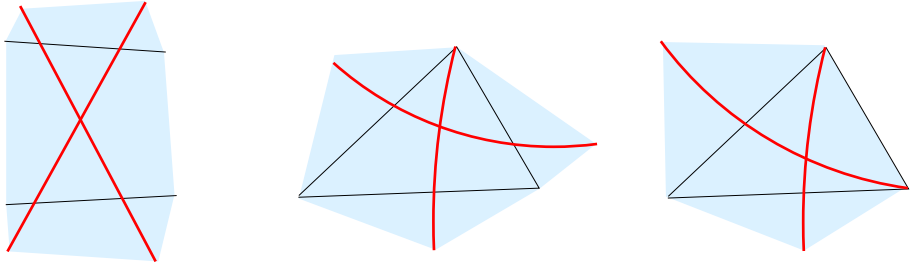,width=0.75\linewidth} 
\caption{Three different types of crossings in a surface.}
\label{Fig:Crossing}
\end{figure}
\item giving a bijection between perfect matchings of $\calg_1\sqcup\calg_2$ and $(\calg_3\sqcup\calg_4)\cup(\calg_5\sqcup\calg_6)$ where $\calg_i$ for $i=1,\dots,6$ are the snake graphs associated with the arcs obtained by smoothing a crossing: this notion is called \emph{resolution} or \emph{grafting} in \cite{CS}:
\begin{itemize}
\item[-] the essential idea of the proof of the bijection between perfect matchings of a pair of snake graphs is to find a \emph{switching position} in an overlap under consideration. When infinite staircase snake graphs are considered, one might need to determine the switching position in an infinite snake graph.  Although it may not be obvious at first sight, this could be done in a similar manner to that of the finite version.  This is due to the fact that we only consider infinite staircase snake graphs (Proposition~\ref{Prop:InfSG}) and perfect matchings of infinite zig-zag snake graphs are characterised explicitly (Lemma~\ref{Lem:UniquePMexcept_1tile}).
\end{itemize} 
\end{itemize}

\begin{definition}[Generalised arc]\label{Def:GenArc}
By a {\it generalised arc} we mean any curve on $\cals$ with endpoints in the marked points of $\cals$ that is not contractible to a boundary segment and not cutting out an unpunctured monogon. As usual, the curves are considered up to isotopy. 
\end{definition}

\begin{definition} [Laurent series associated with generalised arcs] Let $\cals$ be an infinite surface and $T$ be a fan triangulation of $\cals.$ Let $\gamma$ be a generalised arc in $\cals$ and $\ga$ be the infinite snake graph associated with the lift of $\gamma$ in the universal cover. The \emph{Laurent series} $x_{\gamma}$ \emph{associated with} $\gamma$ is defined to be equal to the Laurent series associated with $\ga.$
\end{definition}

\begin{remark}
The assumption that a generalised arc does not cut out an unpunctured monogon can be removed from the definition, and in this case one could use smoothings of self-crossings to associate Laurent series as in~\cite{MSW2}.  
\end{remark}

If $\gamma$ is a generalised arc, then $x_{\gamma}$ is not a cluster variable; however, applying induction on the number of fans crossed by the lift of $\gamma$ as in Theorem~\ref{Thm:MSW-main} one can show that $x_{\gamma}$ is an element of the associated cluster algebra.

\begin{proposition}
The Laurent series associated with a generalised arc is an element of the cluster algebra $\mathcal A(\cals)$.
\end{proposition}

By following the same arguments as in~\cite{CS}, we obtain skein relations for cluster algebras associated with unpunctured infinite surfaces.

\begin{theorem}\label{Thm:Skein}
Let $\cals$ be an infinite unpunctured surface and $\gamma_1,\gamma_2$ be crossing generalised arcs. Let $\gamma_3,\dots,\gamma_6$ be the (generalised) arcs obtained by smoothing a crossing of $\gamma_1$ and $\gamma_2$ and $x_1,\dots,x_6$ be the corresponding elements  in the associated infinite surface cluster algebra $\cala(\cals)$. Then the identity
\begin{align*}
x_1x_2=x_3x_4+x_5x_6
\end{align*} 
holds in $\cala(\cals)$.
\end{theorem}
\section{Properties of infinite rank cluster algebras} \label{Sec:Consq}

In this section we generalise some properties of (surface) cluster algebras to the setting of cluster algebras from infinite surfaces.

\begin{theorem}\label{Thm:Consq}
Let $\cals$ be an infinite surface and $\cala(\cals)$ be the corresponding cluster algebra. Then
\begin{itemize}
\item seeds are uniquely determined by their clusters; 
\item for any two seeds containing a particular cluster variable $x_\gamma$ there exists a mutation sequence $\mu^\bullet\in \{\mu^{(0)}, \mu^{(n)}, \mu^{(\infty)} \}$ such  that $x_\gamma$ belongs to every cluster obtained in the course of mutation $\mu^\bullet$;
\item there is a cluster containing  cluster variables $\{x_i \mid i\in I\}$, where $I$ is a finite or infinite index set, if and only if for every choice of $i,j\in I$ there exists a cluster containing $x_i$ and $x_j$.
\end{itemize}
Moreover, if $T$ is a fan triangulation of $\cals$ then
\begin{itemize}
\item the ``Laurent phenomenon'' holds, i.e. any cluster variable in $\cala(\cals,T)$ is a Laurent series in cluster variables corresponding to the arcs (and boundary arcs) of $T$;
\item ``Positivity'' holds, i.e. the coefficients in the Laurent series expansion of a cluster variable  in $\cala(\cals,T)$ are positive. 
\end{itemize}
If in addition the surface ${\cals}$ is unpunctured then 
\begin{itemize}
\item the exponents in the denominator of a cluster variable $x_{\gamma}$ corresponding to an arc $\gamma$ are equal to the intersection numbers  $|\gamma\cap \tau_i|$ of $\gamma$ with arcs $\tau_i$ of $T$.
\end{itemize}
\end{theorem}

\begin{proof}
First three properties follow from the bijection between cluster variables and arcs in $\cals$ established in Theorem~\ref{variables are distinct}.

The other three properties are  immediate consequences of Theorem~\ref{Thm:MSW-main} since the expansion formula is a Laurent series by definition and the coefficients of the terms are parametrised by perfect matchings of snake graphs.
The coefficients are positive integers in view of Corollary~\ref{positivity}. In addition, in the case of unpunctured surface, the denominator of the expansion formula is given by the crossing pattern of the arc. 
\end{proof}

\begin{remark}
\label{rem: outgoing fan}
For  clusters associated with outgoing fan triangulations, the Laurent phenomenon  takes the classical  form: in this case the cluster variables are Laurent polynomials as every arc on $\cals$ crosses only finitely many arcs of a given outgoing fan triangulation.
\end{remark}

\begin{remark}
All results (and the proofs) of Sections~\ref{Sec:Combinatorics} and~\ref{Sec:ClusAlg}  hold also in presence of finite number of orbifold points (in the sense of~\cite{FeSTu3}). More precisely, we can have two types of orbifold points: the cone points with angle $\pi$ (orbifold points of weight $1/2$) and the punctures with self-conjugate horocycles (orbifold points of weight $2$).

\end{remark}

\begin{remark} In principle, one could also introduce laminations (hence principal coefficients) using \emph{outgoing} fan triangulations. Then, one could give skein relations for infinite rank surface cluster algebras with principal coefficients. Moreover, one could associate infinite adjacency quivers of infinite triangulations of surfaces, discuss quiver mutations for cluster variables and also generalise \cite{MSW} expansion formula for snake graphs to give Laurent series expansions for cluster variables with principle coefficients. 
\end{remark}



\begin{thebibliography}{99}
%
%
\bibitem{ABCP} I.~Assem, T.~Br\"ustle, G.~Charbonneau-Jodoin and P.G.~Plamondon, {\em Gentle algebras arising from surface triangulations}, Algebra Number Theory {\bf 4}, (2010), no. 2, 201--229.
%
\bibitem{BG} K.~Baur and S.~Gratz, {\em Transfinite mutations in the completed infinity-gon},  
J. Combin. Theory, Ser. A, {\bf 155} (2018),  321--359.

\bibitem{BFZ} A.~Berenstein, S.~Fomin and A.~Zelevinsky, {\it Cluster algebras III: Upper bounds and double Bruhat cells}. Duke Math. J. \textbf{126}, (2005), {\bf 1}, 1--52. 

\bibitem{BHJ} C.~Bessendrodt, T.~Holm and P.~J\o rgensen, {\em All $Sl_2$-tilings come from infinite triangulations},  
Adv. Math. {\bf 315} (2017), 194--245. 

\bibitem{BZ} T.~Br\"ustle and J.~Zhang, {\em On the cluster category of a marked surface},  Algebra and Number Theory, {\bf 5} (2011), {\bf 4}, 529--566.

%
%

\bibitem{CLS} \I.~\Canakci, K.~Lee and R.~Schiffler, {\em On cluster algebras for surfaces without punctures and one marked point},  Proc. Amer. Math. Soc., Ser. B, {\bf 2} (2015), 35--49.

\bibitem{CS} \I.~\Canakci \ and R.~Schiffler, {\em Snake graph calculus and cluster algebras from surfaces},  J. Algebra, \textbf{382}, (2013) 240--281.


\bibitem{CS2} \I.~\Canakci \ and R.~Schiffler, {\em Snake graph calculus and cluster algebras from surfaces II: Self-crossing snake graphs},  Math. Z., {\bf 281} (1), (2015), 55--102. 

\bibitem{CS3} \I.~\Canakci \ and R.~Schiffler, {\em Snake graph calculus and cluster algebras from surfaces III: Band graphs and snake rings}, 
Int. Math. Res. Not.  (2017) 1--82.
 

\bibitem{CaSc} \I.~\Canakci \ and S.~Schroll, with an appendix by C.~Amiot, {\em Extensions in Jacobian algebras and cluster categories of marked surfaces},   Adv.  Math. {\bf 313} (2017), 1--49. 

%
%
%
%
%
\bibitem{FeTu} A.~Felikson and P.~Tumarkin, {\em Bases for cluster algebras from orbifolds}, 
Adv. Math. {\bf 318} (2017), 191--232.

\bibitem{FeSTu} A.~Felikson, M.~Shapiro and P.~Tumarkin, {\em Skew-symmetric cluster algebras of finite mutation type},  J. Eur. Math. Soc. \textbf{14}, (2012), 1135--1180.

\bibitem{FeSTu2} A.~Felikson, M.~Shapiro and P.~Tumarkin, {\em Cluster algebras of finite mutation type via unfoldings},  Int. Math. Res. Not.  {\bf 8}, (2012), 1768--1804.

\bibitem{FeSTu3} A.~Felikson, M.~Shapiro and P.~Tumarkin, {\em Cluster algebras and triangulated orbifolds},  Adv. Math. \textbf{231}, 5, (2012),  2953--3002.

\bibitem{FG1} V.~Fock and A.~Goncharov, {\em Moduli spaces of local systems and higher Teichm\"uller theory}   Publ. Math. Inst. Hautes \'Etudes Sci.  {\bf 103}  (2006), 1--211. 

\bibitem{FG2} V.~Fock,  and A.~Goncharov,  {\em Dual Teichm\"uller and lamination spaces},
Handbook of Teichm\"uller Theory, Vol. I,   647--684, IRMA Lect. Math. Theor. Phys., 11, Eur. Math. Soc., Z\"urich, 2007.



\bibitem{FST} S.~Fomin, M.~Shapiro and D.~Thurston, {\em Cluster algebras and triangulated surfaces. Part I: Cluster complexes}, Acta Math. \textbf{201} (2008), 83--146. 


\bibitem{FT} S.~Fomin and D.~Thurston, {\em Cluster algebras and triangulated surfaces. Part II: Lambda Lengths}, 
Mem. Amer. Math. Soc., {\bf 255} (2018), no. 1223. 
%
\bibitem{FZ1} S.~Fomin and A.~Zelevinsky, {\em Cluster algebras I: Foundations}, J. Amer. Math. Soc. \textbf{15} \rm (2002), 497--529.

 \bibitem{FZ2} S.~Fomin and A.~Zelevinsky, {\em Cluster algebras II: Finite type classification}, Invent. Math. \textbf{154} (2003), 1, 63--121.

\bibitem{FZ4} S.~Fomin and A.~Zelevinsky, {\em Cluster algebras IV: Coefficients}, Compos. Math. \textbf{143} \rm (2007), 112--164.

 
\bibitem{FK} P.~Di~Francesco and R.~Kedem, {\em $Q$-systems as cluster algebras II}, Lett. Math. Phys. \textbf{89} no 3 (2009), 183--216.

%
%



\bibitem{GG} J.~Grabowski and S.~Gratz, with  appendix by Michael Groechenig, {\it Cluster algebras of infinite rank}, J.  Lond.  Math.  Soc.  (2), \textbf{89}  (2014), no. 2, 337--363.

\bibitem{GG1} J.~Grabowski and S.~Gratz,
{\it Graded quantum cluster algebras of infinite rank as colimits},
J.  Pure  Appl. Alg., {\bf 222} (2018), no. 11, 3395--3413. 

\bibitem{G} S.~Gratz,
{\it Cluster algebras of infinite rank as colimits},
Math. Z., {\bf 281}  (2015) no.3-4, 1137--1169. 

%

\bibitem{H} A.~Hatcher, {\em On triangulations of surfaces}, Topology Appl. {\bf 40 } (1991), 189--194.

\bibitem{HL} D.~Hernandez and B.~Leclerc, {\em A cluster algebra approach to $q$-characters of Kirillov-Reshetikhin modules},    J. Eur. Math. Soc.  {\bf 18} (2016), no. 5, 1113--1159. 

\bibitem{HL2} D.~Hernandez and B.~Leclerc, {\em Cluster algebras and category $O$ for representations of Borel subalgebras of quantum affine algebras},   Algebra Number Theory, {\bf 10} (2016), no. 9, 2015--2052.  

\bibitem{HJ} T.~Holm and P.~J\o rgensen, {\em On a cluster category of infinite Dynkin type, and the relation to triangulations of the infinity-gon},  Math. Z. {\bf 270} (2012), no. 1-2, 277--295. 

\bibitem{HJ2} T.~Holm and P.~J\o rgensen, {\em $SL_2$-tilings and triangulations of the strip},  J. Combin. Theory, Ser. A, {\bf 120} (2013), no. 7, 1817--1834. 

%

\bibitem{IT} K. Igusa and G. Todorov, {\em Continuous cluster categories I},  Algebr. Represent. Theory, {\bf 18} (2015), no. 1, 65--101. 

\bibitem{IT2} K. Igusa and G. Todorov, {\em Continuous Frobenius categories}, (English summary) Algebras, quivers and representations, 115--143, Abel Symp., 8, Springer, Heidelberg, 2013. 

%
%
\bibitem{Labardini} D.~Labardini-Fragoso, {Quivers with potentials associated to triangulated surfaces, part IV: removing boundary assumptions}, Selecta Math. (N.S.) 22 (2016), no. 1, 145--189. 


\bibitem{Ladkani} S.~Ladkani, {\em On Jacobian algebras from closed surfaces}, {\tt arXiv:1207.3778}.

\bibitem{LS4} K.~Lee and R.~Schiffler, {\em Positivity for cluster algebras},   Annals of Math. \textbf{182} (1), (2015) 73--125.



\bibitem{LP} S.~Liu, C.~Paquette, {\em Cluster categories of type $A^{\infty}_{\infty}$ and triangulations of the infinite strip},  {\tt arXiv:1505.06062}.






\bibitem{M} G.~Muller, {\em Skein algebras and cluster algebras of marked surfaces},   Quantum Topol. {\bf 7} (2016), no. 3, 435--503. 

\bibitem{MS} G.~Musiker and R.~Schiffler, {\em Cluster expansion formulas and perfect matchings}, \emph{J. Algebraic Combin.} \textbf{32} (2010), no. 2, 187--209.

\bibitem{MSW} G.~Musiker, R.~Schiffler and L.~Williams, {\em Positivity for cluster algebras from surfaces},  
Adv. Math.  \textbf{227}, (2011), 2241--2308.

\bibitem{MSW2} G.~Musiker, R.~Schiffler and L.~Williams, {\em Bases for cluster algebras from surfaces},  
 Compos. Math. \textbf{149}, 2, (2013), 217--263.

\bibitem{MW} G.~Musiker and L.~Williams, {\em Matrix formulae and skein relations for cluster algebras from surfaces},  Int. Math. Res. Not. \textbf{13}, (2013),  2891--2944.


%



\bibitem{P} R.~C.~Penner, {\em Decorated Teichm\"uller theory}, with a foreword by Yuri I. Manin.  QGM Master Class Series. European Mathematical Society (EMS), Z\"urich, 2012. 

 



%
\bibitem{QZ} Y. Qui and Y. Zhou, {\em Cluster categories for marked surfaces: punctured case}, 
 Compos. Math. {\bf 153}, 9 (2017), 1779--1819.

\bibitem{S2} R. Schiffler, {\em A cluster expansion formula ($A_n$ case)},  Electron. J. Combin. {\textbf 15} (2008), no. 1, Research paper 64, 9 pp. 

\bibitem{S3} R. Schiffler, {\em On cluster algebras arising from unpunctured surfaces II}, Adv. Math. \textbf{223}, (2010), 1885--1923.  

\bibitem{ST} R.~Schiffler and H.~Thomas, {\em On cluster algebras arising from unpunctured surfaces}, Int. Math. Res. Not.  {\textbf 17} (2009), 3160--3189.

%

\bibitem{T} D.~Thurston, {\em A positive basis for surface skein algebras},  Proc. Natl. Acad. Sci. USA \textbf{111}, 27, (2014), 9725--9732.

%
%
%
\bibitem{ZZZ} J.~Zhang, Y.~Zhou and B.~Zhu, {\em Cotorsion pairs in the cluster category of a marked surface}, 
J. Algebra {\bf 391} (2013), 209--226. 

\end{thebibliography}
\end{document}